\documentclass[a4paper,twoside,11pt,reqno]{amsart}
\usepackage[margin=1in]{geometry}
\usepackage[english]{babel}
\usepackage[utf8]{inputenc}
\usepackage{amsmath}
\usepackage{amssymb}
\usepackage{amsfonts}
\usepackage{amsthm}
\usepackage{bm}
\usepackage{mathrsfs}
\usepackage{hyperref}
\usepackage{comment}
\hypersetup{
	colorlinks=true,
	linkcolor=blue,
	filecolor=blue,
	citecolor=purple,      
	urlcolor=cyan,
}
\usepackage{xcolor}
\usepackage{dsfont}
\usepackage{physics}
\usepackage{stmaryrd}
\usepackage{enumitem}
\DeclareMathAlphabet{\mathpzc}{OT1}{pzc}{m}{it}
\usepackage{tikz-cd}
\usepackage{tikz}
\usepackage{tikzsymbols}
\usepackage{mathtools}
\usetikzlibrary{matrix}

\newcommand{\motpi}{\pi^{\mathbb{A}^1}}
\newcommand{\A}{\mathbb{A}}
\newcommand{\G}{\mathbf{G}}
\newcommand{\C}{\mathbb{C}}
\newcommand{\Z}{\mathbb{Z}}

\theoremstyle{plain}
\newtheorem{theorem}[subsubsection]{Theorem}

\newtheorem{lemma}[subsubsection]{Lemma}
\newtheorem{proposition}[subsubsection]{Proposition}

\newtheorem{corollary}[subsubsection]{Corollary}
\theoremstyle{definition}
\newtheorem{definition}[subsubsection]{Definition}
\newtheorem{definition-proposition}[subsubsection]{Definition/Proposition}
\theoremstyle{remark}
\newtheorem{remark}[subsubsection]{Remark}

\numberwithin{equation}{section}

\title{Adjacent non-stable layers in the $\A^1$-homotopy of the special linear tower}
\author{Haoyang Liu and Fan Yang}
\date{}
\address{University of California, Santa Barbara}
\email{haoyangliu@ucsb.edu}
\address{University of Southern California}
\email{fyang399@usc.edu, fanyang399@gmail.com}

\begin{document}

\begin{abstract}
We determine the two cross-stage morphisms arising from adjacent Stiefel fiber
sequences in the special linear tower.  Over characteristic-zero fields, the
first is the Wood morphism at even stages and vanishes at odd stages; the
second is induced stably by the motivic Hopf element at even stages and is zero
at odd stages. Complex
realization consequently classifies oriented rank-\((n-2)\) vector bundles on
the split quadric \(Q_{2n-1}\) for \(n\geq5\). We then apply these calculations to give exact corank-two splitting and efficient generation criteria for projective modules.
\end{abstract}

\maketitle
\tableofcontents

\section{Introduction}

The non-stable $\A^1$-homotopy sheaves of the classical algebraic groups
connect algebraic $K$-theory, quadratic-form invariants, and the non-stable
topology of the corresponding Lie groups; see
\cite{MV99,MorelA1} for the foundations of $\A^1$-homotopy theory.
In the stable range one has
\[
\pi_i^{\A^1}(SL_n)\cong \mathbf K_{i+1}^Q
\]
\cite[Lemma~3.1 and Theorem~3.2]{AF14}, whereas immediately beyond this
range additional Milnor, Milnor--Witt, and Grothendieck--Witt terms
appear.  These non-stable terms also have geometric significance: after
contraction, they occur as coefficient sheaves in obstruction problems
for algebraic vector bundles and projective modules
\cite{AF15b,AOSS26}.

The first non-stable layer of the linear groups was computed by
Asok--Fasel \cite{AF14}.  Their calculation describes the first failure of
stabilization, gives the classification of critical-rank oriented bundles
on split odd quadrics, and explains how the factorial and Witt-theoretic
contributions are detected by complex and real realization.  The aim of
this paper is to continue this analysis through the next adjacent stages
of the special linear tower.

\subsection*{Main results}

The results fall into four groups.  The first identifies the cross-stage
map in the second non-stable layer.  We use throughout the Euler/Gysin
orientation of \cite[Lemma~3.5]{AF14}; with this convention, the signs below
are part of the statement.

\medskip\noindent
\textbf{Theorem A.}
Let $k$ be a field of characteristic zero.  For $n\geq4$, the cross-stage
morphism $d_{n,1}$ vanishes when $n$ is odd.  When $n$ is even, its
Grothendieck--Witt component is the Wood morphism given by multiplication by
$\eta$.  The infinite-loop adjoints of the unit
$\mathds1\to\mathbf{KO}$ give a normalization for which the comparison is
strict, and the comparison with the printed Suslin matrices is a uniquely
determined four-periodic unit.

This combines Theorem~\ref{thm:main-diagrams},
Proposition~\ref{prop:oriented-unit-wood}, and
Theorem~\ref{prop:explicit-suslin-rank}.  The second result treats the
next cross-stage map and its deepest nonzero contractions.

\medskip\noindent
\textbf{Theorem B.}
Let $k$ be a field of characteristic zero and $n\geq5$.  After identifying
the relevant punctured affine spaces with motivic spheres, the morphism
$d_{n,2}$ is induced by the action of $\eta$ when $n$ is even and is zero
when $n$ is odd.  Its residual cyclotomic boundary
$\boldsymbol\mu_{24}\to\mathbf K_2^M/2$ vanishes over every such field, and
the next two contracted cross-stage maps are zero.

The precise statements are Proposition~\ref{prop:dn2-stable-eta} and
Theorems~\ref{thm:dn2-contracted-zero}
and~\ref{thm:epsilon-global-zero}.  The next result determines the realization
of the two adjacent non-stable sheaves and converts it into a classification
on split quadrics.

\medskip\noindent
\textbf{Theorem C.}
For $m\geq2$, complex realization induces an isomorphism
\[
 \pi_{2m,2m+2}^{\A^1}(SL_{2m})(\C)
 \cong \mathbb Z/2\oplus\mathbb Z/(2m+1)!,
\]
and, for $m\geq1$, it induces an isomorphism
\[
 \pi_{2m+1,2m+3}^{\A^1}(SL_{2m+1})(\C)
 \cong \mathbb Z/((2m+2)!/2).
\]
Consequently, for every \(n\geq5\), complex realization classifies the
oriented vector bundles of rank \(n-2\) on the split odd quadric
\(Q_{2n-1}\).

See Theorems~\ref{thm:realization-SLn}
and~\ref{thm:comparison-quadric} and
Corollary~\ref{cor:rank-n-2-explicit}.
Finally, the secondary coefficient calculation and its parameterized
stable comparison have geometric consequences.

\medskip\noindent
\textbf{Theorem D.}
Let $V\to S$ be a vector bundle of rank $t\geq2$ on a smooth affine
scheme of pure dimension $s$ over a characteristic-zero field, put
$X=\operatorname{Tot}(V)$, and assume $d=s+t\geq4$.  A vector bundle of
rank $d-2$ on $X$ splits off a trivial line bundle if and only if its Euler class
vanishes.  If $t\geq3$, every such bundle splits.  The same vanishing gives
the corresponding efficient-generation statement for projective modules.

This is Theorem~\ref{thm:corank-two-vector-bundle-total-space} and
Corollary~\ref{cor:generation-vector-bundle-total-space}.  The secondary
calculation used in its proof separates the three-primary coefficient
complex from a two-primary hermitian quotient and a residual Milnor layer.
On the stable three-primary side the first invariant is the first motivic
reduced power.  Section~\ref{sec:secondary} records the parameterized
comparison that identifies its Thom conjugate with the geometric change of
Euler nullhomotopy.

Our main point is that the individual non-stable sheaves are not the only
relevant objects.  One must also understand the morphisms relating the
punctured affine-space terms at consecutive ranks.  The standard
inclusions $SL_n\hookrightarrow SL_{n+1}$ give $\A^1$-fiber sequences
\[
\A^{n+1}\setminus\{0\}
\longrightarrow
BSL_n
\longrightarrow
BSL_{n+1};
\]
the relevant representability and homogeneous-space results may be found
in \cite{AffineRep2}.  Comparing these sequences for adjacent values of
$n$ produces cross-stage morphisms between the homotopy sheaves of
consecutive punctured affine spaces.  Here the word ``differential''
refers to the cross-composite arising from the linear exact couple; see
\cite[Proposition~2.1.1]{AF17Spheres}.  Its two defining arrows come from
adjacent fiber sequences and are not consecutive arrows in a single long
exact sequence.  Accordingly, the two long exact sequences must be kept
separate.

\subsection*{The second non-stable layer}

The first cross-stage morphism considered in this paper is
\[
d_{n,1}:
\pi_{n+1}^{\A^1}(\A^{n+1}\setminus\{0\})
\longrightarrow
\pi_n^{\A^1}(\A^n\setminus\{0\}).
\]
It is obtained by composing the boundary map from the fiber sequence at
rank $n+1$ with the quotient map arising at rank $n$.  Equivalently, it
is the connecting homomorphism of the two-frame Stiefel fiber sequence
\[
\A^n\setminus\{0\}
\longrightarrow
\operatorname{St}_2(n+1)
\longrightarrow
\A^{n+1}\setminus\{0\};
\]
compare \cite[Definition~12]{AOSS26}.

For $n\geq4$, the Asok--Bachmann--Hopkins calculation
\cite[Theorem~7.2.1]{ABH} places the source and target of $d_{n,1}$ in
exact sequences whose kernels are Milnor $K$-theory sheaves modulo $24$
and whose right-hand terms are higher Grothendieck--Witt sheaves.  We
identify the morphism between these exact sequences.

When $n$ is odd, the projection
\[
\operatorname{St}_2(n+1)
\longrightarrow
\A^{n+1}\setminus\{0\}
\]
admits a symplectic $\A^1$-homotopy section; see the argument of
\cite[Lemma~18]{AOSS26}.  Consequently every connecting homomorphism of
this Stiefel fiber sequence vanishes, and in particular
\[
d_{n,1}=0.
\]

When $n$ is even, the Thom--Euler transgression identifies the induced
Grothendieck--Witt morphism with multiplication by the motivic Hopf class
$\eta$, initially up to a unit in $\mathbf{GW}_0^0(k)$.  The relevant
comparison between obstruction-theoretic and Thom--Euler classes, together
with the explicit Suslin $KO$-degree construction, is developed in
\cite{AF16Euler,AF17}; the corresponding fundamental exact sequence in
hermitian $K$-theory is treated in \cite{Sch17}.  More precisely, for
arbitrarily chosen geometric representatives of the stable-unit generator
there is a unit
$c_n\in\mathbf{GW}_0^0(k)^\times$ such that
\[
(\delta_n^u)_*\circ d_{n,1}
=
c_n\eta\circ(\delta_{n+1}^u)_*.
\]
By choosing the Suslin generators compatibly in adjacent ranks, these units
can be absorbed.  In the stable range there is also an intrinsic spectrum-level
construction: fix the Euler/Gysin sphere orientation used in
\cite[Lemma~3.5]{AF14} and coherent Bott coordinates, and take the
infinite-loop adjoints of the unit $\mathds1\to\mathbf{KO}$.  Relative to
these data the resulting Wood diagram is strictly commutative:
\[
(\delta_n)_*\circ d_{n,1}
=
\eta\circ(\delta_{n+1})_*.
\]
The morphism on the displayed Milnor $K$-theory kernels is zero.  This
strictness is relative to the stated orientation data; module-linearity does
not itself choose those data.  Once those data are imposed, the unit
ambiguity in \cite[Theorem~4.2.2]{AF17} can also be removed for the explicit
matrices.  The contraction compatibility of the classical Suslin class fixes
its algebraic $K$-theory sign.  The Gysin formulas of
\cite[Lemmas~3.5.2--3.5.3]{AF17}, applied after explicit coordinate
permutations, then compute the Witt image of the comparison unit.  Rank and
Witt image together detect an element of $GW(k)$ and give the four-periodic
formula of Theorem~\ref{prop:explicit-suslin-rank}.  In particular, for the
matrix convention of \cite[Definition~3.3.2]{AF17}, the even-stage strict
formula contains $-\eta$ rather than $\eta$; see
Corollary~\ref{cor:explicit-suslin-wood}.  In every normalization the kernels,
cokernels, and the applications developed below are unchanged.

This gives parity-dependent structural descriptions of the contracted
cokernels entering the second non-stable special-linear sheaves.  In odd
stages the cross-stage morphism itself vanishes.  In even stages the
remaining structure is controlled by the Wood sequence together with an
explicit snake-lemma connecting morphism.  Further simplification is
limited by the sheaf-theoretic Suslin--Hurewicz image problem
\cite{AFW20,Rondigs23}.

\subsection*{A contracted part of the next non-stable layer}

One simplicial degree higher, the same construction gives a second
cross-composite
\[
d_{n,2}:
\pi_{n+2}^{\A^1}(\A^{n+1}\setminus\{0\})
\longrightarrow
\pi_{n+1}^{\A^1}(\A^n\setminus\{0\}).
\]
Set
\[
\mathcal P_n:=
\pi_{n+1}^{\A^1}(\A^n\setminus\{0\}).
\]
Then
\[
d_{n,2}:\mathcal P_{n+1}\longrightarrow\mathcal P_n.
\]
As above, its two defining arrows belong to two adjacent long exact
sequences.  The same symplectic Stiefel section
\cite[Lemma~18]{AOSS26} shows that
\[
d_{n,2}=0
\]
before contraction whenever $n$ is odd.

For arbitrary parity, the uncontracted source and target involve the
second stable stem of the motivic sphere.  For $n\geq5$, the stabilization
theorem of \cite{ABH} gives a canonical identification
\[
\mathcal P_n
\cong
\boldsymbol\pi^s_{2-n\alpha}(\mathds1),
\]
and the calculation of the motivic $2$-line in
\cite[Theorem~1.2]{RSO21} gives an exact sequence
\[
0\longrightarrow
\mathbf H^{n+1,n+2}/24
\oplus
\mathbf K^M_{n+4}/2
\longrightarrow
\mathcal P_n
\longrightarrow
\boldsymbol\pi^s_{2-n\alpha}(\mathbf{kq}).
\]
A closely related formulation is recorded in
\cite[Example~6.6]{Gant_Williams_2025}.  Thus the known $2$-line of the
motivic sphere provides a concrete description of the part detected by
very effective hermitian $K$-theory.

More generally, after $r\geq n-2$ contractions, the comparison
$\mathbf{kq}\to\mathbf{KQ}$ in the range calculated in
\cite[Theorem~2.38]{RSO21} allows the periodic $KO$-compatibility to
descend to a $\mathbf{kq}$-visible diagram.  On the visible quotient, the
induced morphism is multiplication by $\eta$ in even stages and zero in
odd stages.  The remaining ambiguity is confined to explicit kernels of
the form
\[
\mathbf H^{a,b}/24
\oplus
\mathbf K^M_c/2.
\]
In particular, the four contractions immediately preceding the boundary
weight give morphisms
\[
\mathcal Q_3\longrightarrow\mathcal Q_2,\qquad
\mathcal Q_2\longrightarrow\mathcal Q_1,\qquad
\mathcal Q_1\longrightarrow\mathcal Q_0,\qquad
\mathcal Q_0\longrightarrow\mathcal Q_{-1},
\]
with horizontal map $\eta$ in even stages and zero in odd stages.

In fact, the visible calculation is the shadow of a sphere-level statement.
Writing $X_{n+1}$ as a simplicial suspension, Morel's simplicial
Freudenthal theorem allows the Stiefel boundary to be restricted to its
first desuspension in both degrees used here.  The bottom homotopy sheaf of
that desuspension is the Euler boundary computed in
\cite[Lemma~3.5]{AF14}.  It follows that, relative to the standard sphere
and purity coordinates used there, the full morphism $d_{n,2}$ is
multiplication by $\eta$ when $n$ is even and is zero when $n$ is odd.  A
change of sphere coordinates can multiply the even-stage map by a unit of
$GW(k)$; this harmless ambiguity disappears after the boundary contraction
because every such unit acts trivially on Milnor \(K\)-theory modulo \(2\).

After sufficiently many contractions, the calculation becomes
particularly simple.  For $n\geq5$, there are canonical identifications
\[
\begin{aligned}
(\mathcal P_n)_{-(n+2)}
&\cong \mathbf K^M_2/2,\\
(\mathcal P_n)_{-(n+3)}
&\cong \mathbf K^M_1/2,\\
(\mathcal P_n)_{-(n+4)}
&\cong \mathbf K^M_0/2.
\end{aligned}
\]
These identifications follow from the motivic $2$-line calculation
\cite[Theorem~1.2]{RSO21}.  The first is a boundary phenomenon.  At this
weight the relevant very effective hermitian $K$-theory sheaf is
\[
\boldsymbol\pi^s_{4,2}(\mathbf{kq})
\cong
\mathbf K^M_0,
\]
and is therefore nonzero
\cite[Theorem~2.38]{RSO21}.  What vanishes is the unit-induced map to this
torsion-free sheaf, since its source is torsion.  Only after one and two
further contractions do the corresponding $\mathbf{kq}$-terms themselves
vanish.

Degree considerations imply
\[
(d_{n,2})_{-(n+3)}=0
\qquad\text{and}\qquad
(d_{n,2})_{-(n+4)}=0.
\]
At the preceding contraction, the source is an extension
\[
0\longrightarrow
\mathbf K^M_3/2
\longrightarrow
(\mathcal P_{n+1})_{-(n+2)}
\longrightarrow
\boldsymbol\mu_{24}
\longrightarrow0,
\]
where the extension is supplied by the bidegree $(3,1)$ case of the
motivic $2$-line calculation \cite[Theorem~1.2]{RSO21}, whereas the target
is $\mathbf K^M_2/2$.  The restriction of
$(d_{n,2})_{-(n+2)}$ to $\mathbf K^M_3/2$ vanishes, so the entire remaining
even-stage ambiguity is reduced to one explicit morphism
\[
\epsilon_n:
\boldsymbol\mu_{24}
\longrightarrow
\mathbf K^M_2/2.
\]
This morphism factors through
$\boldsymbol\mu_{24}/2\boldsymbol\mu_{24}$ and vanishes when $n$ is odd.
At the boundary weight the $\mathbf{kq}$-compatibility square reduces to
$0=0$, so that square alone cannot compute $\epsilon_n$.  The sphere-level
identification instead shows that $\epsilon_n$ is the operation induced by
$\eta$.  Its only possible universal values occur over the fourth and
eighth cyclotomic fields.  At every odd finite place, $2$-complete
Henselian rigidity reduces those values to finite fields, where the target
Milnor $K$-group vanishes.  Since each of the two cyclotomic fields has a
unique place above $2$ and no real place, Brauer--Hasse--Noether reciprocity
forces the last local invariant to vanish as well.  Hence
$\epsilon_n=0$ over every characteristic-zero field.

\subsection*{Realization and vector bundles on split quadrics}

Complex realization extends to a symmetric monoidal functor from the
pointed $\A^1$-homotopy category to the classical pointed homotopy
category \cite[\S3]{MV99}, and gives homomorphisms
\[
\pi_{i,j}^{\A^1}(SL_n)(\C)
\longrightarrow
\pi_{i+j}(SU(n)).
\]
Combining the sheaf-level calculation with classical information about
the two-frame Stiefel boundary, we prove that complex realization is an
isomorphism on the relevant second non-stable special-linear group in
both parities.

More precisely, for $m\geq2$,
\[
\pi^{\A^1}_{2m,\,2m+2}(SL_{2m})(\C)
\cong
\pi_{4m+2}(SU(2m))
\cong
\mathbb Z/2\oplus\mathbb Z/(2m+1)!,
\]
whereas for $m\geq1$,
\[
\pi^{\A^1}_{2m+1,\,2m+3}(SL_{2m+1})(\C)
\cong
\pi_{4m+4}(SU(2m+1))
\cong
\mathbb Z/\big((2m+2)!/2\big).
\]
The corresponding classical non-stable unitary homotopy groups were
calculated by Bott and Kervaire \cite{Bott57,Kervaire60}.  The odd-rank
argument is not a formal five-lemma consequence of the Wood diagram.  Its
additional input is Kervaire's calculation of the two-frame Stiefel
boundary \cite{Kervaire60}, which identifies the deeply contracted
adjacent even-stage differential with the nonzero map
\[
\mathbb Z/2\longrightarrow\mathbb Z/2.
\]
This also gives an independent topological consistency check of the
sheaf-level calculation.

The motivic-sphere realization theorem of
\cite{gant2026motivichomotopygroupsspheres} applies directly to the three
highly contracted groups above.  For $n\geq5$, complex realization
identifies them with
\[
\pi^S_4=0,\qquad
\pi^S_5=0,\qquad
\pi^S_6\cong\mathbb Z/2.
\]
Under this comparison, the terminal class in
$\mathbf K^M_0/2$ realizes to the nonzero class
$\nu^2\in\pi^S_6$; compare also
\cite[Example~6.6]{Gant_Williams_2025}.

These results have applications to vector bundles on the split odd
quadric
\[
Q_{2n-1}
=
\left\{
\sum_{i=1}^n x_i y_i=1
\right\},
\]
which is $\A^1$-weakly equivalent to
$\A^n\setminus\{0\}$; see \cite{MorelA1,AF14}.  The ranks $n-1$, $n-2$,
and $n-3$ are governed successively by the first, second, and third
non-stable layers of the special linear tower.  The rank $n-1$ case is the
Asok--Fasel classification \cite{AF14}.

At the next rank, for $n\geq5$, complex realization induces an
isomorphism
\[
[Q_{2n-1},BSL_{n-2}]_{\A^1}
\xrightarrow{\ \sim\ }
[S^{2n-1},BSU(n-2)]_{\mathrm{top}}.
\]
Consequently,
\[
\mathrm{Vect}^{\mathrm{alg},o}_{n-2}(Q_{2n-1})
\cong
\begin{cases}
\mathbb Z/2\oplus\mathbb Z/(n-1)!,
& n\geq6\text{ even},\\[3pt]
\mathbb Z/\big((n-1)!/2\big),
& n\geq5\text{ odd}.
\end{cases}
\]

At rank $n-3$, the contracted third-layer calculation gives a splitting
theorem rather than a complete classification.  If
$\operatorname{char}(k)=0$, $k^\times=(k^\times)^2$, and $n\geq8$, then
stabilization induces a surjection
\[
\mathrm{Vect}^{\mathrm{alg},o}_{n-4}(Q_{2n-1})
\longrightarrow
\mathrm{Vect}^{\mathrm{alg},o}_{n-3}(Q_{2n-1}).
\]
Hence every oriented rank $n-3$ vector bundle on $Q_{2n-1}$ splits off a
trivial line bundle.

Real realization behaves differently because motivic weight collapses
after passing to real points.  We describe the resulting comparison for
the punctured affine-space terms and identify explicit classes on the
second stable stem, using the description in
\cite[Example~6.6]{Gant_Williams_2025}.  A direct real analogue of the
comparison theorem for $SL_n$, however, requires additional information
about the relevant special-linear groups.

\subsection*{Efficient generation of projective modules}

The same non-stable sheaves enter the Stiefel obstruction theory for
efficient generation of projective modules developed in \cite{AOSS26}.
Let
\[
X=\operatorname{Spec}R
\]
be a smooth affine scheme of dimension $d$, and let $M$ be a projective
$R$-module of rank $r$.  The fiber sequence
\[
\operatorname{St}_r(N)
\longrightarrow
Gr_r(N)
\longrightarrow
BGL_r
\]
encodes whether $M$ can be generated by $N$ elements
\cite{AOSS26}.

The first non-stable coefficients appear in the results of
Asok--Opie--Shin--Syed \cite{AOSS26} concerning the bounds
$r+d-1$ and $r+d-2$.  At the next stage, relevant to an $r+d-3$ bound,
the sheaf $\mathcal P_{d-2}$ maps directly into the corresponding Stiefel
coefficient, and the following boundary lands in the first non-stable
sheaf studied through $d_{n,1}$.  Thus the two adjacent non-stable layers
occur in one exact obstruction-theoretic segment.

For $d\geq7$, the top contraction satisfies
\[
(\mathcal P_{d-2})_{-d}
\cong
\mathbf K^M_2/2.
\]
It follows that the punctured-affine-space contribution to the top
obstruction for an $r+d-3$ generation bound is annihilated by $2$.  If the
ground field has characteristic $0$ and $2$-cohomological dimension at
most $1$, the norm-residue theorem
\cite{OVV07} forces the relevant codimension-$d$ Rost--Schmid terms to
vanish.

Under one additional, explicit secondary-cohomology vanishing hypothesis,
this yields a corank-two splitting criterion: a rank \(d-2\) vector bundle
splits off a trivial line bundle if and only if its Euler class vanishes.
Applied to complementary modules, it gives an \(r+d-3\) generation criterion for
projective modules of rank $r\geq4$ that are already generated by
$r+d-2$ elements.  The remaining secondary obstruction is expressed in
terms of a concrete Rost--Schmid complex; for the general theory of these
complexes and strictly $\A^1$-invariant sheaves, see \cite{MorelA1}.
Its $3$-primary summand is the usual Rost--Schmid complex for Milnor
$K$-theory modulo $3$.  The corresponding stable \(k\)-invariant is the
first reduced power \(P^1\).  Thom conjugation for the given vector bundle
adds the mod-\(3\) Wu class
\(\operatorname{red}_3(c_1^2-2c_2)\).  A parameterized comparison through
the first two relative stages proves that the resulting corrected operation governs the
change of the three-primary secondary obstruction and gives a canonical
cokernel class.  Surjectivity of this corrected reduced-power map then
removes the three-primary obstruction.  Its
\(2\)-primary summand is the Milnor complex
modulo \(8\), together with valuation parity and one explicit order-two
hermitian correction.  The hermitian quotient of its cohomology is the
cokernel of a twisted Steenrod square, and the remaining layer is controlled
by Milnor \(K\)-theory modulo \(8\).  This gives a second corank-two criterion
whose hypotheses are the surjectivity of \(Sq_L^2\) and of the
Thom-corrected reduced-power map, together with the vanishing of one Milnor
\(K\)-cohomology group modulo \(8\).  The same comparison proves
lift-independent additivity of the full secondary change.  Its
two-primary stable \(k\)-invariant is not yet expressed in standard
cohomology operations.
Independently of the cohomological-dimension hypothesis, the secondary
group vanishes when \(X\) is the total space of a vector bundle of
relative rank at least two over a smooth affine scheme.
This gives an unconditional family of the corank-two splitting criterion;
in relative rank at least three, the Euler group also vanishes and every
bundle under consideration splits.

\subsection*{Scope and organization}

The calculations using the Asok--Bachmann--Hopkins description
\cite{ABH} and the second stable stem
\cite{RSO21} are carried out over fields of characteristic $0$.  Results
recalled from \cite{AF14} are used under the hypotheses stated there,
usually over infinite perfect fields of characteristic different from
$2$.

Two issues delimit the applications developed here.  The first is the
even-stage morphism
\[
\epsilon_n:
\boldsymbol\mu_{24}
\longrightarrow
\mathbf K^M_2/2.
\]
Proposition~\ref{prop:dn2-stable-eta} and
Theorem~\ref{thm:epsilon-global-zero} compute it and prove that it vanishes.
The first-stem module relation and the Milnor--Witt residue product rule
determine the remaining $2$-primary hermitian extension part of the Gersten
differential governing the secondary obstruction in the corank-two
splitting problem.  Its only extra term is an order-two correction inside
the Milnor kernel modulo \(8\), while its $3$-primary part is ordinary
Milnor \(K\)-theory cohomology.  The associated sphere operation is
\(P^1\operatorname{red}_3\), and its Thom conjugate for \(E\) is
\(\mathcal P_E^1\operatorname{red}_3\).  Identifying the latter with
nullhomotopy dependence in the geometric lifting problem follows from the
augmentation-one stable section comparison below.  The hermitian quotient
of the two-primary cohomology is measured by the twisted Steenrod square
\(Sq_L^2\), leaving a Milnor \(K\)-theory layer modulo \(8\).  The secondary change is
lift-independent and additive, so its two-primary part has a
hermitian projection and, on its kernel, a residual Milnor component.  The
second issue is to compute these two action maps in standard cohomological
terms.
The target groups vanish on the vector-bundle total spaces just described.
For the explicit Suslin matrices, the full comparison unit is determined
below over every field after fixing the standard ordered Koszul orientation.
Its rank and Witt image are computed separately and combine to a
four-periodic element of \(GW(k)^\times\).

Section~2 fixes notation and recalls contraction, Milnor and
Milnor--Witt $K$-theory, higher Grothendieck--Witt theory, and the Wood
cofiber sequence.  Section~3 studies $d_{n,1}$ and the second non-stable
special-linear sheaves.  It proves the Suslin--Stiefel--Wood compatibility
up to a Grothendieck--Witt unit, constructs a compatible normalization for
which the Wood diagram is strict, separates the two parity outcomes, and
records the limitation imposed by the Suslin--Hurewicz image problem.
Section~4 additionally proves strict compatibility for the Euler-oriented
unit-adjoint system, computes the full comparison with the printed Suslin
matrices by combining their algebraic $K$-theory and Witt coordinates, and
introduces $d_{n,2}$.  It keeps the two defining long exact sequences
separate, records the input from the $2$-line of the motivic sphere, and
performs the contracted calculation.  It also determines the
$\mathbf{kq}$-visible quotient after $r\geq n-2$ contractions, identifies
the full stable cross-stage map by simplicial desuspension, and proves the
vanishing of the resulting boundary defect
\[
\boldsymbol\mu_{24}\longrightarrow\mathbf K^M_2/2.
\]
Section~5 treats complex and real realization, vector bundles on split
odd quadrics, and applications to efficient generation of projective
modules.  It also separates the secondary obstruction into a reduced-power
stable invariant and its Thom--Wu correction at \(3\), and a
twisted-Steenrod/Milnor exact sequence at \(2\).  It constructs the
parameterized comparison, proves additivity, and turns these coefficient
calculations into geometric orbit formulas.

\subsection*{Acknowledgment}

The second author would like to express his sincere gratitude to Professor
Aravind Asok for his invaluable guidance, constant encouragement, and
unwavering support throughout his Ph.D. studies and the completion of this
work.  The first author would also like to thank Morgan Opie, Keyao Peng, Biman Roy, Ben Williams and Shangjie Zhang for helpful
discussions.

\section{Preliminaries}
This section collects notation and foundational results used throughout the
paper.  We first fix conventions for simplicial Nisnevich sheaves on
\(\mathpzc{Sm}_k\) and recall the simplicial and \(\A^1\)-localized homotopy
categories.  We then review contraction and its relation to the bigrading on
\(\A^1\)-homotopy sheaves. Finally, we recall unramified Milnor and Milnor--Witt \(K\)-theory sheaves,
the fundamental ideal in the Witt ring, and higher Grothendieck--Witt theory.
The Wood cofiber sequence, equivalently Karoubi's fundamental exact sequence,
provides the canonical Grothendieck--Witt morphisms in our main diagrams.

\subsection{Basic notation}

Assume that \(k\) is a field.  We write \(\mathpzc{Sm}_k\) for the category
of smooth separated \(k\)-schemes of finite type.

We write \(\mathpzc{Spc}_k\), respectively
\(\mathpzc{Spc}_{k,\bullet}\), for the category of simplicial, respectively
pointed simplicial, Nisnevich sheaves on \(\mathpzc{Sm}_k\).  We call their
objects \(k\)-spaces, or simply spaces when the context is clear.

We write \(\mathpzc{H}_s(k)\), respectively
\(\mathpzc{H}_{s,\bullet}(k)\), for the associated simplicial homotopy
category, respectively its pointed version.  Here cofibrations are
monomorphisms and weak equivalences are stalkwise weak equivalences.

The Morel--Voevodsky \(\A^1\)-homotopy category is denoted by
\(\mathpzc{H}(k)\), and its pointed version by
\(\mathpzc{H}_\bullet(k)\) \cite{MV99}.  They are obtained by localizing the
corresponding simplicial homotopy categories at \(\A^1\)-weak equivalences.

For spaces \(\mathpzc{X}\) and \(\mathpzc{Y}\), set
\[
[\mathpzc{X},\mathpzc{Y}]_s := Hom_{\mathpzc{H}_s(k)}(\mathpzc{X},\mathpzc{Y})
\quad\text{and}\quad
[\mathpzc{X},\mathpzc{Y}]_{\A^1} := Hom_{\mathpzc{H}(k)}(\mathpzc{X},\mathpzc{Y}),
\]
and use analogous notation in the pointed categories.

We denote Milnor--Witt and Milnor \(K\)-theory by \(K^{MW}_*(k)\) and
\(K^M_*(k)\), respectively.  The graded ring \(K^{MW}_*(k)\) is generated by
symbols \([u]\), for \(u\in k^\times\), in degree \(1\), and by
\(\eta\) in degree \(-1\), subject to the usual relations
\cite[Definition~3.1]{MorelA1}.  There is a canonical isomorphism
\[
K^M_*(k)\cong K^{MW}_*(k)/(\eta),
\]
so Milnor \(K\)-theory is the quotient by the ideal generated by \(\eta\).

Let \(GW(k)\) and \(W(k)\) be the Grothendieck--Witt and Witt rings of
nondegenerate symmetric bilinear forms over \(k\).  Thus \(W(k)\) is the
quotient of \(GW(k)\) by the ideal generated by the hyperbolic form
\(h=\langle1\rangle\oplus\langle-1\rangle\).  Let
\(I(k)\subset W(k)\) be the fundamental ideal, the kernel of the rank map
to \(\mathbb Z/2\), and write \(I^n(k)\) for its powers.  For \(m\geq1\),
\(I^m(k)\) is additively generated by \(m\)-fold Pfister forms; by convention,
\(I^r(k)=W(k)\) for \(r\leq0\).

The relationships between these objects are as follows. There are natural graded ring homomorphisms
\[
K^{MW}_*(k)\longrightarrow K^{M}_*(k),
\qquad
K^{MW}_*(k)\longrightarrow I^*(k),
\]
where the first is the quotient map modulo the ideal generated by $\eta$, and the second sends a pure symbol
\[
[a_1]\cdots [a_n]
\]
to the $n$-fold Pfister form
\[
\langle\!\langle a_1,\dots,a_n\rangle\!\rangle;
\]
see, for example,~\cite[\S3]{MorelA1}.

Let
\[
k_*(k):=K^M_*(k)/2K^M_*(k).
\]
There is a canonical graded ring homomorphism
\[
K^M_*(k)\longrightarrow k_*(k),
\]
and a graded ring homomorphism
\[
K^M_*(k)\longrightarrow I^*(k)/I^{*+1}(k)
\]
sending a pure symbol $\{a_1,\dots,a_n\}$ to the Pfister form
\[
\langle\!\langle a_1,\dots,a_n\rangle\!\rangle.
\]
The Milnor conjecture on quadratic forms \cite[\S4]{OVV07} says that this
homomorphism factors through \(k_*(k)\) and induces an isomorphism
\[
k_*(k)\xrightarrow{\ \sim\ } I^*(k)/I^{*+1}(k).
\]
Composing the inverse of this isomorphism with the projection
\[
I^*(k)\twoheadrightarrow I^*(k)/I^{*+1}(k)
\]
yields a graded ring homomorphism
\[
I^*(k)\longrightarrow k_*(k).
\]
These maps fit into a cartesian square of graded rings; see~\cite[Theorem~5.3]{Morel04}:
\[
\begin{tikzcd}
K^{MW}_*(k) \ar[r] \ar[d] & K^{M}_*(k) \ar[d] \\
I^*(k) \ar[r] & k_*(k).
\end{tikzcd}
\]

The preceding square sheafifies to a cartesian square of unramified sheaves
of graded rings \cite{MorelA1,Morel05}:
\[
\begin{tikzcd}
\mathbf{K}^{MW}_* \ar[r] \ar[d] & \mathbf{K}^{M}_* \ar[d] \\
\mathbf{I}^* \ar[r] & \mathbf{K}^{M}_*/2.
\end{tikzcd}
\]

The following consequence is due to Morel
\cite[\S5]{Morel04}.
For negative degrees we use the standard graded convention
\(\mathbf K_n^M=0\) and \(\mathbf I^n=\mathbf W\) for \(n<0\).

\begin{proposition}\label{prop:MW-filtration}
For every $n\in \mathbb{Z}$ there are short exact sequences
\[
0 \longrightarrow \mathbf{I}^{n+1} \longrightarrow \mathbf{K}^{MW}_n \longrightarrow \mathbf{K}^M_n \longrightarrow 0
\]
and
\[
0 \longrightarrow 2\,\mathbf{K}^M_n \longrightarrow \mathbf{K}^{MW}_n \longrightarrow \mathbf{I}^{n} \longrightarrow 0.
\]
Moreover, multiplication by \(\eta\) factors as
\[
\mathbf{K}^{MW}_n \longrightarrow \mathbf{I}^{n} \longrightarrow \mathbf{K}^{MW}_{n-1},
\]
where the maps are those appearing in the short exact sequences above.
\end{proposition}

\subsection{Contracted homotopy sheaves}\label{subsec:contraction}

If $\mathcal{M}$ is a strictly $\A^1$-invariant sheaf, its \emph{contraction} $\mathcal{M}_{-1}$ is defined by
\[
\mathcal{M}_{-1}(U)
:=\ker\!\Big(\mathcal{M}(\G_m\times U)\xrightarrow{(1 \times id)^*}\mathcal{M}(U)\Big),
\]
where $1:\operatorname{Spec}(k)\to \G_m$ is the unit map.

Some basic properties of contraction are recorded in the following result of
Morel \cite[Lemmas~2.32 and~7.33]{MorelA1}.

\begin{lemma}\label{lem:contraction-exact}
The assignment $\mathcal{M}\mapsto \mathcal{M}_{-1}$ defines an endofunctor on the category of
strictly (or strongly) $\A^1$-invariant sheaves which preserves exact sequences.
\end{lemma}

\begin{remark}\label{rem:contraction-iterated}
For $i\ge 1$, the $i$-fold contraction is defined inductively by
\[
\mathcal{M}_{-i}:=(\mathcal{M}_{-(i-1)})_{-1}.
\]
\end{remark}

In particular, Lemma~\ref{lem:contraction-exact} implies that contraction is an exact functor on
strongly $\A^1$-invariant sheaves. Consequently, if $\mathcal{M}$ is strictly $\A^1$-invariant,
then so is its contraction $\mathcal{M}_{-1}$ (and hence every iterated contraction
$\mathcal{M}_{-j}$ for $j\ge 0$).

We will also use Morel's compatibility between contraction and the bigrading
on $\A^1$-homotopy sheaves \cite[Theorem~6.13]{MorelA1}.

\begin{theorem}\label{thm:contraction-bigrading}
If $(\mathpzc{X},x)$ is an $\A^1$-connected pointed space, then for any integers $i\ge 1$ and $j\ge 0$ there are canonical isomorphisms
\[
{\pi}^{\A^1}_{i,j}(\mathpzc{X},x)\ \cong\ \big({\pi}^{\A^1}_{i}(\mathpzc{X},x)\big)_{-j}.
\]
\end{theorem}

We will repeatedly use the contraction behavior of Milnor and Milnor--Witt
$K$-theory sheaves; the following standard identifications are recalled from
\cite[Lemma~2.7]{AF14} and \cite[Proposition~2.9]{AF15b}.

\begin{theorem}\label{thm:contraction-K}
For any integers $i,j\ge 0$ and any integer $n>0$, there are canonical isomorphisms
\[
(\mathbf{K}^M_i)_{-j}\ \cong\
\begin{cases}
\mathbf{K}^M_{i-j}, & \text{if } j\le i,\\[2pt]
0, & \text{if } j> i,
\end{cases}
\qquad
(\mathbf{K}^Q_i)_{-j}\ \cong\
\begin{cases}
\mathbf{K}^Q_{i-j}, & \text{if } j\le i,\\[2pt]
0, & \text{if } j> i.
\end{cases}
\]
Moreover, for any $j\ge 0$ there are canonical isomorphisms
\[
(\mathbf{I}^n)_{-j}\ \cong\ \mathbf{I}^{n-j},
\qquad\text{and}\qquad
(\mathbf{K}^{MW}_n)_{-j}\ \cong\ \mathbf{K}^{MW}_{n-j}.
\]
\end{theorem}

\subsection{Higher Grothendieck--Witt theory}\label{sec:HGW}

\noindent
Assume that \(2\) is invertible in \(k\).  For a smooth \(k\)-scheme \(X\)
and a line bundle \(L\) on \(X\), we write \(GW_i^j(X,L)\) for
Schlichting's nonconnective higher Grothendieck--Witt groups
\cite{Sch2010b,Sch17}.  Both indices may be arbitrary integers; the upper
index records the shift of the duality and is \(4\)-periodic.  We write
\(GW_i^j(X)=GW_i^j(X,\mathcal O_X)\), and \(\mathbf{GW}_i^j\) for the
associated strictly \(\A^1\)-invariant Nisnevich sheaf.

Our indexing is fixed by
\[
\eta\in GW_{-1}^{-1}(k)\cong W(k).
\]
Multiplication by \(\eta\) gives natural maps
\[
\eta:GW_i^j(X,L)\longrightarrow GW_{i-1}^{j-1}(X,L).
\]
With these conventions, Karoubi's fundamental exact sequence is
\begin{equation}\label{eq:karoubi-convention}
\cdots\longrightarrow K_i(X)
\xrightarrow{H}GW_i^j(X,L)
\xrightarrow{\eta}GW_{i-1}^{j-1}(X,L)
\xrightarrow{f}K_{i-1}(X)
\xrightarrow{H}GW_{i-1}^j(X,L)
\longrightarrow\cdots ,
\end{equation}
where \(H\) and \(f\) are the hyperbolic and forgetful maps.  This is the
long exact sequence induced by the motivic Wood cofiber sequence.  Every
negative Grothendieck--Witt index below is interpreted using this
nonconnective convention.

\section{The second non-stable homotopy sheaves}

In this section we study the next non-stable layer in the
\(\A^1\)-homotopy sheaves of the special linear groups.  We first recall the
stable range and the first non-stable sheaves computed by Asok--Fasel
\cite{AF14}, and then turn to \(\pi^{\A^1}_{n}(SL_n)\).  As \(n\) varies, we
refer to these as the \emph{second non-stable} \(\A^1\)-homotopy sheaves of
\(SL_n\).

Our starting point is the $\A^1$-fiber sequence
\[
SL_{n+1}/SL_n \simeq_{\A^1} \A^{n+1}\setminus\{0\}\longrightarrow BSL_n\longrightarrow BSL_{n+1}.
\]
Expanding this sequence yields the following diagram:
\[
\begin{tikzcd}[column sep=small, row sep=small]
\Omega_s^{1}(SL_{n+1}/SL_n) \ar[d] & & & SL_{n+1}/SL_n \ar[d] \\
\Omega_s^{1}BSL_n \ar[r] & SL_n/SL_{n-1} \ar[r] & BSL_{n-1} \ar[r] & BSL_n \ar[d] \\
& & & BSL_{n+1}.
\end{tikzcd}
\]
Applying $\motpi_n$ to this diagram yields the composite morphism
\[
\begin{aligned}
d_{n,1}:\ \motpi_{n+1}(\A^{n+1}\setminus\{0\})
&\xrightarrow{\ \delta_{n,1}\ }\motpi_{n+1}(BSL_n)\\
&\xrightarrow{\ q_{n-1,1}\ }\motpi_n(\A^n\setminus\{0\}).
\end{aligned}
\]
We interpret $d_{n,1}$ as the relevant ``differential'' controlling the passage from the first
to the second non-stable range.  For odd \(n\), the Stiefel section determines
this differential.  For even \(n\), the Thom--Euler transgression proves the
  comparison, up to a Grothendieck--Witt unit, with the quotients of
  \cite[Theorem~7.2.1]{ABH}.  A compatible choice of quotient generators
  makes the resulting diagram strictly commutative.  We also treat the
  low-dimensional cases \(n=2,3\).

The resulting sheaf-level diagram uses this compatible normalization; its
kernel and cokernel statements are independent of the choice.
Section~5 determines the deeper contraction needed for complex realization
directly from Kervaire's calculation of the classical two-frame Stiefel
boundary, giving an independent check after evaluation on \(\mathbb C\).

\subsection{Computation in the stable range}\label{subsec:stable-range}

We briefly recall the description of the $\A^1$-homotopy sheaves of $GL_n$
and $SL_n$ in the stable range.  The next result summarizes the features from
\cite[Lemma~3.1 and Theorem~3.2]{AF14} that we will need.

\begin{theorem}\label{thm:stable-range}
Assume that \(k\) is a perfect field of characteristic different from \(2\).
\begin{enumerate}
\item For every $i\ge 0$ and $n\ge 1$, there are natural isomorphisms
\[
\pi^{\A^1}_i(GL_n)\ \cong\ \pi^{\A^1}_{i+1}(BGL_n),
\qquad
\pi^{\A^1}_i(SL_n)\ \cong\ \pi^{\A^1}_{i+1}(BSL_n).
\]
Moreover, the canonical morphism $BSL_n \to BGL_n$ induces isomorphisms
\[
\pi^{\A^1}_{i+1}(BSL_n)\ \xrightarrow{\ \sim\ }\ \pi^{\A^1}_{i+1}(BGL_n)
\qquad\text{for all } i\ge 1.
\]

\item Let \(i>0\) and \(n>1\).  The stabilization morphisms
\[
\pi^{\A^1}_i(SL_{n-1}) \longrightarrow \pi^{\A^1}_i(SL_n)
\]
are surjective for \(i\leq n-2\) and are isomorphisms for
\(i\leq n-3\).  In particular, for \(1\leq i\leq n-2\) there is a canonical
identification
\[
  \pi^{\A^1}_i(SL_n)\ \cong\ \mathbf{K}^Q_{i+1},
\]
where \(\mathbf K^Q_{i+1}\) denotes the corresponding Quillen
\(K\)-theory sheaf.
\end{enumerate}
\end{theorem}

\subsection{\texorpdfstring{Computation of
\(\pi^{\A^1}_{n-2}(SL_{n-1})\)}
{Computation of the first adjacent non-stable group}}
\label{subsec:pi-n-2-SLn-1}

To analyze ${\pi}^{\A^1}_{n-2}(SL_{n-1})$, we begin with the $\A^1$-fiber sequence
\[
SL_n/SL_{n-1}\longrightarrow BSL_{n-1}\longrightarrow BSL_n
\]
This sequence is recalled in \cite{MorelA1}; its long exact sequence in
\(\A^1\)-homotopy sheaves
contains the segment
\begin{equation}\label{eq:les-segment}
{\pi}^{\A^1}_{n}(BSL_n)\longrightarrow
{\pi}^{\A^1}_{n-1}(SL_n/SL_{n-1})\longrightarrow
{\pi}^{\A^1}_{n-1}(BSL_{n-1})\longrightarrow
{\pi}^{\A^1}_{n-1}(BSL_n).
\end{equation}
For \(n\geq3\), the terms are sheaves of abelian groups and are therefore
strictly \(\A^1\)-invariant.  In particular,
\[
{\pi}^{\A^1}_{1}(SL_{2})\cong \mathbf{K}^{MW}_{2},
\qquad
{\pi}^{\A^1}_{1}(SL_{3})\cong \mathbf{K}^{M}_{2}
\]
by \cite[Chapter~6]{MorelA1}; the higher cases follow from standard
\(\A^1\)-connectivity.

Next, recall the canonical $\A^1$-weak equivalence
\[
\A^{n}\setminus\{0\}\ \simeq_{\A^1}\ S^{n-1}_s\wedge \G_m^{\wedge n}.
\]
Moreover, the quotient $SL_n/SL_{n-1}$ is $\A^1$-weakly equivalent to $\A^n\setminus\{0\}$. Consequently,
\begin{equation}\label{eq:pi-quotient}
{\pi}^{\A^1}_{n-1}(SL_n/SL_{n-1})
\ \cong\
{\pi}^{\A^1}_{n-1}(\A^{n}\setminus\{0\})
\ \cong\
\mathbf{K}^{MW}_n.
\end{equation}
Using the canonical identifications
\[
{\pi}^{\A^1}_{n}(BSL_n)\cong {\pi}^{\A^1}_{n-1}(SL_n),
\qquad
{\pi}^{\A^1}_{n-1}(BSL_{n-1})\cong {\pi}^{\A^1}_{n-2}(SL_{n-1}),
\qquad
{\pi}^{\A^1}_{n-1}(BSL_n)\cong {\pi}^{\A^1}_{n-2}(SL_n),
\]
the segment \eqref{eq:les-segment} may be rewritten as
\begin{equation}\label{eq:stable-segment}
{\pi}^{\A^1}_{n-1}(SL_n)\xrightarrow{\,q_{n-1}\,}
\mathbf{K}^{MW}_n \xrightarrow{\,\delta_{n-1}\,}
{\pi}^{\A^1}_{n-2}(SL_{n-1})
\longrightarrow
{\pi}^{\A^1}_{n-2}(SL_n).
\end{equation}
Finally, ${\pi}^{\A^1}_{n-2}(SL_n)$ lies in the stable range, hence
\[
{\pi}^{\A^1}_{n-2}(SL_n)\ \cong\ \mathbf{K}^{Q}_{n-1}.
\]
Therefore, \eqref{eq:stable-segment} yields an exact sequence
\begin{equation}\label{eq:exact-seq-pi}
{\pi}^{\A^1}_{n-1}(SL_n)\xrightarrow{\,q_{n-1}\,}
\mathbf{K}^{MW}_n \xrightarrow{\,\delta_{n-1}\,}
{\pi}^{\A^1}_{n-2}(SL_{n-1})
\longrightarrow
\mathbf{K}^{Q}_{n-1}
\longrightarrow 0.
\end{equation}
Thus, understanding ${\pi}^{\A^1}_{n-2}(SL_{n-1})$ reduces to identifying the map
\(
q_{n-1}:{\pi}^{\A^1}_{n-1}(SL_n)\to \mathbf{K}^{MW}_n
\)
and, more specifically, the composite that appears after one more stabilization step.

\smallskip
\noindent\textbf{The composite $q_{n-1}\circ\delta_n$.}
Consider the $\A^1$-fiber sequence
\[
SL_{n+1}/SL_n\longrightarrow BSL_n\longrightarrow BSL_{n+1}.
\]
As in the introduction, it gives rise to a commutative diagram
\begin{center}
\begin{tikzcd}[column sep=large, row sep=large]
\Omega_s^{1}(SL_{n+1}/SL_n) \ar[d] & & & SL_{n+1}/SL_n \ar[d]\\
\Omega_s^{1}BSL_n \ar[r] & SL_n/SL_{n-1} \ar[r] & BSL_{n-1} \ar[r] & BSL_n \ar[d]\\
& & & BSL_{n+1}.
\end{tikzcd}
\end{center}
Applying ${\pi}^{\A^1}_{n-1}(-)$ and using the identifications
\[
\begin{aligned}
{\pi}^{\A^1}_{n-1}\!\big(\Omega_s^{1}(SL_{n+1}/SL_n)\big)
&\cong
{\pi}^{\A^1}_{n}(SL_{n+1}/SL_n)
\cong
\mathbf{K}^{MW}_{n+1},\\
{\pi}^{\A^1}_{n-1}(\Omega_s^{1}BSL_n)
&\cong
{\pi}^{\A^1}_{n}(BSL_n)
\cong
{\pi}^{\A^1}_{n-1}(SL_n),
\end{aligned}
\]
we see that the induced map
\[
\mathbf{K}^{MW}_{n+1}\longrightarrow \mathbf{K}^{MW}_n
\]
is precisely the composite
\begin{equation}\label{eq:composite-q-delta}
\mathbf{K}^{MW}_{n+1}\xrightarrow{\,\delta_n\,}
{\pi}^{\A^1}_{n-1}(SL_n)
\xrightarrow{\,q_{n-1}\,}
\mathbf{K}^{MW}_n.
\end{equation}

The quotient \(SL_{n+1}/SL_n\) is
\(\A^1\)-\((n-1)\)-connected.  For a strictly \(\A^1\)-invariant sheaf
\(\mathbf A\), maps out of its first nontrivial homotopy sheaf are detected
by Nisnevich cohomology \cite{AABD}.  Thus
\eqref{eq:composite-q-delta} is determined by a class in
\[
H^n_{\mathrm{Nis}}\big(SL_{n+1}/SL_n,\mathbf{K}^{MW}_n\big).
\]
Using the identification $SL_{n+1}/SL_n \simeq_{\A^1} Q_{2n+1}$, where $Q_{2n+1}\subset \A^{2n+2}$ is the split quadric
\[
\sum_{i=1}^{n+1} x_i y_i = 1,
\]
we may view this class as an element of
\[
H^n_{\mathrm{Nis}}(Q_{2n+1},\mathbf{K}^{MW}_n).
\]
Moreover, the morphism \(SL_{n+1}/SL_n\to BSL_n\) classifies the
tautological \(SL_n\)-torsor
\(SL_{n+1}\to SL_{n+1}/SL_n\).  The resulting cohomology class is Morel's
Euler class \cite{MorelA1}.  Fasel computes it using unimodular rows
\cite{faseldegree}, and \cite[Lemma~3.5]{AF14} yields the following formula.

\begin{theorem}\label{thm:q-delta}
Assume that \(k\) is a perfect field of characteristic different from \(2\).
For \(n\geq2\), the composite \eqref{eq:composite-q-delta} satisfies
\[
q_{n-1}\circ \delta_n=
\begin{cases}
\eta, & \text{if } n=2m,\\
0, & \text{if } n=2m+1.
\end{cases}
\]
\end{theorem}

\smallskip
\noindent\textbf{The first non-stable sheaf of $BSL_n$.}
Following \cite{AF14}, we introduce the sheaves \(\mathbf S_n\) and
\(\mathbf T_n\) used to describe the first non-stable
\(\A^1\)-homotopy sheaf of \(BSL_n\).  Let
\(\psi_n:\mathbf K^Q_n\to\mathbf K^M_n\) be the natural morphism for
\(n\geq2\), and put \(\mathbf S_n:=\operatorname{coker}(\psi_n)\).  There is
an epimorphism
\(\mathbf K^M_n/(n-1)!\twoheadrightarrow\mathbf S_n\); for odd
\(n\geq3\), there is also a morphism
\(\mathbf S_n\to\mathbf K^M_n/2\).  For such odd \(n\), define
\(\mathbf T_n\) by the fiber square
\[
\begin{tikzcd}
\mathbf{T}_n \ar[r] \ar[d] & \mathbf{I}^n \ar[d]\\
\mathbf{S}_n \ar[r] & \mathbf{K}^M_n/2.
\end{tikzcd}
\]

\begin{theorem}\label{thm:first-nonstable-BSL}
Assume that \(k\) is a perfect field of characteristic different from \(2\).
For every \(m\geq1\), there are short exact sequences of strictly
\(\A^1\)-invariant sheaves
\[
0\longrightarrow \mathbf{T}_{2m+1}\longrightarrow
{\pi}^{\A^1}_{2m}(BSL_{2m})
\longrightarrow \mathbf{K}^{Q}_{2m}\longrightarrow 0,
\]
and
\[
0\longrightarrow \mathbf{S}_{2m+2}\longrightarrow
{\pi}^{\A^1}_{2m+1}(BSL_{2m+1})
\longrightarrow \mathbf{K}^{Q}_{2m+1}\longrightarrow 0.
\]
\end{theorem}

\subsection{\texorpdfstring{Computation of
\(\pi^{\A^1}_{n}(SL_n)\)}
{Computation of the second adjacent non-stable group}}
\label{subsec:pi-n-SLn}

To analyze the structure of ${\pi}^{\A^1}_{n}(SL_n)$, we begin with the $\A^1$-fiber sequence
\begin{equation}\label{eq:fiber-BSL}
SL_{n+1}/SL_n \ \simeq_{\A^1}\ \A^{n+1}\setminus\{0\}
\longrightarrow BSL_n \longrightarrow BSL_{n+1}.
\end{equation}
Applying ${\pi}^{\A^1}_n(-)$ to \eqref{eq:fiber-BSL}, we obtain the following diagram:
\begin{equation}\label{eq:diagram-pi-n}
\begin{tikzcd}[column sep=0.25in, row sep=large]
{\pi}^{\A^1}_n\!\big(\Omega_s^1(SL_{n+1}/SL_n)\big)
\ar[d, "\delta_{n,1}"']
\ar[dr, dashed, "d_{n,1}"] & &
{\pi}^{\A^1}_n(SL_{n}/SL_{n-1})\ar[d] \\
{\pi}^{\A^1}_n\!\big(\Omega_s^1 BSL_n\big) \ar[r, "q_{n-1,1}"] &
{\pi}^{\A^1}_n(SL_n/SL_{n-1}) \ar[r] &
{\pi}^{\A^1}_n(BSL_{n-1}) \ar[r] &
{\pi}^{\A^1}_n(BSL_n) \ar[d]\\
& & &
{\pi}^{\A^1}_n(BSL_{n+1}).
\end{tikzcd}
\end{equation}
Here the dashed arrow is the composite
\[
d_{n,1}:=q_{n-1,1}\circ \delta_{n,1} :
{\pi}^{\A^1}_n\!\big(\Omega_s^1(SL_{n+1}/SL_n)\big)
\longrightarrow
{\pi}^{\A^1}_n(SL_n/SL_{n-1}).
\]

From now on we focus on the morphism $d_{n,1}$.  For $n\ge 2$, we have $\A^1$-weak equivalences
\[
SL_{n+1}/SL_n \simeq_{\A^1} \A^{n+1}\setminus\{0\},
\qquad
SL_n/SL_{n-1} \simeq_{\A^1} \A^{n}\setminus\{0\},
\]
and hence the map $d_{n,1}$ may be viewed as a morphism
\[
d_{n,1}:\ {\pi}^{\A^1}_{n+1}\!\big(\A^{n+1}\setminus\{0\}\big)
\longrightarrow
{\pi}^{\A^1}_{n}\!\big(\A^{n}\setminus\{0\}\big).
\]

If \(k\) has characteristic \(0\) and \(n\geq4\), then
\cite[Theorem~7.2.1]{ABH} gives exact sequences
\begin{equation}\label{eq:ABH-exact-next}
0 \longrightarrow \mathbf{K}^M_{n+3}/24
\xrightarrow{\,(\nu_{n+1})_*\,}
{\pi}^{\A^1}_{n+1}\!\big(\A^{n+1}\setminus\{0\}\big)
\xrightarrow{\,(\delta_{n+1})_*\,}
\mathbf{GW}^{n+1}_{n+2},
\end{equation}
and
\begin{equation}\label{eq:ABH-exact-n}
0 \longrightarrow \mathbf{K}^M_{n+2}/24
\xrightarrow{\,(\nu_{n})_*\,}
{\pi}^{\A^1}_{n}\!\big(\A^{n}\setminus\{0\}\big)
\xrightarrow{\,(\delta_{n})_*\,}
\mathbf{GW}^{n}_{n+1}.
\end{equation}
The morphism \(d_{n,1}\) fits into a commutative diagram relating
\eqref{eq:ABH-exact-next} and \eqref{eq:ABH-exact-n}; it is induced by the
composite \(q_{n-1,1}\circ\delta_{n,1}\) in
\eqref{eq:diagram-pi-n}.
The next lemma provides a more detailed description of this diagram.

\begin{lemma}\label{lem:stiefel-parity}
Let \(n\geq3\).  The cross-composite \(d_{n,1}\) is the connecting
homomorphism
\[
\pi_{n+1}^{\A^1}(\A^{n+1}\setminus0)
\longrightarrow
\pi_n^{\A^1}(\A^n\setminus0)
\]
associated with the Stiefel fiber sequence
\[
\A^n\setminus0\longrightarrow
\operatorname{St}_2(n+1)\longrightarrow
\A^{n+1}\setminus0.
\]
If \(n\) is odd, the last morphism has an \(\A^1\)-homotopy section and,
consequently, \(d_{n,1}=0\).  In fact, every connecting homomorphism of this
fiber sequence is zero in that parity.
\end{lemma}

\begin{proof}
Using the quotient descriptions
\[
\operatorname{St}_2(n+1)\cong SL_{n+1}/SL_{n-1},
\qquad
\A^{n+1}\setminus0\simeq_{\A^1} SL_{n+1}/SL_n,
\]
the frame-forgetting morphism has fiber
\(SL_n/SL_{n-1}\simeq_{\A^1}\A^n\setminus0\).  Naturality of the homotopy long exact
sequence identifies its boundary with the composite of the boundary for
\(SL_n\to SL_{n+1}\to\A^{n+1}\setminus0\) and the quotient morphism
\(SL_n\to\A^n\setminus0\).  This composite is exactly
\(q_{n-1,1}\circ\delta_{n,1}=d_{n,1}\).  Equivalently, this is the instance
of the frame-forgetting fiber sequence in \cite[Definition~12]{AOSS26} with
\(j=n-1\).
If \(n\) is odd, then \(j\) is even.  In this case the natural morphism from
the rank-two symplectic Stiefel variety supplies a section of
\(\operatorname{St}_2(n+1)\to\A^{n+1}\setminus0\); see the proof of
\cite[Lemma~18]{AOSS26}.  A fiber sequence with a section has zero connecting
homomorphisms.
\end{proof}

\medskip\noindent
\textbf{\(KO\)-degree maps and their normalization ambiguity.}
Put \(X_m:=\A^m\setminus0\), and let
\[
\mathcal K_m:=\Omega_{\mathbb P^1}^{-m}\mathbf O
\]
be the geometric Bott model used in \cite[Theorem~2.2.2]{AF17}.  For every
\(m\geq2\), choose an arbitrary geometric representative
\[
\Psi_m^u:X_m\longrightarrow\mathcal K_m
\]
of the stable-unit generator, as in the discussion preceding
\cite[Theorem~7.2.1]{ABH}, and write
\[
(\delta_m^u)_*:
\pi_m^{\A^1}(X_m)\longrightarrow\mathbf{GW}^{m}_{m+1}
\]
for the induced morphism.  For \(m\geq4\) this is the right-hand morphism in
\eqref{eq:ABH-exact-n}.  For \(m=2,3\), it is the low-rank \(KO\)-degree
morphism identified in \cite[Examples~4.4.3--4.4.4]{AF17}.  The explicit
Suslin-matrix map of \cite[Definition~3.3.5]{AF17} is a generator of
\[
[X_m,\mathcal K_m]_{\A^1}\cong\mathbf{GW}_0^0(k)
\]
by \cite[Theorem~3.3.6]{AF17}.  Its comparison with the natural homomorphism
induced by the stable unit is only up to multiplication by a unit; see
\cite[Theorem~4.2.2]{AF17}.  We retain this ambiguity rather than building a
choice into the notation.

\begin{lemma}\label{lem:relative-transgression}
Let
\[
F\longrightarrow E\xrightarrow{\,p\,}B
\]
be a pointed \(\A^1\)-fiber sequence, with connecting morphism
\(\partial:\Omega_sB\to F\).  Let \(A\) be a grouplike
\(\A^1\)-local \(H\)-space with a delooping \(B_sA\).  A pointed map
\(\alpha:B\to B_sA\), together with a pointed nullhomotopy
\(h:\alpha p\simeq *\), determines a transgression
\(\tau_h(\alpha):F\to A\) for which the square
\begin{equation}\label{eq:relative-transgression-square}
\begin{tikzcd}[column sep=large]
\Omega_sB \ar[r,"\partial"] \ar[d,"\Omega_s\alpha"']&
F \ar[d,"\tau_h(\alpha)"]\\
\Omega_sB_sA \ar[r,"\simeq"']&A
\end{tikzcd}
\end{equation}
commutes in the pointed \(\A^1\)-homotopy category.  The square is natural
under maps of deloopable grouplike \(H\)-spaces.
\end{lemma}

\begin{proof}
The nullhomotopy \(h\) makes
\[
\begin{tikzcd}
E \ar[r] \ar[d,"p"']& * \ar[d]\\
B \ar[r,"\alpha"']&B_sA
\end{tikzcd}
\]
a homotopy-commutative pointed square.  Passing to homotopy fibers of the
vertical maps gives
\[
\tau_h(\alpha):
\operatorname{hofib}(p)\longrightarrow
\operatorname{hofib}(*\to B_sA).
\]
The source and target are \(F\) and \(\Omega_sB_sA\simeq A\),
respectively.  Naturality of the connecting morphism of a homotopy-fiber
square is exactly \eqref{eq:relative-transgression-square}.  Applying a map
\(A\to A'\) to the displayed homotopy-commutative square proves the last
assertion.  No stabilization or interchange of suspension spectra with
non-stable loop spaces is used.
\end{proof}

\begin{lemma}\label{lem:KO-euler-transgression}
Let \(\gamma_n\) be the universal oriented rank-\(n\) bundle on \(BSL_n\),
and let
\[
e_n^{KO}:BSL_n\longrightarrow B_s\mathcal K_n
\]
represent its \(KO\)-Euler class.  There is a unit
\(a_n\in\mathbf{GW}_0^0(k)^\times\) such that
\begin{equation}\label{eq:KO-euler-loop}
\Omega_s e_n^{KO}\simeq
a_n\bigl(\Psi_n^u\circ q_{n-1}\bigr).
\end{equation}
\end{lemma}

\begin{proof}
Let
\[
s_{n-1}:BSL_{n-1}\longrightarrow BSL_n
\]
be stabilization.  The pullback of \(\gamma_n\) is
\(\gamma_{n-1}\oplus\mathcal O\).  Its distinguished nowhere-zero section
gives a pointed nullhomotopy
\[
h_n:s_{n-1}^*e_n^{KO}\simeq *.
\]
Apply Lemma~\ref{lem:relative-transgression} to
\[
X_n\longrightarrow BSL_{n-1}\xrightarrow{\,s_{n-1}\,}BSL_n
\]
and to the relative class \((e_n^{KO},h_n)\).  Writing
\(\tau_n^{KO}:X_n\to\mathcal K_n\) for the resulting transgression gives the
commutative square
\begin{equation}\label{eq:KO-euler-transgression-square}
\begin{tikzcd}[column sep=large]
SL_n\simeq\Omega_sBSL_n
  \ar[r,"q_{n-1}"] \ar[d,"\Omega_se_n^{KO}"']&
X_n \ar[d,"\tau_n^{KO}"]\\
\mathcal K_n \ar[r,equal]&\mathcal K_n.
\end{tikzcd}
\end{equation}

Let \(V_n\to BSL_n\) be the total space of \(\gamma_n\), and let
\(V_n^\circ\) be the complement of its zero section.  There are
\(\A^1\)-equivalences
\[
V_n\simeq_{\A^1}BSL_n,
\qquad
V_n^\circ\simeq_{\A^1}BSL_{n-1},
\]
and homotopy purity identifies the cofiber of \(V_n^\circ\to V_n\) with
\(\operatorname{Th}(\gamma_n)\).  For the second equivalence, the stabilizer
of a nonzero vector is
\(SL_{n-1}\ltimes\mathbb A^{n-1}\), whose unipotent radical is
\(\A^1\)-contractible.  This is the oriented analogue of the universal
fiber--cofiber model in \cite[Lemma~5.2.1]{AF16Euler}.

Homotopy purity gives a cofiber sequence
\[
V_n^\circ\longrightarrow V_n\longrightarrow\operatorname{Th}(\gamma_n).
\]
Under the two equivalences above, its relative connecting construction is
the homotopy-fiber construction defining \(\tau_n^{KO}\).  Restricting the
\(KO\)-Thom class to an oriented trivial fiber, and then applying the
loop--suspension and Bott adjunctions, gives a map
\(\Phi_n:X_n\to\mathcal K_n\).  Naturality of the purity cofiber square
identifies
\[
\tau_n^{KO}\simeq\Phi_n.
\]
This is the universal Thom transgression displayed for Milnor--Witt
cohomology in \cite[Proposition~5.3.1]{AF16Euler}; the construction uses only
the functoriality of cofibers and therefore applies to the represented
theory \(KO\).  Substitution in
\eqref{eq:KO-euler-transgression-square} gives
\[
\Omega_s e_n^{KO}\simeq\Phi_n\circ q_{n-1}.
\]

The Thom isomorphism shows that \(\Phi_n\) is a generator, while
\(\Psi_n^u\) is a generator by its construction as a geometric
desuspension of the unit.  Thus both maps represent generators of the free
rank-one \(\mathbf{GW}_0^0(k)\)-module
\([X_n,\mathcal K_n]_{\A^1}\).  They therefore differ by a unit \(a_n\).
Substitution gives \eqref{eq:KO-euler-loop}.  Notice that this argument
does not identify that unit.
\end{proof}

\begin{lemma}\label{lem:kernel-KO-euler}
Let
\[
i_n:X_{n+1}\longrightarrow BSL_n
\]
be the fiber inclusion in
\(X_{n+1}\to BSL_n\to BSL_{n+1}\).  If \(n\) is even, there is a unit
\(b_n\in\mathbf{GW}_0^0(k)^\times\) such that
\begin{equation}\label{eq:kernel-KO-euler}
i_n^*e_n^{KO}=b_n\eta\,\Psi_{n+1}^u,
\end{equation}
where the right-hand side is viewed, by Bott periodicity, as a map to
\(B_s\mathcal K_n\).  If \(n\) is odd, \(i_n^*e_n^{KO}=0\).
\end{lemma}

\begin{proof}
The map \(i_n\) classifies the oriented kernel bundle \(E_n\) in
\[
0\longrightarrow E_n\longrightarrow\mathcal O_{X_{n+1}}^{n+1}
\xrightarrow{(x_1,\ldots,x_{n+1})}
\mathcal O_{X_{n+1}}\longrightarrow0.
\]
Thus \(i_n^*e_n^{KO}=e_{KO}(E_n)\) by pullback naturality.
The displayed exact sequence also supplies a stable framing
\[
E_n\oplus\mathcal O\cong\mathcal O^{n+1}.
\]
Consequently the zero section of \(E_n\) has a framed
Pontryagin--Thom Euler class.  Applying the unit
\(u:\mathds1\to\mathbf{KO}\) to its Thom collapse gives
\(e_{KO}(E_n)\): this follows directly by applying \(u\) to the collapse
map and then pulling the resulting Thom class back along the zero section.

The discussion preceding \cite[Lemma~3.5]{AF14} identifies
\(q_{n-1}\delta_n\) with the obstruction-theoretic Euler class of \(E_n\).
Under
\[
H^n_{\mathrm{Nis}}(X_{n+1},\mathbf K_n^{MW})
\cong\mathbf K_{-1}^{MW},
\]
\cite[Lemma~3.5]{AF14} computes this class as \(\eta\) for even \(n\) and
as zero for odd \(n\).  The comparison of this obstruction class with the
relative-Hurewicz image of the framed Thom class is only determined up to a
unit in \(\mathbf{GW}_0^0(k)\).  Indeed,
\cite[Proposition~5.3.1]{AF16Euler} gives the commuting purity and relative
Hurewicz diagram, and \cite[Theorem~5.3.2]{AF16Euler} identifies its
obstruction and Thom Euler classes up to precisely such a unit.

Here \(X_{n+1}\) is a single motivic sphere.  The relative Hurewicz
comparison and the Bott adjunctions give the vertical isomorphisms in the
following diagram; the lower morphism defines the upper
change-of-coefficients morphism \(u_*\):
\begin{equation}\label{eq:unit-coefficient-square}
\begin{tikzcd}[column sep=large]
H^n_{\mathrm{Nis}}(X_{n+1},\mathbf K_n^{MW})
  \ar[r,"u_*"] \ar[d,"\cong"']&
{[X_{n+1},B_s\mathcal K_n]_{\A^1,*}}
  \ar[d,"\cong"]\\
\mathbf K_{-1}^{MW}(k)
  \ar[r,"\mu_{-1}"']&
\mathbf{GW}_{-1}^{-1}(k).
\end{tikzcd}
\end{equation}
The lower map is the homomorphism on zero-stem coefficients induced by
\(u\), as explained in \cite[Remark~4.1.3]{AF17}.  Naturality of the
relative Hurewicz map says that, on the framed Thom class used above,
\(u_*\) is obtained by applying \(u\) to the Thom collapse before taking
its relative Hurewicz image.  Thus the diagram records an actual
change-of-coefficients construction, rather than an identification chosen
only after evaluating the Euler class.

Combining this naturality with the comparison of the relative-Hurewicz and
Thom generators in \cite[Theorem~5.3.2]{AF16Euler} gives a unit
\(\rho_n\in\mathbf{GW}_0^0(k)^\times\) such that
\begin{equation}\label{eq:euler-change-of-coefficients}
u_*\bigl(e_{\mathrm{ob}}(E_n)\bigr)
=\rho_n e_{KO}(E_n).
\end{equation}
This equation is the only point at which the Euler-class normalization
enters the proof.

By \cite[Theorem~4.1.2]{AF17}, \(\mu_{-1}\) is an isomorphism and sends the
motivic Hopf class \(\eta\) to the Karoubi element \(\eta\).  The framed
Thom-collapse observation above, the commutativity of
\eqref{eq:unit-coefficient-square}, and
\eqref{eq:euler-change-of-coefficients} therefore show that
\[
e_{KO}(E_n)=
\begin{cases}
r_n\eta,&n\ \text{even},\\
0,&n\ \text{odd},
\end{cases}
\qquad
r_n\in\mathbf{GW}_0^0(k)^\times,
\]
under the right-hand vertical identification in
\eqref{eq:unit-coefficient-square}.  Finally,
\cite[Theorem~4.2.2]{AF17} compares the geometric desuspension of the
stable-unit generator with the Suslin \(KO\)-degree generator, again up to
a unit \(s_n\in\mathbf{GW}_0^0(k)^\times\).  Thus
\(b_n=r_ns_n\) gives \eqref{eq:kernel-KO-euler}.  The odd-rank equality is
literal because the zero class is unaffected by either normalization.
\end{proof}

\begin{proposition}\label{prop:suslin-stiefel-wood-up-to-unit}
Let \(k\) be a field of characteristic \(0\), and let \(n\geq2\).  For the
chosen geometric maps \((\delta_m^u)_*\), there is a unit
\(c_n\in\mathbf{GW}_0^0(k)^\times\) such that
\begin{equation}\label{eq:wood-compatibility-up-to-unit}
(\delta_n^u)_*\circ d_{n,1}
=
\begin{cases}
c_n\eta\circ(\delta_{n+1}^u)_*,&n\text{ even},\\
0,&n\text{ odd}.
\end{cases}
\end{equation}
\end{proposition}

\begin{proof}
The boundary
\[
\delta_{n,1}:\Omega_sX_{n+1}\longrightarrow SL_n
\]
is \(\Omega_si_n\), and \(q_{n-1}\delta_{n,1}\) induces \(d_{n,1}\) on
\(\pi_n^{\A^1}\).  Lemma~\ref{lem:KO-euler-transgression} and functoriality
of looping give
\[
a_n\Psi_n^uq_{n-1}\delta_{n,1}
\simeq
\Omega_se_n^{KO}\,\Omega_si_n
\simeq
\Omega_s(i_n^*e_n^{KO})
\]
as maps \(\Omega_sX_{n+1}\to\mathcal K_n\).  For even \(n\),
Lemma~\ref{lem:kernel-KO-euler} identifies the last map with
\[
b_n\Omega_s(\eta\Psi_{n+1}^u):
\Omega_sX_{n+1}\longrightarrow\mathcal K_n,
\]
where \(\eta\Psi_{n+1}^u:X_{n+1}\to B_s\mathcal K_n\) is interpreted using
the Bott equivalence from that lemma.  Now
\[
\pi_n^{\A^1}(\Omega_sX_{n+1})
=\pi_{n+1}^{\A^1}(X_{n+1}),
\qquad
\pi_n^{\A^1}(\mathcal K_n)=\mathbf{GW}_{n+1}^{n}.
\]
Under these identifications,
\(\pi_n^{\A^1}(\Psi_n^uq_{n-1}\delta_{n,1})\) is
\((\delta_n^u)_*d_{n,1}\), while
\(\pi_n^{\A^1}(\Omega_s(\eta\Psi_{n+1}^u))\) is
\(\eta(\delta_{n+1}^u)_*\).  Setting \(c_n=a_n^{-1}b_n\) therefore gives
\eqref{eq:wood-compatibility-up-to-unit}.  For odd \(n\),
Lemma~\ref{lem:stiefel-parity} already gives \(d_{n,1}=0\).
\end{proof}

\begin{corollary}\label{prop:suslin-stiefel-wood}
Let \(k\) be a field of characteristic \(0\).  For every \(m\geq2\), one can choose a unit
\(\lambda_m\in\mathbf{GW}_0^0(k)^\times\) and put
\[
(\delta_m)_*:=\lambda_m(\delta_m^u)_*
\]
so that, for every \(n\geq2\),
\begin{equation}\label{eq:wood-compatibility}
(\delta_n)_*\circ d_{n,1}
=
\vartheta_n\circ(\delta_{n+1})_*,
\qquad
\vartheta_n=
\begin{cases}
\eta,&n\text{ even},\\
0,&n\text{ odd}.
\end{cases}
\end{equation}
For \(m\geq4\), postcomposing the quotient map in
\eqref{eq:ABH-exact-n} with \(\lambda_m\) preserves its exactness.  The same
is true of the low-rank quotient sequences for \(m=2,3\).
\end{corollary}

\begin{proof}
Choose \(\lambda_n\) arbitrarily for even \(n\geq2\), and impose
\(\lambda_{n+1}=\lambda_nc_n\).  These are all the compatibility
conditions, since the odd-stage differential is zero.  Multiplying
\eqref{eq:wood-compatibility-up-to-unit} by \(\lambda_n\) gives
\eqref{eq:wood-compatibility}.  Since multiplication by \(\lambda_m\) is an
automorphism of the Grothendieck--Witt sheaf, it does not alter the kernel
or image of the quotient morphism in \eqref{eq:ABH-exact-n}.
\end{proof}

\begin{remark}\label{rem:wood-compatibility-status}
From now on, \((\delta_m)_*\) denotes a compatible normalization supplied
by Corollary~\ref{prop:suslin-stiefel-wood}.  The strict equality
\eqref{eq:wood-compatibility} is therefore a statement about a compatible
choice of Suslin generators.  For arbitrary geometric representatives, the
unconditional statement is \eqref{eq:wood-compatibility-up-to-unit}.
Proposition~\ref{prop:oriented-unit-wood} below removes this ambiguity in the
stable range after the sphere coordinates are given the Euler/Gysin
orientation used in \cite[Lemma~3.5]{AF14} and coherent Bott coordinates are
fixed.  These orientation data are hypotheses, not consequences of
module-linearity.  This argument alone does not compare those data with the
explicit Suslin matrices or with a prescribed algebraic $K$-theory
orientation.  The convention-level calculation needed for that comparison is
carried out in Theorem~\ref{prop:explicit-suslin-rank} below.

The proof above by itself does not identify the Stiefel boundary with a
stable sphere map.  That identification will instead be obtained by first
restricting the actual connecting map along a simplicial
loop--suspension unit and only then using the Freudenthal isomorphism; no
commutation of suspension spectra with non-stable loop spaces is asserted.
\end{remark}

\begin{proposition}\label{prop:abhcase4}
Let \(k\) be a field of characteristic \(0\), and let \(n\geq4\).  Then
for the compatible normalization fixed in
Remark~\ref{rem:wood-compatibility-status},
there is a commutative diagram
\begin{equation}\label{eq:ABH-diagram}
\begin{tikzcd}
0 \arrow[r]
  & \mathbf K^M_{n+3}/24
      \arrow[r,"(\nu_{n+1})_*"]
      \arrow[d,"0"]
  & \pi_{n+1}^{\A^1}(\A^{n+1}\setminus0)
      \arrow[r,"(\delta_{n+1})_*"]
      \arrow[d,"d_{n,1}"]
  & \mathbf{GW}_{n+2}^{n+1}
      \arrow[d,"\vartheta_n"]\\
0 \arrow[r]
  & \mathbf K^M_{n+2}/24
      \arrow[r,"(\nu_n)_*"]
  & \pi_n^{\A^1}(\A^n\setminus0)
      \arrow[r,"(\delta_n)_*"]
  & \mathbf{GW}_{n+1}^{n},
\end{tikzcd}
\end{equation}
where \(d_{n,1}=q_{n-1,1}\circ\delta_{n,1}\), and
\[
\vartheta_n=
\begin{cases}
\eta,&n\text{ even},\\
0,&n\text{ odd}.
\end{cases}
\]
\end{proposition}

\begin{proof}
Corollary~\ref{prop:suslin-stiefel-wood} gives the commutativity of the
right square.  Exactness of the two rows then produces a unique morphism
\[
\varphi:\mathbf K^M_{n+3}/24\longrightarrow\mathbf K^M_{n+2}/24
\]
making the left square commute.

The morphism \((\nu_{n+1})_*\) is induced from
\(\mathbf K^{MW}_{n+3}\) through the canonical epimorphism
\(\mathbf K^{MW}_{n+3}\twoheadrightarrow\mathbf K^M_{n+3}/24\);
see \cite[Theorem~5.2.9]{AWW17} and
\cite[Theorem~3.18]{AFW20}.  Hence \(\varphi\) is detected by a morphism
\[
\mathbf K^{MW}_{n+3}\longrightarrow\mathbf K^M_{n+2}/24.
\]
By \cite[Lemma~5.1.3]{AWW17}, the group of such morphisms is
\[
\bigl(\mathbf K^M_{n+2}/24\bigr)_{-(n+3)}(k)
\cong\mathbf K^M_{-1}(k)/24=0.
\]
Thus \(\varphi=0\), which proves the left square.
\end{proof}

We next treat the two low-dimensional cases \(n=3\) and \(n=2\).  To
describe \(d_{3,1}\), we first recall the structure of the
sheaf ${\pi}^{\A^1}_3(\A^3\setminus\{0\})$ from
\cite[Theorem~7.2.2]{ABH}.

\begin{proposition}\label{prop:ABH-722}
Assume that $k$ is a perfect field. Then the following statements hold.
\begin{enumerate}
\item For any $d\ge 0$ there are isomorphisms
\[
\pi_{4+r,\,r}\!\big(S^{3+d,d}\big)\ \xrightarrow{\ \sim\ }\ \pi_{d+r,\,r}\!\big(S^{2d-1,d}\big).
\]

\item For any $d\ge 1$ there is a short exact sequence
\begin{multline*}
0\longrightarrow
\mathrm{coker}\!\Big(\pi_{6+r,\,r}\!\big(S^{4+d,d}\big)\longrightarrow
\mathbf{K}^{MW}_{2d-r}\Big/\!h^{\frac{1+(-1)^{d+1}}{2}}\Big)
\\
\longrightarrow
\pi_{3+r,\,r}\!\big(S^{2+d,d}\big)
\longrightarrow
\pi_{4+r,\,r}\!\big(S^{3+d,d}\big)
\longrightarrow 0.
\end{multline*}

\item If, in addition, $\mathrm{char}(k)=0$, then for any $d\ge 2$ there is a short exact sequence
\[
0\longrightarrow
\mathbf{K}^{M}_{d+2-r}/24
\longrightarrow
\pi_{4+r,\,r}\!\big(S^{3+d,d}\big)
\longrightarrow
\mathbf{GW}^{d-r}_{d+1-r},
\]
and the right-hand map is surjective for $r\ge d-3$.
\end{enumerate}
\end{proposition}

We apply Proposition~\ref{prop:ABH-722} with $(d,r)=(3,0)$. Using (1) to identify
${\pi}^{\A^1}_{3}(\A^3\setminus\{0\})\cong \pi_{4,0}(S^{6,3})$ and then (3), we obtain
a short exact sequence
\begin{equation}\label{eq:pi3A3-ABH}
0 \longrightarrow \mathbf{K}^M_{5}/24
\longrightarrow {\pi}^{\A^1}_{3}\!\big(\A^3\setminus\{0\}\big)
\longrightarrow \mathbf{GW}^{3}_{4}
\longrightarrow 0.
\end{equation}

\begin{lemma}\label{lem:abhcase3-KM}
Let $k$ be a field of characteristic $0$. There is a commutative diagram of strictly
$\A^1$-invariant sheaves
\begin{equation}\label{eq:abhcase3-KM}
\begin{tikzcd}
0 \ar[r] &
\mathbf{K}^M_{6}/24 \ar[r,"(\nu_{4})_*"] \ar[d,"0"'] &
{\pi}^{\A^1}_{4}\!\big(\A^{4}\setminus\{0\}\big) \ar[r,"(\delta_{4})_*"] \ar[d,"d_{3,1}"'] &
\mathbf{GW}^{4}_{5} \ar[d,"0"] \\
0 \ar[r] &
\mathbf{K}^M_{5}/24 \ar[r,"(\nu_{3})_*"'] &
{\pi}^{\A^1}_{3}\!\big(\A^{3}\setminus\{0\}\big) \ar[r,"(\delta_{3})_*"'] &
\mathbf{GW}^{3}_{4} \ar[r] & 0,
\end{tikzcd}
\end{equation}
where $d_{3,1}:=q_{2,1}\circ \delta_{3,1}$.  Moreover, \(d_{3,1}=0\).
\end{lemma}

\begin{proof}
This is the case \(n=3\) of Lemma~\ref{lem:stiefel-parity}.  Thus the middle
vertical morphism is zero, and the two remaining vertical morphisms are zero
as well.  The rows are the ABH sequences recalled above, so the displayed
diagram is commutative.
\end{proof}

\smallskip

For \(n=2\), the equivalence
\(\A^2\setminus0\simeq_{\A^1}SL_2=Sp_2\) and
\cite[Theorem~3.3]{AF14b} give a short exact sequence
\begin{equation}\label{eq:pi2A2-low-rank}
0\longrightarrow\mathbf T'_4
\longrightarrow\pi^{\A^1}_2(\A^2\setminus0)
\xrightarrow{\,(\delta_2)_*\,}\mathbf{GW}^2_3
\longrightarrow0.
\end{equation}
Here \(\mathbf T'_4\) denotes the kernel sheaf in that theorem, and the
identification of its quotient morphism with a Suslin \(KO\)-degree map
follows from \cite[Example~4.4.3]{AF17}.  Postcomposition by the unit used
in the compatible normalization does not change its kernel or exactness.
Likewise, \cite[Example~4.4.4]{AF17} identifies the quotient morphism in
\eqref{eq:pi3A3-ABH} with the rank-three \(KO\)-degree map, and compatible
renormalization again preserves that exact sequence.

\begin{proposition}\label{prop:low-rank-n2-diagram}
Let \(k\) be a field of characteristic \(0\), and use the compatible
normalization of Corollary~\ref{prop:suslin-stiefel-wood}.  There is a unique
morphism
\[
\kappa_2:\mathbf K^M_5/24\longrightarrow\mathbf T'_4
\]
for which the following diagram has exact rows and is commutative:
\begin{equation}\label{eq:low-rank-n2-diagram}
\begin{tikzcd}[column sep=large, row sep=large]
0 \ar[r]&
\mathbf K^M_5/24 \ar[r,"(\nu_3)_*"] \ar[d,"\kappa_2"']&
\pi^{\A^1}_3(\A^3\setminus0)
  \ar[r,"(\delta_3)_*"] \ar[d,"d_{2,1}"']&
\mathbf{GW}^3_4 \ar[r]\ar[d,"\eta"]&
0\\
0 \ar[r]&
\mathbf T'_4 \ar[r]&
\pi^{\A^1}_2(\A^2\setminus0)
  \ar[r,"(\delta_2)_*"']&
\mathbf{GW}^2_3 \ar[r]&
0.
\end{tikzcd}
\end{equation}
No vanishing assertion is made about \(\kappa_2\).
\end{proposition}

\begin{proof}
Corollary~\ref{prop:suslin-stiefel-wood}, applied with \(n=2\), gives
\[
(\delta_2)_*\circ d_{2,1}
=\eta\circ(\delta_3)_*,
\]
which proves commutativity of the right square.  Hence \(d_{2,1}\) carries
\(\ker(\delta_3)_*=\mathbf K^M_5/24\) into
\(\ker(\delta_2)_*=\mathbf T'_4\).  Restriction to these kernels produces the
unique morphism \(\kappa_2\) and proves commutativity of the left square.
Exactness is \eqref{eq:pi3A3-ABH} and \eqref{eq:pi2A2-low-rank}.
\end{proof}

The preceding results give the following summary in the range used below.

\begin{theorem}\label{thm:main-diagrams}
Let $k$ be a field of characteristic $0$, and fix a compatible normalization
as in Corollary~\ref{prop:suslin-stiefel-wood}.

\smallskip
\noindent\textup{(i)} Let \(n\geq4\).  Then there is a commutative diagram
\begin{equation}
\begin{tikzcd}[column sep=large, row sep=large]
0 \ar[r] &
\mathbf{K}^M_{n+3}/24 \ar[r,"(\nu_{n+1})_*"] \ar[d,"0"'] &
\pi^{\A^1}_{n+1}\!\big(\A^{n+1}\setminus\{0\}\big) \ar[r,"(\delta_{n+1})_*"] \ar[d,"q_{n-1,1}\circ \delta_{n,1}"'] &
\mathbf{GW}^{n+1}_{n+2} \ar[d,"\vartheta_n"] \\
0 \ar[r] &
\mathbf{K}^M_{n+2}/24 \ar[r,"(\nu_n)_*"] &
\pi^{\A^1}_{n}\!\big(\A^{n}\setminus\{0\}\big) \ar[r,"(\delta_{n})_*"] &
\mathbf{GW}^{n}_{n+1}.
\end{tikzcd}
\end{equation}
where \(\vartheta_n=\eta\) for even \(n\) and \(\vartheta_n=0\) for odd
\(n\).  In particular, the middle vertical morphism is zero when \(n\) is
odd.

\smallskip
\noindent\textup{(ii)} If $n=3$, there is a commutative diagram
\begin{equation}
\begin{tikzcd}[column sep=large, row sep=large]
0 \ar[r] &
\mathbf{K}^M_{6}/24 \ar[r] \ar[d,"0"'] &
\pi^{\A^1}_{4}\!\big(\A^{4}\setminus\{0\}\big) \ar[r,"(\delta_4)_*"] \ar[d,"q_{2,1}\circ \delta_{3,1}"'] &
\mathbf{GW}^{4}_{5} \ar[d,"0"] \\
0 \ar[r] &
\mathbf{K}^M_{5}/24 \ar[r,"(\nu_3)_*"'] &
\pi^{\A^1}_{3}\!\big(\A^{3}\setminus\{0\}\big) \ar[r,"(\delta_3)_*"'] &
\mathbf{GW}^{3}_{4} \ar[r] & 0.
\end{tikzcd}
\end{equation}

\smallskip
\noindent\textup{(iii)} If \(n=2\), there is a commutative diagram with
exact rows
\begin{equation}
\begin{tikzcd}[column sep=large, row sep=large]
0 \ar[r]&
\mathbf K^M_5/24 \ar[r,"(\nu_3)_*"] \ar[d,"\kappa_2"']&
\pi^{\A^1}_3(\A^3\setminus0)
  \ar[r,"(\delta_3)_*"] \ar[d,"q_{1,1}\circ\delta_{2,1}"']&
\mathbf{GW}^3_4 \ar[r]\ar[d,"\eta"]&
0\\
0 \ar[r]&
\mathbf T'_4 \ar[r]&
\pi^{\A^1}_2(\A^2\setminus0)
  \ar[r,"(\delta_2)_*"']&
\mathbf{GW}^2_3 \ar[r]&
0.
\end{tikzcd}
\end{equation}
The morphism \(\kappa_2\) is the kernel morphism of
Proposition~\ref{prop:low-rank-n2-diagram}; it is not asserted to vanish.

\end{theorem}

After interpreting the composite as a differential, we return to the left half of the original diagram.
Consider the diagram
\begin{equation}\label{eq:left-half-cokers}
\begin{tikzcd}[column sep=0.25in, row sep=large]
\pi_{n+1}^{\A^1}\!\big(\A^{n+1}\setminus\{0\}\big)
  \ar[d,"\delta_{n,1}"']
  \ar[dr,dotted,"d_{n,1}"]
\\
\pi_{n}^{\A^1}(SL_n)
  \ar[r,"q_{n-1,1}"']
  \ar[d]
&
\pi_{n}^{\A^1}\!\big(\A^{n}\setminus\{0\}\big)
  \ar[r]
  \ar[d]
&
\pi_{n-1}^{\A^1}(SL_{n-1})
  \ar[r,"(\mathrm{inc}_{n-1})_*"]
&
\pi_{n-1}^{\A^1}(SL_n)
  \ar[r,"q_{n-1}"]
&
\mathbf{K}^{MW}_{n}
\\
\mathrm{coker}(\delta_{n,1})
  \ar[r]
  \ar[d]
&
\mathrm{coker}(d_{n,1})
\\
0.
\end{tikzcd}
\end{equation}
Our first goal is to understand $\mathrm{coker}(d_{n,1})$.

\medskip\noindent
For \(n\geq4\), Theorem~\ref{thm:main-diagrams}\textup{(i)} identifies
\(d_{n,1}\) with the
vertical map in the following diagram with exact rows:
\begin{equation}\label{eq:ABH-coker-diagram}
\begin{tikzcd}[column sep=large, row sep=large]
0 \ar[r] &
\mathbf K^{M}_{n+3}/24 \ar[r,"(\nu_{n+1})_*"] \ar[d,"0"'] &
\pi_{n+1}^{\A^1}\!\big(\A^{n+1}\setminus\{0\}\big) \ar[r,"(\delta_{n+1})_*"] \ar[d,"d_{n,1}"'] &
\mathbf{GW}^{n+1}_{n+2} \ar[d,"\vartheta_n"] \\
0 \ar[r] &
\mathbf K^{M}_{n+2}/24 \ar[r,"(\nu_n)_*"] &
\pi_{n}^{\A^1}\!\big(\A^{n}\setminus\{0\}\big) \ar[r,"(\delta_n)_*"] &
\mathbf{GW}^{n}_{n+1}.
\end{tikzcd}
\end{equation}
Passing to cokernels in the middle vertical map yields an induced morphism
\[
\mathrm{coker}(d_{n,1}) \longrightarrow \mathrm{coker}(\vartheta_n).
\]

\medskip\noindent
Now apply $(n-2)$-fold contraction. Since contraction is exact on strictly $\A^1$-invariant sheaves
(Lemma~\ref{lem:contraction-exact}), it commutes with cokernels. Contracting
\eqref{eq:ABH-coker-diagram} therefore yields a commutative diagram with exact rows
\begin{equation}\label{eq:ABH-coker-diagram-contracted}
\begin{tikzcd}[column sep=large, row sep=large]
0 \ar[r] &
\mathbf K^{M}_{5}/24 \ar[r] \ar[d,"0"'] &
\big(\pi_{n+1}^{\A^1}(\A^{n+1}\setminus\{0\})\big)_{-(n-2)}
  \ar[r]
  \ar[d,"(d_{n,1})_{-(n-2)}"'] &
\big(\mathbf{GW}^{n+1}_{n+2}\big)_{-(n-2)}
  \ar[d,"(\vartheta_n)_{-(n-2)}"]
\\
0 \ar[r] &
\mathbf K^{M}_{4}/24 \ar[r] &
\big(\pi_{n}^{\A^1}(\A^{n}\setminus\{0\})\big)_{-(n-2)}
  \ar[r] &
\big(\mathbf{GW}^{n}_{n+1}\big)_{-(n-2)}.
\end{tikzcd}
\end{equation}
In particular,
\[
\mathrm{coker}(d_{n,1})_{-(n-2)}\ \cong\ \mathrm{coker}\bigl((d_{n,1})_{-(n-2)}\bigr),
\qquad
\mathrm{coker}(\vartheta_n)_{-(n-2)}\ \cong\ \mathrm{coker}\bigl((\vartheta_n)_{-(n-2)}\bigr).
\]
Both rows of \eqref{eq:ABH-coker-diagram-contracted} are short exact: for the
upper row the surjectivity condition is attained at the boundary value
\(r=d-3\), and for the lower row it lies one step inside that range; compare
Proposition~\ref{prop:ABH-722}\textup{(3)} and
\cite[Theorem~4.4.5]{AF17}.

\begin{theorem}\label{thm:coker-dn1-structure}
Let \(k\) be a field of characteristic \(0\), let $n\ge 4$, and write
\[
\overline d_{n,1}:=(d_{n,1})_{-(n-2)},
\qquad
\overline\vartheta_n:=(\vartheta_n)_{-(n-2)}.
\]
If \(n\) is even, the chosen compatible normalization gives an exact sequence
\begin{equation}\label{eq:coker-dn1-snake}
\begin{aligned}
\ker(\overline\vartheta_n)
\xrightarrow{\ \partial_{\mathrm{sn}}\ }
\mathbf K^M_4/24
\longrightarrow
\mathrm{coker}(d_{n,1})_{-(n-2)}
\longrightarrow
\ker\!\bigl(\mathbf K^Q_3\longrightarrow\mathbf K^{MW}_3\bigr)
\longrightarrow0.
\end{aligned}
\end{equation}
Equivalently, if
\[
\mathbf C_n:=
\mathrm{coker}\!\left(
\ker(\overline\vartheta_n)\xrightarrow{\partial_{\mathrm{sn}}}
\mathbf K^M_4/24\right),
\]
then there is a short exact sequence
\[
0\longrightarrow\mathbf C_n
\longrightarrow\mathrm{coker}(d_{n,1})_{-(n-2)}
\longrightarrow
\ker\!\bigl(\mathbf K^Q_3\longrightarrow\mathbf K^{MW}_3\bigr)
\longrightarrow0.
\]
If \(n\) is odd, then \(d_{n,1}=0\), and there is instead a canonical short
exact sequence
\begin{equation}\label{eq:coker-dn1-odd}
0\longrightarrow\mathbf K^M_4/24
\longrightarrow\mathrm{coker}(d_{n,1})_{-(n-2)}
\longrightarrow\mathbf{GW}^2_3
\longrightarrow0.
\end{equation}
\end{theorem}

\begin{proof}
After $(n-2)$ contractions, the two rows of
\eqref{eq:ABH-coker-diagram-contracted} are short exact.  Apply the snake
lemma first when \(n\) is even.  Since the left vertical morphism is zero, the
relevant portion of the resulting exact sequence is
\[
\ker(\overline\vartheta_n)
\xrightarrow{\partial_{\mathrm{sn}}}
\mathbf K^M_4/24
\longrightarrow
\mathrm{coker}(\overline d_{n,1})
\longrightarrow
\mathrm{coker}(\overline\vartheta_n)
\longrightarrow0.
\]
Exactness of contraction identifies
\(\mathrm{coker}(\overline d_{n,1})\) with
\(\mathrm{coker}(d_{n,1})_{-(n-2)}\).  Finally, the long exact sequence in
Grothendieck--Witt theory associated with the Wood cofiber sequence contains
\[
\mathbf{GW}^3_4\xrightarrow{\ \overline\vartheta_n\ }\mathbf{GW}^2_3
\longrightarrow\mathbf K^Q_3
\longrightarrow\mathbf{GW}^3_3\cong\mathbf K^{MW}_3,
\]
and hence
\[
\mathrm{coker}(\overline\vartheta_n)
\cong\ker\!\bigl(\mathbf K^Q_3\to\mathbf K^{MW}_3\bigr).
\]
This proves \eqref{eq:coker-dn1-snake}; the short exact reformulation follows
by taking the cokernel of the connecting morphism.

If \(n\) is odd, Lemma~\ref{lem:stiefel-parity} gives \(d_{n,1}=0\).
Consequently its cokernel is the target sheaf.  The lower short exact row of
\eqref{eq:ABH-coker-diagram-contracted} is precisely
\eqref{eq:coker-dn1-odd}.
\end{proof}

The next task is to understand the second half of the diagram \eqref{eq:left-half-cokers}, i.e.\
the image of the maps $(\mathrm{inc}_n)_*$.
By \cite{AF14}, for $m\ge 1$ there are commutative diagrams with exact rows and columns
\begin{equation}\label{eq:AF-even}
\begin{tikzcd}
& & 0 \arrow[d] & 0 \arrow[d] \\
& & \mathbf{T}_{2m+1} \arrow[d] & \mathbf{I}^{2m+1} \arrow[d] \\
& \pi_{2m-1}^{\A^1}(SL_{2m-1}) \arrow[r, "(\mathrm{inc}_{2m-1})_*"]
& \pi_{2m-1}^{\A^1}(SL_{2m}) \arrow[r, "q_{2m-1}"] \arrow[d]
& \mathbf{K}^{MW}_{2m} \arrow[d] \\
& & \mathbf{K}^Q_{2m} \arrow[r, "\psi_{2m}"] \arrow[d]
& \mathbf{K}^M_{2m} \arrow[d] \\
& & 0 & 0
\end{tikzcd}
\end{equation}
and
\begin{equation}\label{eq:AF-odd}
\begin{tikzcd}
& & 0 \arrow[d] & 0 \arrow[d] \\
& & \mathbf{S}_{2m+2} \arrow[d] & \mathbf{I}^{2m+2} \arrow[d] \\
& \pi_{2m}^{\A^1}(SL_{2m}) \arrow[r, "(\mathrm{inc}_{2m})_*"]
& \pi_{2m}^{\A^1}(SL_{2m+1}) \arrow[r, "q_{2m}"] \arrow[d]
& \mathbf{K}^{MW}_{2m+1} \arrow[d] \\
& & \mathbf{K}^Q_{2m+1} \arrow[r, "\psi_{2m+1}"] \arrow[d]
& \mathbf{K}^M_{2m+1} \arrow[d] \\
& & 0 & 0
\end{tikzcd}
\end{equation}
Taking kernels in the right-hand squares of \eqref{eq:AF-even} and \eqref{eq:AF-odd}, and using
exactness, we obtain short exact sequences
\begin{equation}\label{eq:im-inc-even}
0 \longrightarrow \ker(\mathbf{T}_{2m+1}\to \mathbf{I}^{2m+1})
\longrightarrow \mathrm{im}\big((\mathrm{inc}_{2m-1})_*\big)
\longrightarrow \ker(\psi_{2m})
\longrightarrow 0,
\end{equation}
and
\begin{equation}\label{eq:im-inc-odd}
0 \longrightarrow \ker(\mathbf{S}_{2m+2}\to \mathbf{I}^{2m+2})
\longrightarrow \mathrm{im}\big((\mathrm{inc}_{2m})_*\big)
\longrightarrow \ker(\psi_{2m+1})
\longrightarrow 0.
\end{equation}

Since $\mathbf{T}_{2m+1}$ is defined as the fiber product
\[
\mathbf{T}_{2m+1}:=\mathbf{S}_{2m+1}\times_{\mathbf{K}^M_{2m+1}/2}\mathbf{I}^{2m+1},
\]
the projection $\mathbf{T}_{2m+1}\to \mathbf{I}^{2m+1}$ has kernel canonically identified with
\[
\ker(\mathbf{T}_{2m+1}\to \mathbf{I}^{2m+1})
\ \cong\
\ker(\mathbf{S}_{2m+1}\to \mathbf{K}^M_{2m+1}/2),
\]
via $s\mapsto (s,0)$. Hence understanding the left term in \eqref{eq:im-inc-even} reduces to
understanding
\[
\ker(\mathbf{S}_{2m+1}\to \mathbf{K}^M_{2m+1}/2),
\]
which is analyzed in
\cite[\S3.5]{AF14}; we refer the reader there for further details.

\begin{remark}\label{rem:suslin-hurewicz-limitation}
The short exact sequences \eqref{eq:im-inc-even} and
\eqref{eq:im-inc-odd} are unconditional.  Their further simplification is
constrained by the image of
\[
\psi_r:\mathbf K^Q_r\longrightarrow\mathbf K^M_r.
\]
The composite
\[
\mathbf K^M_r\longrightarrow\mathbf K^Q_r
\xrightarrow{\ \psi_r\ }\mathbf K^M_r
\]
is multiplication by $(r-1)!$.  Consequently
\[
(r-1)!\,\mathbf K^M_r\subseteq\operatorname{im}(\psi_r),
\]
which gives the canonical epimorphism
\[
\mathbf K^M_r/(r-1)!\twoheadrightarrow\mathbf S_r
\]
already used above.  The reverse inclusion
\(\operatorname{im}(\psi_r)\subseteq(r-1)!\mathbf K^M_r\) is not a formal
consequence of the composite calculation: it is the sheaf-theoretic form of
the Suslin--Hurewicz image conjecture; see \cite{AFW20}.  It is known in
degree \(3\), in degree \(4\) under the hypotheses of \cite{Rondigs23}, and
in degree \(5\) under the hypotheses of \cite{AFW20}.  We therefore retain
the kernels involving \(\mathbf S_r\), \(\mathbf T_r\), and \(\psi_r\) in
\eqref{eq:im-inc-even} and \eqref{eq:im-inc-odd}, and make no unconditional
factorization claim for arbitrary \(r\).
\end{remark}

\section{A contracted third non-stable calculation}\label{sec:third-unstable}

In this section we carry the construction of Section~3 one degree further.
The full uncontracted calculation is governed by the $2$-line of the motivic
sphere and is correspondingly more intricate.  We first isolate the formal
cross-composite that compares two adjacent long exact sequences, then record the known
description of its punctured affine-space terms, and finally give an explicit
calculation after sufficiently many contractions.

\subsection{The next differential}\label{subsec:dn2}

Recall the diagram of fiber sequences used in
\eqref{eq:diagram-pi-n}.  Applying $\pi_{n+1}^{\A^1}$ instead of
$\pi_n^{\A^1}$ gives a composite
\begin{equation}\label{eq:def-dn2}
\begin{aligned}
d_{n,2}:\
\pi_{n+2}^{\A^1}(\A^{n+1}\setminus\{0\})
&\xrightarrow{\ \delta_{n,2}\ }
\pi_{n+1}^{\A^1}(SL_n)\\
&\xrightarrow{\ q_{n-1,2}\ }
\pi_{n+1}^{\A^1}(\A^n\setminus\{0\}).
\end{aligned}
\end{equation}
Thus $d_{n,2}:=q_{n-1,2}\circ\delta_{n,2}$ is the next differential
associated with the special linear tower.  To simplify notation, set
\[
\mathcal P_n:=
\pi_{n+1}^{\A^1}(\A^n\setminus\{0\}).
\]
With this notation $d_{n,2}$ is a morphism
\[
d_{n,2}:\mathcal P_{n+1}\longrightarrow\mathcal P_n.
\]

The two arrows in \eqref{eq:def-dn2} occur in different long exact
sequences.  We record those sequences separately.

\begin{proposition}\label{prop:third-unstable-two-rows}
For every $n\geq3$, the fiber sequences for two adjacent stages of the
special linear tower give exact sequences of strongly $\A^1$-invariant
sheaves
\begin{equation}\label{eq:third-unstable-upper-row}
\begin{aligned}
\mathcal P_{n+1}
\xrightarrow{\ \delta_{n,2}\ }
\pi_{n+1}^{\A^1}(SL_n)
\xrightarrow{\ (\mathrm{inc}_n)_*\ }
\pi_{n+1}^{\A^1}(SL_{n+1})
\longrightarrow
\pi_{n+1}^{\A^1}(\A^{n+1}\setminus\{0\}),
\end{aligned}
\end{equation}
and
\begin{equation}\label{eq:third-unstable-lower-row}
\begin{aligned}
\pi_{n+1}^{\A^1}(SL_{n-1})
\xrightarrow{\ (\mathrm{inc}_{n-1})_*\ }
\pi_{n+1}^{\A^1}(SL_n)
\xrightarrow{\ q_{n-1,2}\ }
\mathcal P_n
\xrightarrow{\ \partial_{n,2}\ }
\pi_n^{\A^1}(SL_{n-1})
\xrightarrow{\ (\mathrm{inc}_{n-1})_*\ }
\pi_n^{\A^1}(SL_n).
\end{aligned}
\end{equation}
The morphism \(d_{n,2}\) is the cross-composite
\[
\mathcal P_{n+1}\xrightarrow{\delta_{n,2}}
\pi_{n+1}^{\A^1}(SL_n)\xrightarrow{q_{n-1,2}}\mathcal P_n.
\]
In particular, \eqref{eq:third-unstable-upper-row} and
\eqref{eq:third-unstable-lower-row} are not to be concatenated at their
common middle term.
\end{proposition}

\begin{proof}
Sequence \eqref{eq:third-unstable-upper-row} is the indicated segment of the
long exact sequence of
\[
SL_n\longrightarrow SL_{n+1}\longrightarrow\A^{n+1}\setminus\{0\},
\]
and \eqref{eq:third-unstable-lower-row} is the corresponding segment for
\[
SL_{n-1}\longrightarrow SL_n\longrightarrow\A^n\setminus\{0\}.
\]
The last assertion is the definition \eqref{eq:def-dn2}.
\end{proof}

\begin{proposition}\label{prop:dn2-odd-zero}
If \(n\geq3\) is odd, then
\[
d_{n,2}:\mathcal P_{n+1}\longrightarrow\mathcal P_n
\]
is the zero morphism.  More generally, every cross-composite obtained from
the same two adjacent stages of the special-linear tower is zero.
\end{proposition}

\begin{proof}
As in Lemma~\ref{lem:stiefel-parity}, the cross-composite is a connecting
homomorphism for
\[
\A^n\setminus0\longrightarrow
\operatorname{St}_2(n+1)\longrightarrow
\A^{n+1}\setminus0.
\]
For odd \(n\), the projection admits the symplectic \(\A^1\)-homotopy
section used in the proof of \cite[Lemma~18]{AOSS26}.  Hence all connecting
homomorphisms, including the one on \(\pi_{n+2}^{\A^1}\), vanish.
\end{proof}

\begin{proposition}\label{prop:dn2-KO-compatibility}
Let \(k\) be a field of characteristic \(0\), and use the compatible
normalization of Corollary~\ref{prop:suslin-stiefel-wood}.  For \(m\geq2\),
write
\[
\deg^{KO}_{m,2}:=
\lambda_m\pi_{m+1}^{\A^1}(\Psi_m^u):
\mathcal P_m\longrightarrow\mathbf{GW}^{m}_{m+2}.
\]
Then, for every \(n\geq2\),
\begin{equation}\label{eq:dn2-KO-compatibility}
\deg^{KO}_{n,2}\circ d_{n,2}
=
\vartheta_n\circ\deg^{KO}_{n+1,2},
\qquad
\vartheta_n=
\begin{cases}
\eta:\mathbf{GW}^{n+1}_{n+3}\longrightarrow\mathbf{GW}^{n}_{n+2},
&n\text{ even},\\
0,&n\text{ odd}.
\end{cases}
\end{equation}
\end{proposition}

\begin{proof}
Before applying homotopy sheaves, the proof of
Proposition~\ref{prop:suslin-stiefel-wood-up-to-unit} establishes a pointed
homotopy
\[
a_n\Psi_n^u q_{n-1}\delta_{n,1}
\simeq
b_n\Omega_s(\eta\Psi_{n+1}^u)
:
\Omega_sX_{n+1}\longrightarrow\mathcal K_n.
\]
Applying \(\pi_{n+1}^{\A^1}\) instead of \(\pi_n^{\A^1}\) identifies the
map induced by \(q_{n-1}\delta_{n,1}\) with \(d_{n,2}\), and gives
\[
\pi_{n+1}^{\A^1}(\mathcal K_n)=\mathbf{GW}^{n}_{n+2},
\qquad
\pi_{n+2}^{\A^1}(\mathcal K_{n+1})
=\mathbf{GW}^{n+1}_{n+3}.
\]
Thus, for even \(n\),
\[
\pi_{n+1}^{\A^1}(\Psi_n^u)\,d_{n,2}
=c_n\eta\,\pi_{n+2}^{\A^1}(\Psi_{n+1}^u),
\qquad c_n=a_n^{-1}b_n.
\]
The relation \(\lambda_{n+1}=\lambda_nc_n\) used in
Corollary~\ref{prop:suslin-stiefel-wood} yields
\eqref{eq:dn2-KO-compatibility}.  For odd \(n\), both sides vanish by
Proposition~\ref{prop:dn2-odd-zero}.
\end{proof}

\subsection{Input from the second stable stem}\label{subsec:second-stem-input}

We now recall the available description of $\mathcal P_n$.  Write
\[
\mathbf H^{a,b}/m
\]
for the Nisnevich sheaf associated with
$U\mapsto H^{a,b}(U,\mathbb Z)/m$.  We write $\mathbf{KQ}$ for the periodic
hermitian \(K\)-theory spectrum represented by the geometric
\(\mathbf{KO}\)-model used above, and $\mathbf{kq}$ for its very effective
cover.  To make the contraction indices transparent, we use the convention
\[
\boldsymbol\pi^s_{p+q\alpha}(E):=\boldsymbol\pi^s_{p+q,q}(E).
\]
Thus one contraction raises the coefficient of $\alpha$ by one:
\begin{equation}\label{eq:stable-contraction-alpha}
\big(\boldsymbol\pi^s_{p+q\alpha}(E)\big)_{-1}
\cong \boldsymbol\pi^s_{p+(q+1)\alpha}(E).
\end{equation}

\begin{theorem}\label{thm:Pn-second-stem}
Let $k$ be a field of characteristic $0$ and let $n\geq 5$.  Stabilization
induces a canonical isomorphism
\begin{equation}\label{eq:Pn-stable-identification}
\mathcal P_n\xrightarrow{\ \sim\ }
\boldsymbol\pi^s_{2-n\alpha}(\mathds 1).
\end{equation}
Under this identification, the stable unit
$\mathds 1\to\mathbf{kq}$ gives an exact sequence
\begin{equation}\label{eq:Pn-second-stem}
0\longrightarrow
\mathbf H^{n+1,n+2}/24\ \oplus\ \mathbf K^M_{n+4}/2
\longrightarrow
\mathcal P_n
\longrightarrow
\boldsymbol\pi^s_{2-n\alpha}(\mathbf{kq}).
\end{equation}
After further contraction there are canonical isomorphisms
\begin{equation}\label{eq:Pn-high-contractions}
\begin{aligned}
(\mathcal P_n)_{-(n+2)}&\cong \mathbf K^M_2/2,\\
(\mathcal P_n)_{-(n+3)}&\cong \mathbf K^M_1/2,\\
(\mathcal P_n)_{-(n+4)}&\cong \mathbf K^M_0/2.
\end{aligned}
\end{equation}
\end{theorem}

\begin{proof}
Write $\A^n\setminus0\simeq S^{2n-1,n}$.  Applying
\cite[Theorem~6.3.3]{ABH} directly to the stable unit
\[
S^{2n-1,n}\longrightarrow Q S^{2n-1,n}
\]
shows that its fiber is weakly $S^{4n-3,2n}$-cellular.  By
\cite[Corollary~3.1.27]{ABH}, it is therefore $(2n-4)$-$\A^1$-connected.
Since $n+1\leq2n-4$ for $n\geq5$, the unit induces an isomorphism on
$\pi_{n+1}^{\A^1}$.  Since $Q$ is the infinite
$\mathbb P^1$-loop--suspension stabilization, this proves
\eqref{eq:Pn-stable-identification}.  In particular, the comparison is
natural in the base field and identifies the map to $\mathbf{kq}$ with the
map induced by the stable unit.

Now combine \eqref{eq:Pn-stable-identification} with the calculation of the
$2$-line of the motivic sphere in \cite[Theorem~1.2]{RSO21}.  This gives
\eqref{eq:Pn-second-stem}; the resulting sequence is also recorded in
\cite[Example~6.6]{Gant_Williams_2025}.  We explain separately what happens
at the first of the three displayed contractions, since it is a boundary case
in the notation of \cite[Theorem~2.38]{RSO21}.

After $(n+2)$ contractions the motivic-cohomology summand vanishes and the
Milnor $K$-theory summand becomes
\[
\big(\mathbf H^{n+1,n+2}/24\big)_{-(n+2)}
  =\mathbf H^{-1,0}/24=0,
\qquad
\big(\mathbf K^M_{n+4}/2\big)_{-(n+2)}
  =\mathbf K^M_2/2.
\]
On the right, \eqref{eq:stable-contraction-alpha} gives
\[
\big(\boldsymbol\pi^s_{2-n\alpha}(\mathbf{kq})\big)_{-(n+2)}
\cong \boldsymbol\pi^s_{2+2\alpha}(\mathbf{kq})
=\boldsymbol\pi^s_{4,2}(\mathbf{kq}).
\]
This sheaf does \emph{not} vanish.  In the notation
$\pi_{2-(w)}=\pi_{2-w,-w}$ of \cite{RSO21}, it is the case $w=-2$, and
\cite[Theorem~2.38]{RSO21} identifies it with
$\mathbf K^M_0\cong\mathbb Z$ (equivalently, with
$\mathbf K^{Sp}_0\cong2\mathbb Z$ under the rank map).  What vanishes is the
map to this torsion-free sheaf.  Indeed, after stabilization it is induced by
the unit $\mathds 1\to\mathbf{kq}$ on
\[
\boldsymbol\pi^s_{4,2}(\mathds 1)\longrightarrow
\boldsymbol\pi^s_{4,2}(\mathbf{kq}).
\]
The source is torsion whereas the target is torsion-free, so this map is zero;
this is also pointed out immediately after \cite[Theorem~1.2]{RSO21}.
Exactness of \eqref{eq:Pn-second-stem} therefore gives
$(\mathcal P_n)_{-(n+2)}\cong\mathbf K^M_2/2$.

After $(n+3)$ and $(n+4)$ contractions, the right-hand terms are
\[
\boldsymbol\pi^s_{2+3\alpha}(\mathbf{kq})
=\pi_{2-(-3)}(\mathbf{kq}),
\qquad
\boldsymbol\pi^s_{2+4\alpha}(\mathbf{kq})
=\pi_{2-(-4)}(\mathbf{kq}).
\]
These do vanish by \cite[Theorem~2.38]{RSO21}, whose vanishing range is the
strict inequality $w<-2$.  The motivic-cohomology summands again vanish, and
the Milnor $K$-theory summands contract to $\mathbf K^M_1/2$ and
$\mathbf K^M_0/2$, respectively.  This proves all three isomorphisms in
\eqref{eq:Pn-high-contractions}.
\end{proof}

\begin{remark}\label{rem:uncontracted-kq-terms}
The point of \eqref{eq:Pn-second-stem} is not that
$\boldsymbol\pi^s_{2-n\alpha}(\mathbf{kq})$ is an unknown black box.
R{\"o}ndigs--Spitzweck--{\O}stv{\ae}r describe the $2$-line by motivic cohomology,
Milnor $K$-theory, powers of the class $\rho=\{-1\}$, and hermitian
$K$-theory.  The boundary value $w=-2$ deserves particular care: the
$\mathbf{kq}$-term is nonzero there, and the first isomorphism in
\eqref{eq:Pn-high-contractions} uses the vanishing of the unit-induced map,
not the vanishing of its target.  The target itself vanishes only at the next
two weights $w=-3,-4$.  Retaining the compact $\mathbf{kq}$-notation makes
clear which part of the answer is stable and which part comes from the
non-stable-to-stable comparison.
\end{remark}

\subsection{\texorpdfstring{The contracted calculation of \(d_{n,2}\)}
{The contracted calculation of the next differential}}

The identifications in Theorem~\ref{thm:Pn-second-stem} turn two contractions
of $d_{n,2}$ into morphisms between adjacent Milnor $K$-theory sheaves:
\[
(d_{n,2})_{-(n+3)}:\mathbf K^M_2/2\longrightarrow\mathbf K^M_1/2
\]
and
\[
(d_{n,2})_{-(n+4)}:\mathbf K^M_1/2\longrightarrow\mathbf K^M_0/2.
\]
Both morphisms vanish for a formal reason.

\begin{lemma}\label{lem:KM-degree-drop-zero}
Let $a>b\geq 0$ and $m\geq 1$.  Every morphism of strictly
$\A^1$-invariant sheaves
\[
\mathbf K^M_a/m\longrightarrow\mathbf K^M_b/m
\]
is zero.
\end{lemma}

\begin{proof}
Precomposition with the canonical epimorphism
$\mathbf K^{MW}_a\twoheadrightarrow\mathbf K^M_a/m$ injects the group of such
morphisms into
\[
\mathrm{Hom}(\mathbf K^{MW}_a,\mathbf K^M_b/m).
\]
By \cite[Lemma~5.1.3]{AWW17}, the latter group is
\[
(\mathbf K^M_b/m)_{-a}(k)
\cong \mathbf K^M_{b-a}(k)/m=0,
\]
where the last equality uses $b-a<0$.
\end{proof}

\begin{theorem}\label{thm:dn2-contracted-zero}
Let $k$ be a field of characteristic $0$ and let $n\geq 5$.  Under the
canonical identifications of Theorem~\ref{thm:Pn-second-stem}, one has
\[
(d_{n,2})_{-(n+3)}=0
\qquad\text{and}\qquad
(d_{n,2})_{-(n+4)}=0.
\]
Consequently,
\begin{equation}\label{eq:coker-dn2-contracted}
\mathrm{coker}(d_{n,2})_{-(n+3)}
\cong \mathbf K^M_1/2,
\qquad
\mathrm{coker}(d_{n,2})_{-(n+4)}
\cong \mathbf K^M_0/2.
\end{equation}
\end{theorem}

\begin{proof}
Apply Theorem~\ref{thm:Pn-second-stem} to $\mathcal P_{n+1}$ and
$\mathcal P_n$.  After $(n+3)$ contractions, the map $d_{n,2}$ becomes
\[
\mathbf K^M_2/2\longrightarrow\mathbf K^M_1/2,
\]
and after $(n+4)$ contractions it becomes
\[
\mathbf K^M_1/2\longrightarrow\mathbf K^M_0/2.
\]
Both are zero by Lemma~\ref{lem:KM-degree-drop-zero}.  Exactness of contraction
(Lemma~\ref{lem:contraction-exact}) shows that contraction commutes with the
displayed cokernels, which proves \eqref{eq:coker-dn2-contracted}.
\end{proof}

Before invoking the periodic compatibility, the preceding vanishing can be
sharpened by one contraction directly from the sphere \(2\)-line.  Write
\[
\mathcal E_{n+1}:=(\mathcal P_{n+1})_{-(n+2)}
\]
and let
\[
\mathcal U_{n+1}:=
\operatorname{im}\!\left(
\mathcal E_{n+1}\longrightarrow
\boldsymbol\pi^s_{3,1}(\mathbf{kq})\cong\G_m
\right).
\]

\begin{theorem}\label{thm:dn2-boundary-defect}
Let \(k\) be a field of characteristic \(0\) and let \(n\geq5\).  One has
\(\mathcal U_{n+1}=\boldsymbol\mu_{24}\), and there is a natural short exact
sequence
\begin{equation}\label{eq:source-boundary-extension}
0\longrightarrow\mathbf K^M_3/2
\longrightarrow\mathcal E_{n+1}
\longrightarrow\boldsymbol\mu_{24}\longrightarrow0.
\end{equation}
Under the canonical identification
\((\mathcal P_n)_{-(n+2)}\cong\mathbf K^M_2/2\), the morphism
\((d_{n,2})_{-(n+2)}\) vanishes on the subobject
\(\mathbf K^M_3/2\).  It is therefore determined by a unique morphism
\begin{equation}\label{eq:epsilon-n-defect}
\epsilon_n:\boldsymbol\mu_{24}\longrightarrow\mathbf K^M_2/2.
\end{equation}
The morphism \(\epsilon_n\) factors through
\(\boldsymbol\mu_{24}/2\boldsymbol\mu_{24}\), and it is zero when \(n\) is
odd.
\end{theorem}

\begin{proof}
Apply the stable identification \eqref{eq:Pn-stable-identification} with
\(n+1\) in place of \(n\) and contract \(n+2\) times.  By
\eqref{eq:stable-contraction-alpha}, this gives a natural identification
\[
\mathcal E_{n+1}\cong\boldsymbol\pi^s_{3,1}(\mathds1),
\]
and the map defining $\mathcal U_{n+1}$ is identified with the stable unit
\[
\boldsymbol\pi^s_{3,1}(\mathds1)\longrightarrow
\boldsymbol\pi^s_{3,1}(\mathbf{kq})\cong\G_m.
\]
The case $n=1$ of \cite[Theorem~1.2]{RSO21} is the short exact sequence
\[
0\longrightarrow\mathbf K^M_3/2
\longrightarrow\boldsymbol\pi^s_{3,1}(\mathds1)
\longrightarrow\boldsymbol\mu_{24}\longrightarrow0,
\]
where the last map is the stable unit followed by restriction to its image.
This proves both $\mathcal U_{n+1}=\boldsymbol\mu_{24}$ and
\eqref{eq:source-boundary-extension}.

The restriction of \(\left(d_{n,2}\right)_{-(n+2)}\) to the kernel in
\eqref{eq:source-boundary-extension} is a morphism
\[
\mathbf K^M_3/2\longrightarrow\mathbf K^M_2/2,
\]
which is zero by Lemma~\ref{lem:KM-degree-drop-zero}.  The universal property
of the cokernel produces the unique morphism \eqref{eq:epsilon-n-defect}.
Since its target is annihilated by \(2\), it factors through
\(\boldsymbol\mu_{24}/2\boldsymbol\mu_{24}\).  Finally,
Proposition~\ref{prop:dn2-odd-zero} gives \(\epsilon_n=0\) for odd \(n\).
\end{proof}

\begin{proposition}\label{prop:dn2-stable-eta}
Let \(k\) be a field of characteristic \(0\) and let \(n\geq5\).  Relative
to the standard sphere and purity identifications used in
\cite[Lemma~3.5]{AF14},
\[
d_{n,2}:
\boldsymbol\pi^s_{2-(n+1)\alpha}(\mathds1)
\longrightarrow
\boldsymbol\pi^s_{2-n\alpha}(\mathds1)
\]
is induced by
\[
\begin{cases}
\eta,&n\text{ even},\\
0,&n\text{ odd}.
\end{cases}
\]
Changing those identifications replaces the even-stage formula by
\(u_n\eta\) for a unit \(u_n\in GW(k)^\times\).  In either convention, if
\(n\) is even, then
\[
D_n:=(d_{n,2})_{-(n+2)}:
\boldsymbol\pi^s_{3,1}(\mathds1)
\longrightarrow
\boldsymbol\pi^s_{4,2}(\mathds1)
\]
is multiplication by \(\eta\) under the canonical identifications
\eqref{eq:Pn-high-contractions}.
\end{proposition}

\begin{proof}
Put \(X_m=\mathbb A^m\setminus0\) and choose the standard sphere
identifications
\[
X_{n+1}\simeq
\Sigma_sY_n,
\qquad
Y_n:=S_s^{n-1}\wedge\mathbb G_m^{\wedge(n+1)},
\qquad
X_n\simeq S_s^{n-1}\wedge\mathbb G_m^{\wedge n}.
\]
Let
\[
E_n:Y_n\longrightarrow\Omega_s\Sigma_sY_n\simeq\Omega_sX_{n+1}
\]
be the loop--suspension unit.  The space \(Y_n\) is
\(\mathbb A^1\)-\((n-2)\)-connected.  Morel's simplicial suspension theorem
\cite[Theorem~6.61]{MorelA1} therefore shows that
\[
(E_n)_*:\pi_i^{\mathbb A^1}(Y_n)
\xrightarrow{\ \cong\ }
\pi_i^{\mathbb A^1}(\Omega_sX_{n+1})
\]
for \(i\leq2n-4\).  In particular this is an isomorphism for
\(i=n-1,n,n+1\) when \(n\geq5\).

Let
\[
\widetilde d_n:\Omega_sX_{n+1}\longrightarrow X_n
\]
be the connecting morphism of the two-frame Stiefel fiber sequence, followed
by the projection to \(X_n\), and set
\[
\alpha_n:=\widetilde d_n\circ E_n:Y_n\longrightarrow X_n.
\]
The map induced by \(\alpha_n\) on the first nonzero homotopy sheaf is the
bottom Stiefel boundary
\[
\mathbf K^{MW}_{n+1}
\longrightarrow
\mathbf K^{MW}_{n}.
\]
Indeed, \((E_n)_*\) is the suspension isomorphism in that degree.  The
calculation of \cite[Lemma~3.5]{AF14} identifies this morphism with
multiplication by \(\eta\) for even \(n\) and with zero for odd \(n\).

This bottom calculation identifies the map \(\alpha_n\) itself.  Namely,
evaluation on the fundamental class of \(Y_n\), followed by Morel's
computation of the first nonzero homotopy sheaf, gives the canonical
bijection
\[
[Y_n,X_n]_{\mathbb A^1,*}
\cong
\big(\mathbf K_n^{MW}\big)_{-(n+1)}(k)
\cong
K_{-1}^{MW}(k).
\]
The first identification is the definition of bigraded homotopy together
with \(\pi_{n-1}^{\mathbb A^1}(X_n)=\mathbf K_n^{MW}\), and the second is
the contraction formula of \cite[Proposition~2.9]{AF14}.  On the other hand,
\cite[Lemma~5.1.3]{AWW17} gives
\[
 \operatorname{Hom}
 (\mathbf K_{n+1}^{MW},\mathbf K_n^{MW})
 \cong(\mathbf K_n^{MW})_{-(n+1)}(k)
 \cong K_{-1}^{MW}(k).
\]
Under these two identifications, the passage from a pointed map to its
morphism on the bottom homotopy sheaf is the identity on
\(K_{-1}^{MW}(k)\).  In particular it is injective, and the corresponding
class acts by multiplication on the bottom Milnor--Witt sheaf.
Therefore the calculation of the bottom Stiefel boundary gives the equality
of pointed maps
\[
\alpha_n=
\begin{cases}
\eta,&n\text{ even},\\
0,&n\text{ odd}.
\end{cases}
\]

Finally, the isomorphism \((E_n)_*\) for \(i=n+1\) identifies the map
induced by \(\widetilde d_n\) on
\[
\pi_{n+1}^{\mathbb A^1}(\Omega_sX_{n+1})
=\pi_{n+2}^{\mathbb A^1}(X_{n+1})
\]
with the map induced by \(\alpha_n\) on
\(\pi_{n+1}^{\mathbb A^1}(Y_n)\).  Stabilization is natural for the actual
map \(\alpha_n\), and Theorem~\ref{thm:Pn-second-stem} identifies the resulting
source and target with the two displayed stable stems.  Thus \(d_{n,2}\)
is induced by \(\eta\) in even stages and by zero in odd stages with the
standard sphere and purity coordinates used in \cite[Lemma~3.5]{AF14}.
Changing a sphere coordinate multiplies the even-stage class by some
\(u_n\in K^{MW}_0(k)^\times=GW(k)^\times\).  Its action on
\(\mathbf K^M_2/2\) is nevertheless the identity: the action factors through
\[
K^{MW}_0(k)\longrightarrow K^M_0(k)/2=\mathbb Z/2,
\]
and a unit of \(GW(k)\) has rank \(\pm1\), hence rank one modulo \(2\).
Consequently the \((n+2)\)-fold contraction is canonically multiplication by
\(\eta\).  This argument uses the simplicial desuspension and the
Freudenthal isomorphism;
it does not require suspension spectra to commute with non-stable loop
spaces.
\end{proof}

\begin{proposition}\label{prop:oriented-unit-wood}
Let \(k\) be a field of characteristic \(0\).  Fix sphere and purity
identifications
\[
X_m\simeq S_s^{m-1}\wedge\mathbb G_m^{\wedge m}
\]
with the Euler/Gysin orientation used in \cite[Lemma~3.5]{AF14}; thus the
adjoint bottom Stiefel boundary is multiplication by \(\eta\) in even stages
and zero in odd stages.  Also fix a coherent system of geometric Bott
equivalences for \(\mathbf{KO}\).  Let
\[
\Psi_m^{\mathrm{or}}:X_m\longrightarrow\mathcal K_m
\]
be the class which, under the spectrum--infinite-loop adjunction, corresponds
to the stable unit \(u:\mathds1\to\mathbf{KO}\) relative to these coordinates.
For
\(j=0,1\), write
\[
(\delta^{\mathrm{or}}_{m,j})_*:
\pi_{m+j}^{\mathbb A^1}(X_m)
\longrightarrow
\pi_{m+j}^{\mathbb A^1}(\mathcal K_m)
=\mathbf{GW}^{m}_{m+j+1}
\]
for the induced morphism.  Define
\[
\vartheta_{n,j}:
\mathbf{GW}^{n+1}_{n+j+2}
\longrightarrow
\mathbf{GW}^{n}_{n+j+1}
\]
to be multiplication by \(\eta\) when \(n\) is even and the zero morphism
when \(n\) is odd.  Then
\begin{equation}\label{eq:oriented-unit-wood}
(\delta^{\mathrm{or}}_{n,j})_*\circ d_{n,j+1}
=
\vartheta_{n,j}\circ(\delta^{\mathrm{or}}_{n+1,j})_*.
\end{equation}
For \(j=0\), this holds for \(n\geq4\); for \(j=1\), it holds for
\(n\geq5\).
\end{proposition}

\begin{proof}
Retain the notation from the proof of
Proposition~\ref{prop:dn2-stable-eta}.  Thus
\[
X_{n+1}\simeq\Sigma_sY_n,\qquad
E_n:Y_n\longrightarrow\Omega_sX_{n+1},
\qquad
\alpha_n=\widetilde d_nE_n:Y_n\longrightarrow X_n.
\]
The bottom Stiefel calculation of \cite[Lemma~3.5]{AF14}, with the orientation
fixed in the statement, together with
\[
[Y_n,X_n]_{\mathbb A^1,*}\cong K_{-1}^{MW}(k),
\]
identifies the actual map \(\alpha_n\), in the fixed coordinates, as
\[
\alpha_n=
\begin{cases}
\eta,&n\text{ even},\\
0,&n\text{ odd}.
\end{cases}
\]
This part of the argument only uses the bottom homotopy sheaf and is valid
for \(n\geq4\).

Let
\[
\operatorname{ad}_s\!
\left(\eta_{\mathbf{KO}}\Psi_{n+1}^{\mathrm{or}}\right):
Y_n\longrightarrow\mathcal K_n
\]
denote the simplicial adjoint, where the target is identified using the
fixed Bott equivalence.  There is an equality
\begin{equation}\label{eq:oriented-unit-map-level}
\Psi_n^{\mathrm{or}}\alpha_n
\simeq
\begin{cases}
\operatorname{ad}_s
  \left(\eta_{\mathbf{KO}}\Psi_{n+1}^{\mathrm{or}}\right),
  &n\text{ even},\\
0,&n\text{ odd},
\end{cases}
\end{equation}
in the pointed \(\mathbb A^1\)-homotopy category.  Indeed, after the
spectrum--infinite-loop adjunction, the two maps in the even case are
respectively the composites
\[
u\circ\eta
\qquad\text{and}\qquad
\eta_{\mathbf{KO}}\circ u.
\]
They agree because \(u\) is a morphism of \(\mathds1\)-modules.  This is an
equality for the unit-adjoint classes themselves; no comparison of a Thom
generator with a Suslin-matrix generator enters it.

Morel's simplicial suspension theorem gives
\[
(E_n)_*:\pi_i^{\mathbb A^1}(Y_n)
\xrightarrow{\ \cong\ }
\pi_i^{\mathbb A^1}(\Omega_sX_{n+1})
\]
for \(i\leq2n-4\).  Take \(i=n+j\).  This is allowed for \(j=0\) when
\(n\geq4\), and for \(j=1\) when \(n\geq5\).  Under this isomorphism,
\((\alpha_n)_*\) is \(d_{n,j+1}\), while the right-hand side of
\eqref{eq:oriented-unit-map-level} induces
\(\vartheta_{n,j}(\delta^{\mathrm{or}}_{n+1,j})_*\).
The left-hand side induces
\((\delta^{\mathrm{or}}_{n,j})_*d_{n,j+1}\).  This proves
\eqref{eq:oriented-unit-wood}.
\end{proof}

\begin{remark}\label{rem:oriented-unit-status}
The orientation clause in Proposition~\ref{prop:oriented-unit-wood} is
essential.  The spectrum unit and module-linearity prove strict compatibility
once the sphere, purity, and Bott coordinates have been aligned with the Euler
calculation of \cite[Lemma~3.5]{AF14}; they do not canonically select that
alignment.  Changing the sphere coordinate can reintroduce a unit in the
even-stage formula.  Thus the proposition is a strict oriented statement, not
a coordinate-free identification with the matrices printed in \cite{AF17}.
\end{remark}

\begin{lemma}\label{lem:unit-suslin-K-normalization}
Let \(k\) be a field of characteristic \(0\) and let \(m\geq1\).  Use the
sphere coordinate on \(Q_{2m-1}\) induced by
\[
p_m:Q_{2m-1}\longrightarrow\mathbb A^m\setminus0,
\qquad (x,y)\longmapsto x,
\]
and the corresponding Bott coordinate for algebraic \(K\)-theory.  The
forgetful image of the unit-adjoint class
\(\Psi_m^{\mathrm{or}}\) is the classical Suslin class
\[
F[\Psi_m^{\mathrm{or}}]=[\alpha_m]
\quad\text{in }\widetilde K_1(Q_{2m-1}).
\]
\end{lemma}

\begin{proof}
Let \(\Phi_m:\mathbf K_m^{MW}\to\mathbf K_m^Q\) be induced by the
classical Suslin matrix \(\alpha_m\).  By
\cite[Lemma~3.8]{AF14}, \(\Phi_m\) differs by a sign from the natural
homomorphism induced by the unit \(\mathds1\to\mathbf{KGL}\).  On the other
hand, \cite[Corollary~3.9]{AF14} gives the exact contraction identity
\[
(\Phi_m)_{-1}=\Phi_{m-1},
\]
while \cite[Lemma~3.7]{AF14} gives the same identity for the natural
homomorphisms.  The coordinates in the statement are coherent under these
two contractions: the projection \(p_m\) fixes the order of the sphere
coordinates, and the chosen Bott coordinates introduce no permutation.
Thus, if \(\Phi_m=s_m\mu_m\) with \(s_m\in\{\pm1\}\), the two cited
contraction identities give \(s_m=s_{m-1}\).  In degree one both maps are
represented by \(t\mapsto t\), so \(s_1=1\).  Induction proves that every
comparison sign is \(1\), which is the asserted equality after adjunction.
\end{proof}

\begin{lemma}\label{lem:koszul-permutation-sign}
Let \(k\) be a field of characteristic different from \(2\), and let
\(i:Z\hookrightarrow Y\) be a regular immersion of smooth \(k\)-schemes
whose conormal bundle is trivialized by an ordered regular sequence
\(f=(f_1,\ldots,f_c)\).  If the order is changed by
\(\tau\in S_c\), then the Koszul Thom coordinate and the associated
Grothendieck--Witt push-forward are changed by the factor
\[
\langle\operatorname{sgn}(\tau)\rangle.
\]
In particular, on a free rank-one \(W(k)\)-summand the factor is
\(\operatorname{sgn}(\tau)\).  Moreover, if a linear automorphism
\(A\in GL_c(k)\) changes the ordered coordinates of a motivic sphere, then
precomposition acts on its oriented rank-one \(W(k)\)-summand by
\(\langle\det A\rangle\).
\end{lemma}

\begin{proof}
The Koszul complex of \(f\) is the exterior algebra on the free conormal
module with its ordered basis.  It is enough to consider an adjacent
transposition.  Interchanging the corresponding degree-one generators and
extending by the graded exterior rule gives an isomorphism of Koszul
complexes.  On the top exterior power this isomorphism is multiplication by
\(-1\).  Comparing the standard self-duality
\[
K(f)\longrightarrow
\operatorname{Hom}\!\left(K(f),\det(I/I^2)\right)[c]
\]
on the two ordered complexes shows directly that the duality square
commutes after the target symmetric form is multiplied by
\(\langle-1\rangle\).  The cohomological shift is the same on both sides,
so it introduces no further sign.  Composing adjacent transpositions gives
the factor
\(\langle\det(P_\tau)\rangle
=\langle\operatorname{sgn}(\tau)\rangle\).

The description of a
regular-immersion push-forward by this symmetric Koszul complex and the
composition theorem for push-forwards
\cite[Proposition~7.1 and Theorem~5.3]{CalmesHornbostel11} show that the same
factor occurs in the Gysin map.  Its image in \(W(k)\) is the ordinary sign
because \(\langle-1\rangle=-1\) there.

For the last assertion, purity identifies the sphere coordinate with the
Thom coordinate of the trivial bundle.  Changing its frame by \(A\) changes
the determinant trivialization by \(\det A\), so the Thom class is multiplied
by \(\langle\det A\rangle\).  This is also the usual formula for the
\(\mathbb A^1\)-degree of a linear automorphism, and proves the claim.
\end{proof}

\begin{lemma}\label{lem:suslin-coordinate-pair-permutation}
Let \(k\) be a field of characteristic \(0\), let \(m\geq2\), let
\(\tau\in S_m\), and let
\[
 \rho_\tau:Q_{2m-1}\longrightarrow Q_{2m-1}
\]
be defined by
\(\rho_\tau^*x_i=x_{\tau(i)}\) and
\(\rho_\tau^*y_i=y_{\tau(i)}\).  On
\[
 G_m=\widetilde{GW}^{m-1}_0(Q_{2m-1})
\]
precomposition by \(\rho_\tau\) is multiplication by
\(\langle\operatorname{sgn}(\tau)\rangle\).  Moreover, if the ordered
tuples \((a,b)\) in the push-forward formulas of
\cite[Lemmas~3.5.2--3.5.3]{AF17} are replaced by their \(\tau\)-permuted
tuples, the resulting right-hand class is literally
\(\rho_\tau^*\eta[\Psi_m^{\mathrm{AF}}]\).  Thus relabeling the coordinate
pairs contributes no sign other than the displayed motivic degree and the
sign already printed in those two lemmas.
\end{lemma}

\begin{proof}
The projection
\(p_m:Q_{2m-1}\to\mathbb A^m\setminus0\) intertwines \(\rho_\tau\) with
the linear coordinate permutation having matrix \(P_\tau\).  The last
assertion of Lemma~\ref{lem:koszul-permutation-sign} therefore shows that
\(\rho_\tau^*\) acts on the free rank-one \(W(k)\)-module \(G_m\) by
\[
 \langle\det P_\tau\rangle
 =\langle\operatorname{sgn}(\tau)\rangle.
\]
For the second assertion, write
\(a^\tau=(a_{\tau(1)},\ldots,a_{\tau(m)})\) and similarly for \(b\).
By Definition~3.3.2 of \cite{AF17}, evaluating the printed formula at
\((a^\tau,b^\tau)\) is exactly the composite
\(\Psi_m^{\mathrm{AF}}\rho_\tau\); the integral matrices \(I_m\) and
\(E_m\) are unchanged because they belong to the target coordinate, not to
the list of functions on the quadric.  The push-forward lemmas of
\cite{AF17} are identities for arbitrary ordered tuples
\((a_1,\ldots,a_m;b_1,\ldots,b_m)\).  Applying them to
\((a^\tau,b^\tau)\) therefore produces
\(\eta[\Psi_m^{\mathrm{AF}}\rho_\tau]\), which is
\(\rho_\tau^*\eta[\Psi_m^{\mathrm{AF}}]\).  No conjugation or further
orientation change is present.
\end{proof}

\begin{theorem}\label{prop:explicit-suslin-rank}
Let \(k\) be a field of characteristic \(0\) and let \(m\geq2\).  Transport
\(\Psi_m^{\mathrm{or}}\) to \(Q_{2m-1}\)
along \(p_m\).  Give the sphere its ordered Koszul orientation
\((x_2,\ldots,x_m)\).  Normalize the degree-one Bott coordinate by the
pointed representative of
\cite[Remark~3.3.4]{AF17}.  Let
\(\Psi_m^{\mathrm{AF}}\) be the stabilized matrix of
\cite[Definitions~3.3.2 and~3.3.5]{AF17}.  Then
\begin{equation}\label{eq:explicit-suslin-full}
[\Psi_m^{\mathrm{AF}}]=\sigma_m[\Psi_m^{\mathrm{or}}],
\end{equation}
where, writing \(\langle-1\rangle\) for the one-dimensional form,
\begin{equation}\label{eq:explicit-suslin-unit}
\sigma_m=
\begin{cases}
1,&m\equiv1\pmod4,\\
-\langle-1\rangle,&m\equiv2\pmod4,\\
\langle-1\rangle,&m\equiv3\pmod4,\\
-1,&m\equiv0\pmod4.
\end{cases}
\end{equation}
In particular,
\begin{equation}\label{eq:explicit-suslin-rank}
\operatorname{rk}(\sigma_m)=(-1)^{m+1}.
\end{equation}
\end{theorem}

\begin{proof}
The printed matrix represents the generator specified in
\cite[Theorem~3.3.6]{AF17}, while the Bott/sphere adjunction identifies
\(\Psi_m^{\mathrm{or}}\) with the unit generator.  Hence both sides of
\eqref{eq:explicit-suslin-full} generate the same free rank-one
\(GW(k)\)-module, so there is a unique
\(\sigma_m\in GW(k)^\times\) relating them.  We determine its rank and its
image in \(W(k)\).

Lemma~\ref{lem:unit-suslin-K-normalization} identifies the forgetful image
of the oriented unit with \([\alpha_m]\).  Constant integral
change-of-basis matrices disappear in reduced \(K_1\), whereas
\cite[Definition~3.3.2]{AF17} has forgetful class
\[
\begin{cases}
[\alpha_m^t],&m\text{ even},\\
[\alpha_m],&m\text{ odd}.
\end{cases}
\]
Since \cite[Lemma~3.4.2]{AF17} gives
\([\alpha_m^t]=(-1)^{m+1}[\alpha_m]\), this proves
\eqref{eq:explicit-suslin-rank}.

It remains to calculate the Witt coordinate.  Put
\[
G_m:=\widetilde{GW}^{m-1}_0(Q_{2m-1}),\qquad
g_m:=\eta[\Psi_m^{\mathrm{AF}}],\qquad
\omega_m:=\eta[\Psi_m^{\mathrm{or}}].
\]
The group \(G_m\) is free of rank one over \(W(k)\).  Let
\[
J_m:G_{m-1}\xrightarrow{\ \cong\ }G_m
\]
be the Gysin isomorphism for \(x_m=0\), with conormal coordinate \(x_m\).
The composition theorem for push-forwards identifies successive
push-forwards with tensoring the length-one Koszul forms with the outer
regular equation first.  In particular, applying \(J_m\) to the ordered
orientation \((x_2,\ldots,x_{m-1})\) gives the Koszul order
\((x_m,x_2,\ldots,x_{m-1})\).  Moving \(x_m\) to the last position has
determinant \((-1)^{m-2}\).  Lemma~\ref{lem:koszul-permutation-sign}
therefore gives, in the \(W(k)\)-module \(G_m\),
\begin{equation}\label{eq:unit-gysin-recursion}
J_m(\omega_{m-1})=(-1)^{m-2}\omega_m.
\end{equation}
In degree one, the calculation at the beginning of the proof of
\cite[Proposition~3.4.3]{AF17} gives
\(\eta(\Psi_1)=\langle t\rangle-\langle1\rangle\); this is also the image
of the unit symbol.  Hence \(g_1=\omega_1\).

We next retain the signs that disappear in the generator-only argument of
\cite[Proposition~3.4.3]{AF17}.  Suppose first that \(m\) is even.  Let
\(\rho_m\) cyclically permute the coordinate pairs by
\[
(1,\ldots,m)\longmapsto(m,1,\ldots,m-1).
\]
Its determinant on the \(x\)-coordinates is \((-1)^{m-1}=-1\), so
Lemma~\ref{lem:suslin-coordinate-pair-permutation} shows that
precomposition by \(\rho_m\) acts on \(G_m\) by
\(\langle-1\rangle=-1\in W(k)\).  Apply
\cite[Lemma~3.5.2]{AF17} with \(a_1=x_m\) and the remaining coordinates in
their original order.  The second assertion of
Lemma~\ref{lem:suslin-coordinate-pair-permutation} shows that the resulting
right-hand class is exactly \(\rho_m^*g_m\); hence the displayed minus sign
in \cite[Lemma~3.5.2]{AF17} gives
\[
J_m(-g_{m-1})=\rho_m^*g_m=-g_m,
\]
and therefore
\begin{equation}\label{eq:explicit-gysin-even}
g_m=J_m(g_{m-1})\qquad(m\text{ even}).
\end{equation}

Now suppose that \(m\geq3\) is odd.  Move the last two coordinate pairs to
the front in the order required by the composite Gysin map:
\[
(1,\ldots,m)\longmapsto(m,m-1,1,\ldots,m-2).
\]
This permutation has determinant \(-1\).  Transitivity and
Lemma~\ref{lem:koszul-permutation-sign} identify
\(J_mJ_{m-1}\) with the push-forward for the ordered regular sequence
\((x_m,x_{m-1})\).  The coordinate permutation therefore contributes
\(\langle-1\rangle=-1\) in \(W(k)\), while the displayed minus sign in
\cite[Lemma~3.5.3]{AF17} contributes a second minus sign.  By
Lemma~\ref{lem:suslin-coordinate-pair-permutation}, the relabeled printed
matrix contributes no further factor.  The two signs cancel, and we obtain
\begin{equation}\label{eq:explicit-gysin-odd}
g_m=J_mJ_{m-1}(g_{m-2})\qquad(m\text{ odd}).
\end{equation}
If \(m\) is even, the sign in
\eqref{eq:unit-gysin-recursion} is \(1\), so
\eqref{eq:explicit-gysin-even} preserves the comparison coefficient.  If
\(m\) is odd, the two successive unit recursions contribute
\((-1)^{m-3}(-1)^{m-2}=-1\), whereas
\eqref{eq:explicit-gysin-odd} has no remaining sign.  Starting from
\(g_1=\omega_1\), we therefore obtain
\begin{equation}\label{eq:explicit-suslin-witt}
g_m=(-1)^{\lfloor(m-1)/2\rfloor}\omega_m.
\end{equation}
Consequently the image of \(\sigma_m\) in \(W(k)^\times\) is
\((-1)^{\lfloor(m-1)/2\rfloor}\).

The first two nontrivial ranks check the orientation signs before the
four-periodic pattern is read off.  For \(m=2\),
\eqref{eq:explicit-gysin-even} and
\eqref{eq:unit-gysin-recursion} give
\[
g_2=J_2(g_1)=J_2(\omega_1)=\omega_2.
\]
Thus the comparison unit has rank \(-1\) and Witt image \(1\), which identifies
it as \(-\langle-1\rangle\).  For \(m=3\),
\eqref{eq:explicit-gysin-odd} and the two unit recursions give
\[
g_3=J_3J_2(g_1)=J_3(\omega_2)=-\omega_3.
\]
The comparison unit now has rank \(1\) and Witt image \(-1\), and hence is
\(\langle-1\rangle\).  These are precisely the \(m=2\) and \(m=3\) entries of
\eqref{eq:explicit-suslin-unit}; in particular, the transpose sign, the Gysin
order, and the ordered Koszul orientation use compatible conventions.

Finally, the homomorphism
\[
GW(k)\longrightarrow\mathbb Z\times W(k),\qquad
q\longmapsto(\operatorname{rk}(q),\overline q),
\]
is injective.  Indeed, the kernel of \(GW(k)\to W(k)\) is generated by the
hyperbolic plane, which has rank \(2\).  Combining
\eqref{eq:explicit-suslin-rank} and
\eqref{eq:explicit-suslin-witt} therefore gives precisely the four cases in
\eqref{eq:explicit-suslin-unit}.
\end{proof}

\begin{corollary}\label{cor:explicit-suslin-wood}
Retain the normalization of Theorem~\ref{prop:explicit-suslin-rank}, and let
\((\delta^{\mathrm{AF}}_{m,j})_*\) be induced by the printed matrices.  In
the ranges of Proposition~\ref{prop:oriented-unit-wood},
\[
(\delta^{\mathrm{AF}}_{n,j})_*\circ d_{n,j+1}
=
\begin{cases}
-\eta\circ(\delta^{\mathrm{AF}}_{n+1,j})_*,&n\text{ even},\\
0,&n\text{ odd}.
\end{cases}
\]
\end{corollary}

\begin{proof}
Substitute \((\delta^{\mathrm{AF}}_{m,j})_*=\sigma_m
(\delta^{\mathrm{or}}_{m,j})_*\) into
Proposition~\ref{prop:oriented-unit-wood}.  Formula
\eqref{eq:explicit-suslin-unit} gives
\(\sigma_n\sigma_{n+1}^{-1}=-1\) for every even \(n\).  The odd-stage map
is zero independently of the normalization.
\end{proof}

\begin{corollary}\label{cor:explicit-suslin-quadratically-closed}
If \(k\) is a quadratically closed field of characteristic \(0\), then
\[
\sigma_m=(-1)^{m+1}.
\]
In particular, this holds over every algebraically closed field of
characteristic \(0\).
\end{corollary}

\begin{proof}
Over a quadratically closed field \(\langle-1\rangle=1\), so the assertion
is immediate from \eqref{eq:explicit-suslin-unit}.
\end{proof}

\begin{remark}\label{rem:explicit-suslin-convention-gap}
Theorem~\ref{prop:explicit-suslin-rank} is convention-sensitive but no longer
contains an undetermined fundamental-ideal component.  The ordered Koszul
orientation is essential: permuting coordinates changes the comparison by
the corresponding motivic degree.  Likewise, replacing the pointed
rank-one Bott model changes every \(\sigma_m\) coherently.

The formula also records the low-rank convention change in \cite{AF17}.
Definition~3.3.2 specializes at \(m=2\) to \(\alpha_2^t\), whereas
Example~4.4.3 writes \(\alpha_2\).  Relative to the orientation fixed above,
the former has comparison unit \(-\langle-1\rangle\); its forgetful rank is
\(-1\), while its Witt image is \(1\).  Thus the transpose changes exactly
the part invisible to either invariant separately.  The calculation does
not produce a coordinate-free preferred generator; it computes the full
unit after all coordinates have been stated.

There is also a low-rank consistency check on the coordinate-permutation
sign.  In the proof of
\cite[Theorem~4.3.1]{AF17}, the authors compare \(\Psi_3\) with the geometric
map \(Q_5\to SL_4/Sp_4\) by the coordinate involution
\(x_2,y_2\mapsto-x_2,-y_2\).  Its motivic degree is
\(\langle-1\rangle\), the same factor as the value \(\sigma_3\) in
\eqref{eq:explicit-suslin-unit}.  This check does not replace the orientation
normalization above, but it agrees with the rank-three matrix convention.
\end{remark}

\begin{corollary}\label{lem:dn2-eta-linearity}
Let \(k\) be a field of characteristic \(0\), let \(n\geq5\), and put
\[
D_{n,n+1}:=(d_{n,2})_{-(n+1)},
\qquad
D_n:=(d_{n,2})_{-(n+2)}.
\]
Then the diagram
\[
\begin{tikzcd}[column sep=large,row sep=large]
\boldsymbol\pi^s_{2,0}(\mathds1)
  \ar[r,"D_{n,n+1}"] \ar[d,"\eta"']&
\boldsymbol\pi^s_{3,1}(\mathds1)
  \ar[d,"\eta"]\\
\boldsymbol\pi^s_{3,1}(\mathds1)
  \ar[r,"D_n"']&
\boldsymbol\pi^s_{4,2}(\mathds1)
\end{tikzcd}
\]
commutes.
\end{corollary}

\begin{proof}
Both horizontal maps are induced by the same stable element described in
Proposition~\ref{prop:dn2-stable-eta}, and multiplication in the stable
motivic homotopy ring is associative.
\end{proof}

\begin{corollary}\label{cor:epsilon-is-eta}
Let \(n\geq5\) be even.  For every characteristic-zero field \(L\) and every
\(\xi\in\boldsymbol\mu_{24}(L)\), choose a lift
\[
\widetilde\xi\in
\boldsymbol\pi^s_{3,1}(\mathds1)(L)
\]
under the quotient in \eqref{eq:source-boundary-extension}.  Then
\[
\epsilon_n(L)(\xi)=\eta\widetilde\xi
\quad\text{in}\quad K^M_2(L)/2.
\]
The right-hand side is independent of the lift.  In particular,
\(\epsilon_n\) is independent of \(n\) among the even stages.
\end{corollary}

\begin{proof}
Proposition~\ref{prop:dn2-stable-eta} identifies \(D_n\) with
multiplication by \(\eta\).  The kernel
\(\mathbf K^M_3/2\) in \eqref{eq:source-boundary-extension} is a
Milnor \(K\)-theory module generated by the contractions of \(\nu^2\).
The relation \(\eta\nu=0\) therefore annihilates this kernel; see
\cite[Theorem~1.2 and the discussion at the beginning of \S3]{RSO21}.
Therefore
\(\eta\widetilde\xi\) depends only on \(\xi\), and
Theorem~\ref{thm:dn2-boundary-defect} identifies the induced quotient map
with \(\epsilon_n\).
\end{proof}

\begin{proposition}\label{prop:epsilon-order-two-zero}
Let \(k\) be a field of characteristic \(0\), let \(n\geq5\), and let
\(\epsilon_n\) be the boundary defect of
Theorem~\ref{thm:dn2-boundary-defect}.  For every field extension \(L/k\),
\[
\epsilon_n(L)(-1)=0.
\]
Consequently, \(\epsilon_n\) factors through
\[
\boldsymbol\mu_{24}/
\bigl(2\boldsymbol\mu_{24}+\langle-1\rangle\bigr).
\]
\end{proposition}

\begin{proof}
The assertion is immediate for odd \(n\), so assume that \(n\) is even.
Let
\[
\eta_{\mathrm c}\in\boldsymbol\pi^s_{1,0}(\mathds1)(L)
\]
be the constant, or topological, Hopf element, and put
\[
a:=\eta\eta_{\mathrm c}^{\,2}
\in\boldsymbol\pi^s_{3,1}(\mathds1)(L).
\]
The relations calculated in \cite[\S6]{RSO21} say that \(a\) maps to
\(-1\in\boldsymbol\mu_{24}(L)\) and that
\[
\eta a=\eta^2\eta_{\mathrm c}^{\,2}=0.
\]
Corollary~\ref{cor:epsilon-is-eta} therefore gives
\(\epsilon_n(L)(-1)=0\).  Combining this with the factorization through
\(\boldsymbol\mu_{24}/2\boldsymbol\mu_{24}\) from
Theorem~\ref{thm:dn2-boundary-defect} proves the sheaf-level factorization.
\end{proof}

\begin{corollary}\label{prop:epsilon-real-zero}
Let \(L/k\) be a field extension such that \(\sqrt{-1}\notin L\).  Then
\[
\epsilon_n(L)=0.
\]
In particular, this holds whenever \(L\) admits a real embedding.
\end{corollary}

\begin{proof}
The group \(\boldsymbol\mu_{24}(L)\) is cyclic.  If
\(\sqrt{-1}\notin L\), its \(2\)-primary subgroup is
\(\boldsymbol\mu_2(L)=\{\pm1\}\).  Therefore
\(\boldsymbol\mu_{24}(L)/2\boldsymbol\mu_{24}(L)\) is generated by the
class of \(-1\), which is killed by
Proposition~\ref{prop:epsilon-order-two-zero}.
\end{proof}

\begin{proposition}\label{prop:epsilon-cyclotomic-reduction}
Fix a primitive eighth root of unity \(\zeta_8\), and let
\(\epsilon_{n,F}\) denote the map on sections over a characteristic-zero
field \(F\).  The fieldwise values of \(\epsilon_n\) are determined by the
two universal classes
\[
e_{n,4}:=\epsilon_{n,\mathbb Q(i)}(i)
\in K^M_2(\mathbb Q(i))/2,
\qquad
e_{n,8}:=\epsilon_{n,\mathbb Q(\zeta_8)}(\zeta_8)
\in K^M_2(\mathbb Q(\zeta_8))/2.
\]
They satisfy
\[
e_{n,4}\in
\bigl(K^M_2(\mathbb Q(i))/2\bigr)^{\operatorname{Gal}(\mathbb Q(i)/\mathbb Q)},
\qquad
e_{n,8}\in
\bigl(K^M_2(\mathbb Q(\zeta_8))/2\bigr)^{
\operatorname{Gal}(\mathbb Q(\zeta_8)/\mathbb Q)}
\]
and
\[
\operatorname{res}_{\mathbb Q(\zeta_8)/\mathbb Q(i)}(e_{n,4})=0.
\]
In particular,
\[
e_{n,4}\in
\{i\}\,K^M_1(\mathbb Q(i))/2
\subset K^M_2(\mathbb Q(i))/2.
\]
\end{proposition}

\begin{proof}
The Stiefel morphisms, the connecting maps, and the canonical stable
identifications used to define \(\epsilon_n\) commute with extension of the
ground field.  Let \(L\) be a characteristic-zero field and
\(\xi\in\boldsymbol\mu_{24}(L)\).  The odd-primary component of \(\xi\) is a
square in \(\boldsymbol\mu_{24}(L)\), and hence is killed by
Theorem~\ref{thm:dn2-boundary-defect}.  The \(2\)-primary component has order
at most \(8\).  The order-two case is killed by
Proposition~\ref{prop:epsilon-order-two-zero}; an order-four component is
obtained from \(i\) by scalar extension and possibly inversion; and an
order-eight component is obtained in the same way from \(\zeta_8\).
Since the target is annihilated by \(2\), inversion does not change the
value.  This proves the determination by \(e_{n,4}\) and \(e_{n,8}\).

Complex conjugation sends \(i\) to \(-i\), and
\[
\epsilon_{n,\mathbb Q(i)}(-i)
=\epsilon_{n,\mathbb Q(i)}(-1)
\mathbin{+}\epsilon_{n,\mathbb Q(i)}(i)
=e_{n,4}.
\]
Every automorphism of \(\mathbb Q(\zeta_8)\) sends \(\zeta_8\) to
\(\zeta_8^r\) for an odd \(r\), and
\(\epsilon_n(\zeta_8^r)=r\epsilon_n(\zeta_8)=e_{n,8}\).
This proves the two invariance statements.

In \(\mathbb Q(\zeta_8)\), one has \(i=\zeta_8^2\).  The factorization
through
\(\boldsymbol\mu_{24}/2\boldsymbol\mu_{24}\) therefore gives
\[
\operatorname{res}_{\mathbb Q(\zeta_8)/\mathbb Q(i)}(e_{n,4})=0.
\]
Finally, \(\mathbb Q(\zeta_8)=\mathbb Q(i)(\sqrt{i})\).  The standard
restriction sequence for the \(2\)-torsion Brauer group, together with the
norm-residue isomorphism, identifies the kernel of
\[
K^M_2(\mathbb Q(i))/2
\longrightarrow
K^M_2(\mathbb Q(\zeta_8))/2
\]
with \(\{i\}K^M_1(\mathbb Q(i))/2\); see \cite{OVV07}.  This proves the last
assertion.
\end{proof}

\begin{theorem}\label{thm:epsilon-global-zero}
Let \(k\) be a field of characteristic \(0\) and let \(n\geq5\).  Then
\[
\epsilon_n:
\boldsymbol\mu_{24}\longrightarrow\mathbf K^M_2/2
\]
is the zero morphism.  Equivalently,
\[
(d_{n,2})_{-(n+2)}=0.
\]
\end{theorem}

\begin{proof}
The assertion for odd \(n\) is Proposition~\ref{prop:dn2-odd-zero}.  Suppose
that \(n\) is even.  By
Proposition~\ref{prop:epsilon-cyclotomic-reduction}, it suffices to prove
\[
e_{n,4}=0
\quad\text{and}\quad
e_{n,8}=0.
\]
Corollary~\ref{cor:epsilon-is-eta} identifies these classes with the
operation induced by \(\eta\) on the fourth and eighth roots of unity.

We first prove an odd-place vanishing statement for this operation.  Let
\[
(F,\xi)=\bigl(\mathbb Q(i),i\bigr)
\quad\text{or}\quad
(F,\xi)=\bigl(\mathbb Q(\zeta_8),\zeta_8\bigr),
\]
and let \(v\nmid2\) be a finite place of \(F\).  Write
\[
R:=\mathcal O_{F,v}^{\,h},
\qquad
K:=\operatorname{Frac}(R),
\qquad
\kappa:=R/\mathfrak m.
\]
The reduction \(\bar\xi\in\kappa^\times\) has the same order as \(\xi\).
Indeed, the difference of two distinct \(2\)-power roots of unity is a unit
at every odd place.

Work in the \(2\)-complete cellular stable motivic categories.  Since
\(1/2\in R\), the rigidity equivalence of
\cite[Proposition~3.1(1)]{BO22} gives
\[
 i^*:\mathbf{SH}(R)^{\mathrm{cell},\wedge}_2
 \xrightarrow{\ \simeq\ }
 \mathbf{SH}(\kappa)^{\mathrm{cell},\wedge}_2.
\]
The functor \(i^*\) is exact and symmetric monoidal.  Thus this equivalence
identifies the sphere units and their \(\eta\)-actions; the argument below
compares the operation itself and not merely the abstract homotopy groups.
The statement applies to the henselization above by the continuity
formulation recorded in the introduction of \cite{BO22}.  More explicitly,
because \(v\nmid2\), the local ring \(\mathcal O_{F,v}\) is a localization
of the regular Dedekind ring \(\mathcal O_F[1/2]\), and its henselization is
the filtered limit of its pointed \(\acute{e}\)tale neighborhoods.  It is
therefore in the essentially smooth henselian-local case covered by that
proposition and continuity.  The sphere is
cellular by definition, and \(\mathbf{KQ}\) is cellular over every base with
no residue characteristic \(2\) by
\cite[Theorem~1.1]{RSO18Cellular}; hence both objects lie in the categories
to which the displayed rigidity equivalence applies.

The positive-characteristic version of
\cite[Theorem~1.2 and the paragraph following it]{RSO21}, followed by
\(2\)-completion, supplies a class
\[
\bar a\in
\pi_{3,1}(\mathds1_\kappa)^\wedge_2
\]
whose image under the stable unit is \(\bar\xi\).  The theorem in positive
characteristic is stated after inverting the exponential characteristic;
this does not change its \(2\)-completion because that characteristic is
odd.  Here we use
\cite[Theorem~2.38]{RSO21} to identify
\(\pi_{3,1}(\mathbf{KQ}_\kappa)\) with \(\kappa^\times\).
Rigidity lifts \(\bar a\) uniquely to a class
\[
a_R\in\pi_{3,1}(\mathds1_R)^\wedge_2.
\]

We record why the unit image of \(a_R\) is the unit \(\xi\), rather than
only some lift of \(\bar\xi\).  The hyperbolic map and Bott periodicity give
a commutative square
\[
\begin{tikzcd}[column sep=large,row sep=large]
(R^\times)^\wedge_2
  \ar[r,"\mathrm{hyp}"] \ar[d,"\cong"']&
\pi_{3,1}(\mathbf{KQ}_R)^\wedge_2
  \ar[d,"\cong"]\\
(\kappa^\times)^\wedge_2
  \ar[r,"\mathrm{hyp}"']&
\pi_{3,1}(\mathbf{KQ}_\kappa)^\wedge_2.
\end{tikzcd}
\]
The right vertical map is the rigidity isomorphism.  The lower horizontal
map is an isomorphism by the Wood sequence calculation
\cite[Lemma~2.33 and Theorem~2.38]{RSO21}.  The left vertical map is an
isomorphism because \(1+\mathfrak m\) is uniquely \(2\)-divisible: for
every \(r\geq1\), Hensel's lemma applied to \(T^{2^r}-u\) shows that
\[
1+\mathfrak m\xrightarrow{\;2^r\;}1+\mathfrak m
\]
is bijective.  Thus the upper horizontal map is also an isomorphism, and
the unit image of \(a_R\) is \(\xi\).

For the finite field \(\kappa\), Matsumoto's theorem gives
\(K^M_2(\kappa)=0\).  In the case \(n=2\) of the positive-characteristic
version of \cite[Theorem~1.2]{RSO21}, the other kernel term is
\(H^{-1,0}(\kappa)/24=0\).  Hence the unit map embeds
\(\pi_{4,2}(\mathds1_\kappa)[1/p]\) into
\(\pi_{4,2}(\mathbf{kq}_\kappa)[1/p]\), where
\(p=\operatorname{char}(\kappa)\).  The paragraph following
\cite[Theorem~1.2]{RSO21} says that the source before localization is
torsion, whereas \cite[Theorem~2.38]{RSO21} identifies the target before
localization with the torsion-free group \(K^M_0(\kappa)\).  Hence the
localized source is zero.  Since \(p\) is odd, inverting \(p\) does not
change the \(2\)-completion, and therefore
\[
\pi_{4,2}(\mathds1_\kappa)^\wedge_2=0.
\]
Consequently \(\eta\bar a=0\).  Rigidity and restriction to the generic
point give
\[
\eta a_K=0,
\qquad
a_K:=a_R|_K,
\]
where \(a_K\) maps to \(\xi\) in
\(\pi_{3,1}(\mathbf{KQ}_K)^\wedge_2\).
The chosen embedding of the henselian fraction field \(K\) into the
completion \(F_v\) carries both \(a_K\) and the equality
\(\eta a_K=0\) to \(F_v\).

Over the characteristic-zero field \(K\), the split exact sequence in
\cite[Theorem~1.2]{RSO21} identifies the kernel of the unit on
\(\pi_{3,1}\) with \(K^M_3(K)/2\); here
\cite[Theorem~2.38]{RSO21} identifies
\(\pi_{3,1}(\mathbf{kq}_K)\to\pi_{3,1}(\mathbf{KQ}_K)\), so the same
description applies to the unit into periodic hermitian \(K\)-theory.  This
kernel is generated by contractions of \(\nu^2\) and is annihilated by the
relation \(\eta\nu=0\); see
\cite[Theorem~1.2 and the discussion at the beginning of \S3]{RSO21}.
It follows that the \(2\)-completion of the class
\(\eta\widetilde\xi\), for any integral lift \(\widetilde\xi\) of \(\xi\),
is zero.  This already implies
\[
\eta\widetilde\xi=0
\quad\text{in}\quad
\pi_{4,2}(\mathds1_K)=K^M_2(K)/2.
\]
The last implication follows injectively: the target is
annihilated by \(2\), so the map
\(\pi_{4,2}(\mathds1_K)\to\pi_{4,2}(\mathds1_K/8)\) is injective by the
Bockstein exact sequence, and \(2\)-completion maps to
\(\mathds1_K/8\).  We have proved that the restriction of the relevant
class to the completion \(F_v\) is zero for every \(v\nmid2\).

By the norm-residue isomorphism,
\[
K^M_2(F)/2\cong\operatorname{Br}(F)[2].
\]
Both \(\mathbb Q(i)\) and \(\mathbb Q(\zeta_8)\) are totally imaginary, and
each has a unique place above \(2\).  The class under consideration has
zero local invariant at every finite place away from \(2\), by the
preceding argument, and at every archimedean place.  The
Brauer--Hasse--Noether exact sequence
\cite[Chapter~IV, Example~2.14(e)]{MilneCFT} says that the sum of all local
invariants is zero.  Since only the single dyadic place remains, its
invariant is zero as well.  Hence the global Brauer class is zero.  This
proves \(e_{n,4}=e_{n,8}=0\), and
Proposition~\ref{prop:epsilon-cyclotomic-reduction} proves the result over
every characteristic-zero field.
\end{proof}

\begin{remark}\label{rem:epsilon-real-realization-check}
Over \(\mathbb R\), real realization gives an independent check.
It sends \(\eta\eta_{\mathrm c}^{\,2}\) to
\(2\eta_{\mathrm{top}}^2=0\), whereas
\cite[Example~6.6]{Gant_Williams_2025} identifies the generator of
\(\mathbf K^M_2(\mathbb R)/2\) with a class realizing to the nonzero element
\(\eta_{\mathrm{top}}^2\).  Hence realization is injective on the target and
also proves \(\epsilon_n(\mathbb R)=0\).
\end{remark}

\subsection{\texorpdfstring{The \(kq\)-visible range before the boundary}
{The kq-visible range before the boundary}}

The periodic compatibility controls more than the single boundary weight.
The strict inequality in \cite[Theorem~2.38]{RSO21} gives an exact starting
point for transporting it back from periodic hermitian \(K\)-theory to
\(\mathbf{kq}\).

\begin{theorem}\label{thm:dn2-kq-visible-range}
Let \(k\) be a field of characteristic \(0\), let \(n\geq5\), and let
\[
n-2\leq r\leq n+2.
\]
Put
\[
w:=n-r,\qquad
D_{n,r}:=(d_{n,2})_{-r},
\qquad
\mathcal Q_j:=\boldsymbol\pi^s_{2-j\alpha}(\mathbf{kq})
              =\boldsymbol\pi^s_{2-j,-j}(\mathbf{kq}),
\]
and, for \(m\in\{n,n+1\}\), put
\[
\mathcal A_{m,r}:=
\mathbf H^{m+1-r,m+2-r}/24
\ \oplus\
\mathbf K^M_{m+4-r}/2.
\]
There are normalized unit morphisms
\[
\overline u_{m,r}:
(\mathcal P_m)_{-r}\longrightarrow\mathcal Q_{m-r}
\]
and a commutative diagram with exact rows
\begin{equation}\label{eq:dn2-kq-visible-diagram}
\begin{tikzcd}[column sep=large, row sep=large]
0 \ar[r]&
\mathcal A_{n+1,r}
  \ar[r] \ar[d,"\kappa_{n,r}"']&
(\mathcal P_{n+1})_{-r}
  \ar[r,"\overline u_{n+1,r}"] \ar[d,"D_{n,r}"']&
\mathcal Q_{w+1}
  \ar[d,"\vartheta_n"]\\
0 \ar[r]&
\mathcal A_{n,r}
  \ar[r]&
(\mathcal P_n)_{-r}
  \ar[r,"\overline u_{n,r}"]&
\mathcal Q_w ,
\end{tikzcd}
\end{equation}
where \(\kappa_{n,r}\) is the restriction of \(D_{n,r}\) and
\[
\vartheta_n=
\begin{cases}
\eta,&n\text{ even},\\
0,&n\text{ odd}.
\end{cases}
\]
For \(r<n+2\), the four \(\mathbf{kq}\)-morphisms supplied by this theorem are
\begin{equation}\label{eq:dn2-kq-visible-table}
\begin{array}{c|c|c}
r&w&\text{\(\mathbf{kq}\)-morphism}\\ \hline
n-2&2&\mathcal Q_3\xrightarrow{\ \vartheta_n\ }\mathcal Q_2\\
n-1&1&\mathcal Q_2\xrightarrow{\ \vartheta_n\ }\mathcal Q_1\\
n&0&\mathcal Q_1\xrightarrow{\ \vartheta_n\ }\mathcal Q_0\\
n+1&-1&\mathcal Q_0\xrightarrow{\ \vartheta_n\ }\mathcal Q_{-1}.
\end{array}
\end{equation}
On sections over a field \(L/k\), \cite[Theorem~2.38]{RSO21} gives
\[
\mathcal Q_{-1}(L)\cong K^M_1(L),
\qquad
\mathcal Q_0(L)\cong K^M_2(L)\oplus k^M_0(L),
\qquad
k^M_*(L):=K^M_*(L)/2,
\]
as abelian groups, while \(\mathcal Q_1(L),\mathcal Q_2(L),\mathcal Q_3(L)\)
are the three explicit Milnor-symbol extensions displayed there.
\end{theorem}

\begin{proof}
The stable identification \eqref{eq:Pn-stable-identification} and
\eqref{eq:stable-contraction-alpha} give
\[
(\mathcal P_{n+1})_{-r}
\cong\boldsymbol\pi^s_{2-(w+1)\alpha}(\mathds1),
\qquad
(\mathcal P_n)_{-r}
\cong\boldsymbol\pi^s_{2-w\alpha}(\mathds1).
\]
In the notation \(\pi_{2-(j)}=\pi_{2-j,-j}\) of \cite{RSO21}, the
corresponding \(\mathbf{kq}\)-terms therefore have indices \(j=w+1\) and
\(j=w\).  Since \(r\geq n-2\), one has
\[
w+1=n+1-r\leq3<4.
\]
Thus \cite[Theorem~2.38]{RSO21} identifies
\(\mathbf{kq}\to\mathbf{KQ}\) on both terms.

After stabilization and \(r\) contractions, the two normalized
\(KO\)-degree maps in \eqref{eq:dn2-KO-compatibility} are the unit-induced
maps into \(\mathbf{KQ}\), followed by the normalization automorphisms.
Transporting them through the two isomorphisms
\[
\boldsymbol\pi^s_{2-j\alpha}(\mathbf{kq})
\xrightarrow{\ \cong\ }
\boldsymbol\pi^s_{2-j\alpha}(\mathbf{KQ}),
\qquad j=w,w+1,
\]
defines \(\overline u_{n,r}\) and \(\overline u_{n+1,r}\).  The comparison
\(\mathbf{kq}\to\mathbf{KQ}\) is a morphism of \(\mathds1\)-modules, so it
commutes with the \(\eta\)-action.
Proposition~\ref{prop:dn2-KO-compatibility} now proves the right-hand square in
\eqref{eq:dn2-kq-visible-diagram}.

Contracting \eqref{eq:Pn-second-stem} is exact by
Lemma~\ref{lem:contraction-exact}, and its kernel is precisely
\(\mathcal A_{m,r}\).  Postcomposition by the normalization automorphism
does not change this kernel.  The right-hand square therefore carries
\(\mathcal A_{n+1,r}\) into \(\mathcal A_{n,r}\), defining
\(\kappa_{n,r}\) and completing the diagram.  The stated fieldwise
descriptions of \(\mathcal Q_{-1},\ldots,\mathcal Q_3\) are exactly the cases
\(j=-1,0,1,2,3\) of
\cite[Theorem~2.38]{RSO21}.  Finally, the four rows in
\eqref{eq:dn2-kq-visible-table} are obtained by substituting
\(r=n-2,n-1,n,n+1\).
\end{proof}

\begin{corollary}\label{cor:dn2-unit-image}
Under the hypotheses of Theorem~\ref{thm:dn2-kq-visible-range}, write
\[
\mathcal I_{m,r}:=\operatorname{im}(\overline u_{m,r}).
\]
The morphism induced by \(D_{n,r}\) modulo the unit kernels is
\[
\overline D_{n,r}:
\mathcal I_{n+1,r}\longrightarrow\mathcal I_{n,r},
\qquad
\overline D_{n,r}
=\vartheta_n|_{\mathcal I_{n+1,r}}.
\]
Consequently,
\begin{equation}\label{eq:dn2-visible-kernel-cokernel}
\ker(\overline D_{n,r})
=\mathcal I_{n+1,r}\cap\ker(\vartheta_n),
\qquad
\operatorname{coker}(\overline D_{n,r})
=
{\mathcal I_{n,r}}/
{\vartheta_n(\mathcal I_{n+1,r})}.
\end{equation}
Thus every part of \(D_{n,r}\) not determined by
\cite[Theorem~2.38]{RSO21} is supported on the target unit kernel
\(\mathcal A_{n,r}\).
\end{corollary}

\begin{proof}
Exactness of the rows in \eqref{eq:dn2-kq-visible-diagram} identifies
\[
(\mathcal P_m)_{-r}/\mathcal A_{m,r}
\cong\mathcal I_{m,r}.
\]
The right-hand square identifies the induced quotient morphism with
\(\vartheta_n\).  Formula \eqref{eq:dn2-visible-kernel-cokernel} follows.
\end{proof}

\begin{remark}\label{rem:dn2-visible-residual-kernels}
For the four contractions strictly before the boundary, the residual target
kernels in Corollary~\ref{cor:dn2-unit-image} are
\[
\begin{array}{c|c}
r&\mathcal A_{n,r}\\ \hline
n-2&\mathbf H^{3,4}/24\oplus\mathbf K^M_6/2\\
n-1&\mathbf H^{2,3}/24\oplus\mathbf K^M_5/2\\
n&\mathbf H^{1,2}/24\oplus\mathbf K^M_4/2\\
n+1&\mathbf H^{0,1}/24\oplus\mathbf K^M_3/2.
\end{array}
\]
Thus Theorem~\ref{thm:dn2-kq-visible-range} gives a complete calculation
whenever the entry in the second column vanishes after evaluation.  It also
shows why fewer contractions are harder than the \(n+2\) boundary:
motivic-cohomology terms, absent from the target there, remain in the last
four pre-boundary kernels.
\end{remark}

\begin{corollary}\label{cor:dn2-visible-injective-target}
Let \(L/k\) be a field extension and suppose
\(\mathcal A_{n,r}(L)=0\).  Then \(D_{n,r}(L)\) is uniquely determined by
\[
\overline u_{n,r}\bigl(D_{n,r}(x)\bigr)
=
\vartheta_n\bigl(\overline u_{n+1,r}(x)\bigr).
\]
In particular,
\[
\ker D_{n,r}(L)
=
\overline u_{n+1,r}^{-1}\!\bigl(\ker\vartheta_n\bigr),
\qquad
\operatorname{coker}D_{n,r}(L)
\cong
{\mathcal I_{n,r}(L)}/
{\vartheta_n(\mathcal I_{n+1,r}(L))}.
\]
\end{corollary}

\begin{proof}
The hypothesis and exactness of the lower row in
\eqref{eq:dn2-kq-visible-diagram} make \(\overline u_{n,r}(L)\) injective.
The assertions follow from commutativity and
Corollary~\ref{cor:dn2-unit-image}.
\end{proof}

\begin{corollary}\label{cor:dn2-visible-low-cd2}
Let \(0\leq s\leq3\), put \(r=n+1-s\), and let \(L/k\) be a field extension
such that
\[
\operatorname{cd}_2(L)\leq s+2,
\qquad
H^{s,s+1}(L,\mathbb Z)/24=0.
\]
Then \(D_{n,r}(L)\) is completely determined by
\[
\overline u_{n,r}\bigl(D_{n,r}(x)\bigr)
=
\vartheta_n\bigl(\overline u_{n+1,r}(x)\bigr).
\]
Equivalently, the four rows of
\eqref{eq:dn2-kq-visible-table} give complete field-valued calculations under
the corresponding vanishing hypotheses.
\end{corollary}

\begin{proof}
For \(r=n+1-s\), the target unit kernel is
\[
\mathcal A_{n,r}(L)
=H^{s,s+1}(L,\mathbb Z)/24
\ \oplus\
K^M_{s+3}(L)/2.
\]
The first summand vanishes by hypothesis.  The norm-residue isomorphism and
\(\operatorname{cd}_2(L)\leq s+2\) give
\[
K^M_{s+3}(L)/2
\cong
H^{s+3}_{\mathrm{\acute et}}
  (L,\boldsymbol\mu_2^{\otimes(s+3)})
=0;
\]
see \cite{OVV07}.  Corollary~\ref{cor:dn2-visible-injective-target} applies.
\end{proof}

\begin{corollary}\label{cor:dn2-next-low-cd2}
Let \(L/k\) be a field extension with \(\operatorname{cd}_2(L)\leq2\).
Then \(D_{n,n+1}(L)\) is completely determined by the last line of
\eqref{eq:dn2-kq-visible-table}.  More precisely, for even \(n\),
\[
\overline u_{n,n+1}\bigl(D_{n,n+1}(x)\bigr)
=
\eta\,\overline u_{n+1,n+1}(x)
\quad\text{in}\quad
\mathcal Q_{-1}(L)\cong K^M_1(L),
\]
while \(D_{n,n+1}(L)=0\) for odd \(n\).  This applies in particular to
algebraically closed fields of characteristic \(0\).
\end{corollary}

\begin{proof}
This is Corollary~\ref{cor:dn2-visible-low-cd2} with \(s=0\), since
\(H^{0,1}(L,\mathbb Z)=0\) by motivic-cohomology connectivity.  The odd case
also follows directly from Proposition~\ref{prop:dn2-odd-zero}.
\end{proof}

\begin{theorem}\label{thm:dn2-kq-blindness}
Let \(k\) be a field of characteristic \(0\), let \(n\geq5\), and put
\[
D_n:=(d_{n,2})_{-(n+2)}:
\boldsymbol\pi^s_{3,1}(\mathds1)
\longrightarrow
\boldsymbol\pi^s_{4,2}(\mathds1).
\]
After contracting the diagram
\eqref{eq:dn2-KO-compatibility} \(n+2\) times and using
  the canonical comparison \(\iota:\mathbf{kq}\to\mathbf{KQ}\), one obtains a
  commutative square
\begin{equation}\label{eq:dn2-kq-blind-square}
\begin{tikzcd}[column sep=large, row sep=large]
\boldsymbol\pi^s_{3,1}(\mathds1)
  \ar[r,"D_n"] \ar[d,"\overline u_{n+1}"']&
\boldsymbol\pi^s_{4,2}(\mathds1)
  \ar[d,"\overline u_n=0"]\\
\boldsymbol\pi^s_{3,1}(\mathbf{kq})\cong\mathbf K^M_1
  \ar[r,"\eta=0"']&
\boldsymbol\pi^s_{4,2}(\mathbf{kq})\cong\mathbf K^M_0.
\end{tikzcd}
\end{equation}
Here \(\overline u_m\) denotes the unit-induced map, transported through
\(\iota_*\), followed by the automorphism coming from the compatible
normalization.  Consequently,
\eqref{eq:dn2-kq-blind-square} imposes no condition on \(D_n\), and in
particular does not determine the defect
\(\epsilon_n:\boldsymbol\mu_{24}\to\mathbf K^M_2/2\).
\end{theorem}

\begin{proof}
In the notation \(\pi_{2-(w)}=\pi_{2-w,-w}\) of \cite{RSO21}, the two
bottom terms have \(w=-1\) and \(w=-2\).  Thus
\cite[Theorem~2.38]{RSO21} says that \(\iota_*\) is an isomorphism on
these sheaves and gives
\[
\boldsymbol\pi^s_{3,1}(\mathbf{kq})\cong\mathbf K^M_1,
\qquad
\boldsymbol\pi^s_{4,2}(\mathbf{kq})\cong\mathbf K^M_0.
\]
The Wood cofiber sequence of \cite[Theorem~2.1]{RSO21} supplies the exact
segment
\[
\mathbf K^M_1
\xrightarrow[\cong]{\ \mathrm{hyper}\ }
\boldsymbol\pi^s_{3,1}(\mathbf{kq})
\xrightarrow{\ \eta\ }
\boldsymbol\pi^s_{4,2}(\mathbf{kq});
\]
the first map is an isomorphism by the calculation in
\cite[Lemma~2.33]{RSO21}.  Hence the lower horizontal map in
\eqref{eq:dn2-kq-blind-square} is zero.

On the other hand, the case of bidegree \((4,2)\) in
\cite[Theorem~1.2]{RSO21} gives
\[
\boldsymbol\pi^s_{4,2}(\mathds1)\cong\mathbf K^M_2/2.
\]
The unit-induced map from this \(2\)-torsion sheaf to
\(\boldsymbol\pi^s_{4,2}(\mathbf{kq})\cong\mathbf K^M_0\) is zero because
the latter is torsion-free.  The normalization automorphism does not change
this vanishing, so the right vertical map is zero.

Theorem~\ref{thm:dn2-kq-visible-range}, with \(r=n+2\), gives the
commutativity of \eqref{eq:dn2-kq-blind-square}.  Since both composites are
zero for every possible \(D_n\), the square cannot detect \(\epsilon_n\).
\end{proof}

\begin{remark}\label{rem:dn2-linear-spectral-sequence}
The linear exact couple has
\[
E^1_{p,q}=\pi_{p+q}^{\A^1}(\A^p\setminus0)
\]
by \cite[Proposition~2.1.1]{AF17Spheres}.  With the exact-couple sign
convention used there, \(d_{n,1}=d^1_{n+1,0}\) and
\[
d_{n,2}=d^1_{n+1,1}.
\]
With the opposite convention for connecting morphisms, both identifications
carry the usual common sign; none of the kernels, cokernels, or vanishing
statements below depends on this choice.
Here \(q\) is the second index of this exact couple, not a rank or a
contraction index.  Remark~20 of \cite{AOSS26} identifies
\(d^1_{j+2,0}\) and also records certain higher and symplectic
differentials, but it does not identify the linear differential
\(d^1_{n+1,1}=d_{n,2}\).  Proposition~\ref{prop:dn2-stable-eta} supplies
the missing identification by a different argument: the simplicial
desuspension of the Stiefel boundary lies in the Freudenthal range and is
detected on its bottom Milnor--Witt homotopy sheaf.  Thus
\(d^1_{n+1,1}\) is induced by \(\eta\) in even stages and is zero in odd
stages.  Theorem~\ref{thm:dn2-kq-visible-range} determines the
\(\mathbf{kq}\)-visible quotient of this \(q=1\) differential after every
contraction \(r\geq n-2\), and
Corollary~\ref{cor:dn2-unit-image} describes its image modulo the explicit
unit kernel \(\mathcal A_{n,r}\).  At \(r=n+2\),
Theorem~\ref{thm:dn2-kq-blindness} still shows that the visible quotient is
the tautology \(0=0\); the vanishing at that weight instead follows from
Theorem~\ref{thm:epsilon-global-zero}.  At \(r=n-3\), the source has index
\(w=4\), exactly outside the strict comparison range \(w<4\) of
\cite[Theorem~2.38]{RSO21}, although the sphere-level differential itself
is already fixed by Proposition~\ref{prop:dn2-stable-eta}.
\end{remark}

Combining this result with the lower exact row of
Proposition~\ref{prop:third-unstable-two-rows} gives the following
contracted statement.

\begin{corollary}\label{cor:third-unstable-contracted-rows}
Let $k$ and $n$ be as in Theorem~\ref{thm:dn2-contracted-zero}.  There are
exact sequences
\begin{equation}\label{eq:third-lower-row-weight-next-three}
\begin{aligned}
\pi^{\A^1}_{n+1,n+3}(SL_{n-1})
&\longrightarrow
\pi^{\A^1}_{n+1,n+3}(SL_n)
\xrightarrow{\ (q_{n-1,2})_{-(n+3)}\ }
\mathbf K^M_1/2\\
&\xrightarrow{\ (\partial_{n,2})_{-(n+3)}\ }
\pi^{\A^1}_{n,n+3}(SL_{n-1})
\longrightarrow
\pi^{\A^1}_{n,n+3}(SL_n),
\end{aligned}
\end{equation}
and
\begin{equation}\label{eq:third-lower-row-weight-next-four}
\begin{aligned}
\pi^{\A^1}_{n+1,n+4}(SL_{n-1})
&\longrightarrow
\pi^{\A^1}_{n+1,n+4}(SL_n)
\xrightarrow{\ (q_{n-1,2})_{-(n+4)}\ }
\mathbf K^M_0/2\\
&\xrightarrow{\ (\partial_{n,2})_{-(n+4)}\ }
\pi^{\A^1}_{n,n+4}(SL_{n-1})
\longrightarrow
\pi^{\A^1}_{n,n+4}(SL_n).
\end{aligned}
\end{equation}
The additional pairs
\[
\mathbf K^M_2/2
\xrightarrow{\ (\delta_{n,2})_{-(n+3)}\ }
\pi^{\A^1}_{n+1,n+3}(SL_n)
\xrightarrow{\ (q_{n-1,2})_{-(n+3)}\ }
\mathbf K^M_1/2
\]
and
\[
\mathbf K^M_1/2
\xrightarrow{\ (\delta_{n,2})_{-(n+4)}\ }
\pi^{\A^1}_{n+1,n+4}(SL_n)
\xrightarrow{\ (q_{n-1,2})_{-(n+4)}\ }
\mathbf K^M_0/2
\]
are complexes, but are not asserted to be exact at their middle terms.  In
each weight, the image of the displayed \(\delta_{n,2}\) is contained in the
image of
\[
\pi^{\A^1}_{n+1,*}(SL_{n-1})
\longrightarrow\pi^{\A^1}_{n+1,*}(SL_n).
\]
Finally, the kernel of stabilization in degree \(n\) and weight \(n+3\) is a
quotient of \(\mathbf K^M_1/2\), and the analogous kernel in weight \(n+4\)
is a quotient of \(\mathbf K^M_0/2\).
\end{corollary}

\begin{proof}
Contract \eqref{eq:third-unstable-lower-row} and use
\eqref{eq:Pn-high-contractions}; this gives
\eqref{eq:third-lower-row-weight-next-three} and
\eqref{eq:third-lower-row-weight-next-four}.  The two complex relations are exactly
the vanishing statements in Theorem~\ref{thm:dn2-contracted-zero}.
Exactness at the middle special linear term identifies the kernel of each
contracted \(q_{n-1,2}\) with the image of stabilization, proving the
factorization assertion.  Exactness at the next special linear term gives the
final quotient statements.
\end{proof}

\begin{remark}\label{rem:third-unstable-meaning}
Corollary~\ref{cor:third-unstable-contracted-rows} deliberately separates two
kinds of information.  The lower row is exact and controls a stabilization
kernel by Milnor \(K\)-theory modulo \(2\).  The cross-composite calculation only
says that the contracted boundary from the adjacent upper row lands in the
stabilization image.  Proposition~\ref{prop:dn2-KO-compatibility} determines
the image of the uncontracted morphism in periodic hermitian \(K\)-theory.
Theorem~\ref{thm:dn2-kq-visible-range} upgrades this to a very-effective
\(\mathbf{kq}\)-level description after \(r\geq n-2\) contractions.
Proposition~\ref{prop:dn2-stable-eta} identifies the sphere-level
cross-composite before contraction, and Theorem~\ref{thm:epsilon-global-zero}
settles its first boundary specialization.
\end{remark}

\begin{corollary}\label{cor:dn2-kq}
For \(n\geq5\), the stable unit gives a commutative square
\[
\begin{tikzcd}[column sep=large,row sep=large]
\mathcal P_{n+1}
  \ar[r,"d_{n,2}"] \ar[d]&
\mathcal P_n \ar[d]\\
\boldsymbol\pi^s_{2-(n+1)\alpha}(\mathbf{kq})
  \ar[r,"\vartheta_n"']&
\boldsymbol\pi^s_{2-n\alpha}(\mathbf{kq}),
\end{tikzcd}
\]
where
\[
\vartheta_n=
\begin{cases}
\eta,&n\text{ even},\\
0,&n\text{ odd}.
\end{cases}
\]
relative to the standard sphere and purity identifications of
Proposition~\ref{prop:dn2-stable-eta}.  With arbitrary sphere coordinates,
the even-stage horizontal map is \(u_n\eta\); all contracted diagrams above
are unchanged.
Thus the compatibility exists before contraction; after contraction it
recovers the diagrams of Theorem~\ref{thm:dn2-kq-visible-range}.
\end{corollary}

\begin{proof}
Under Theorem~\ref{thm:Pn-second-stem}, the upper horizontal map is
\(\vartheta_n\) by Proposition~\ref{prop:dn2-stable-eta}.  The stable unit
is a morphism of \(\mathds1\)-modules, so it commutes with multiplication by
\(\eta\).  The odd case is the zero square.
\end{proof}

\section{Applications}
In this section, we illustrate how the description of the higher $\A^1$-homotopy sheaves of
$SL_n$ can be used to connect motivic calculations with classical topology and with concrete
classification and generation problems for algebraic vector bundles.

The \(kq\)-visible range also has a direct field-valued consequence:
Corollary~\ref{cor:dn2-visible-low-cd2} turns the quotient calculation into a
complete calculation whenever the displayed motivic-cohomology group and the
mod-\(2\) Milnor kernel vanish.  In particular,
Corollary~\ref{cor:dn2-next-low-cd2} completely determines the
\((n+1)\)-fold contraction over fields of \(2\)-cohomological dimension at
most \(2\).

We first specialize to \(k=\C\) and study complex realization on the
bigraded \(\A^1\)-homotopy sheaves appearing in the preceding diagrams.  In
both parities we obtain an unconditional comparison isomorphism.  The odd-rank
case is not a formal five-lemma argument: it uses Kervaire's computation to
show that the deeply contracted even-stage Stiefel differential is the
nonzero map \(\mathbb Z/2\to\mathbb Z/2\).  We also use the realization range
for motivic spheres to identify the three contracted groups from Section~4
with $\pi^S_4$, $\pi^S_5$, and $\pi^S_6$.

We then apply these comparison results to vector bundles on the split odd quadric
$Q_{2n-1}$.  Since $Q_{2n-1}$ is $\A^1$-equivalent to $\A^n\setminus\{0\}$ and
$Q_{2n-1}(\C)\simeq S^{2n-1}$, oriented algebraic bundles of rank $n-2$ on $Q_{2n-1}$ are governed by
the same non-stable $\A^1$-homotopy group that controls the corresponding topological classification.
As a consequence, complex realization yields an unconditional isomorphism
between algebraic and topological classification sets in both parities and
gives explicit finite classification groups.  Two ranks lower, the contracted
exact row shows that all oriented rank $n-3$ bundles
split off a trivial line bundle whenever the square-class group of the ground field
vanishes.

Over a real subfield, the weight behaves differently: real realization sends
$\G_m$ to $S^0$, so all weights with a fixed simplicial degree map to the same
classical homotopy group.  We use the explicit $\rho$- and $\nu$-classes in
the second stable stem to state a real comparison theorem for the punctured
affine-space terms, while recording why a direct analogue of the complex
comparison theorem for $SL_n$ requires additional information.

Finally, we place the calculations in the Stiefel-variety obstruction theory
of \cite{AOSS26}.  Exact Stiefel segments locate the sheaves studied in
Sections~3 and~4 among the terms entering secondary and tertiary obstruction
problems for efficient generation of projective modules.  The calculation in
Section~4 proves that the punctured-affine-space contribution to the top
tertiary obstruction is $2$-primary.  We do not identify the corresponding
two-primary Moore--Postnikov $k$-invariant; its three-primary analogue is
identified in Subsection~\ref{sec:secondary}.
\subsection{Complex realization and comparison with classical homotopy theory}\label{subsec:realization}

Assume throughout that $k=\C$.
If $(X,x)$ is a pointed smooth $k$-scheme, we write $X(\C)$ for the space of complex points
equipped with the analytic topology. The assignment
\[
X \longmapsto X(\C)
\]
extends to a (symmetric monoidal) complex realization functor on the pointed $\A^1$-homotopy
category
\[
\mathfrak{R}_{\C}:\ \mathpzc{H}_\bullet(\C)\longrightarrow \mathpzc{H}^{\mathrm{top}}_\bullet,
\]
where $\mathpzc{H}^{\mathrm{top}}_\bullet$ denotes the classical pointed topological homotopy
category; see \cite[\S3]{MV99} for details.

In particular, for every pointed space $(\mathpzc{X},x)$ and every pair of integers $(i,j)$,
complex realization induces homomorphisms
\[
{\pi}^{\A^1}_{i,j}(\mathpzc{X},x)(\C)\ \longrightarrow \pi_{i+j}\!\big(\mathpzc{X}(\C)\big).
\]

\medskip
\noindent
We now specialize to $\mathpzc{X}=SL_n$. Complex realization allows us to compare the motivic
computations of the first and second non-stable $\A^1$-homotopy sheaves of $SL_n$ with the
corresponding classical homotopy groups of the Lie group $SL_n(\C)\simeq SU(n)$.
The guiding philosophy is that the non-stable motivic behavior of $SL_n$ should mirror classical
phenomena, such as those studied by Bott \cite{Bott57} and Kervaire \cite{Kervaire60}.

\subsubsection*{Classical topological results}

We record the following computation of Kervaire.

\begin{theorem}\label{thm:Kervaire}
For $k>1$, the homotopy groups of the unitary groups satisfy
\[
\pi_{4k}\!\left(U(2k-1)\right)
\;=\; \mathbb{Z}/\big((2k)!/2\big),
\qquad
\pi_{4k+2}\!\left(U(2k)\right)
\;=\; \mathbb{Z}/2\oplus\mathbb{Z}/(2k+1)!.
\]
For $k=1$, one has
\[
\pi_{4k+2}\!\left(U(2k)\right)
\;=\;
\pi_{6}(U(2))
\;=\;
\pi_{6}(S^{1}\times S^{3})
\;=\;
\mathbb{Z}/12;
\]
the last equality is the classical computation
\(\pi_6(S^3)\cong\mathbb Z/12\).
\end{theorem}

\begin{lemma}\label{lem:Kervaire-maps}
For $k>1$, the homotopy exact sequences of the standard unitary fibrations
contain short exact sequences
\begin{equation}\label{eq:Kervaire-even-short-exact}
0\longrightarrow
\pi_{4k+3}(S^{4k+1})\cong\mathbb Z/2
\longrightarrow
\pi_{4k+2}(U(2k))
\longrightarrow
\pi_{4k+2}(U(2k+1))\cong\mathbb Z/(2k+1)!
\longrightarrow0
\end{equation}
and
\begin{equation}\label{eq:Kervaire-odd-short-exact}
0\longrightarrow
\pi_{4k}(U(2k-1))\cong\mathbb Z/\big((2k)!/2\big)
\longrightarrow
\pi_{4k}(U(2k))\cong\mathbb Z/(2k)!
\longrightarrow
\pi_{4k}(S^{4k-1})\cong\mathbb Z/2
\longrightarrow0.
\end{equation}
In particular, in the long exact sequence for
$U(2k-1)\to U(2k)\to S^{4k-1}$ the preceding boundary
\begin{equation}\label{eq:Kervaire-left-boundary-zero}
\pi_{4k+1}(S^{4k-1})\longrightarrow\pi_{4k}(U(2k-1))
\end{equation}
is zero.
\end{lemma}

\begin{proof}
These are the exact sequences used in the proof of
\cite[Lemma~I.6]{Kervaire60}.  For the last assertion, Kervaire's
Lemmas~I.4--I.5 show that the projection induces a surjection
\[
\pi_{4k+1}(U(2k))\longrightarrow\pi_{4k+1}(S^{4k-1})\cong\mathbb Z/2.
\]
Exactness therefore forces \eqref{eq:Kervaire-left-boundary-zero} to vanish,
or equivalently forces the first nonzero arrow in
\eqref{eq:Kervaire-odd-short-exact} to be injective.  The passage between
$U(r)$ and $SU(r)$ causes no change in the positive homotopy groups used
here.
\end{proof}

\begin{lemma}\label{lem:Kervaire-two-frame-boundary}
Let \(n\geq4\) be even.  Write \(W_2(\mathbb C^{n+1})\) for the space of
orthonormal complex \(2\)-frames.  The boundary homomorphism of the
two-frame Stiefel fibration
\[
S^{2n-1}\longrightarrow W_2(\mathbb C^{n+1})
\longrightarrow S^{2n+1}
\]
is an isomorphism
\[
\partial^{\mathrm{St}}_{\mathrm{top}}:
\pi_{2n+2}(S^{2n+1})\xrightarrow{\ \sim\ }
\pi_{2n+1}(S^{2n-1}).
\]
Both groups are cyclic of order \(2\).
\end{lemma}

\begin{proof}
Write \(n=2k\).  Under
\[
W_2(\mathbb C^{2k+1})=U(2k+1)/U(2k-1),
\]
the Stiefel boundary is the composite of the boundary for
\[
U(2k)\longrightarrow U(2k+1)\longrightarrow S^{4k+1}
\]
and the projection \(q:U(2k)\to S^{4k-1}\).  Kervaire
\cite[Lemma~I.5]{Kervaire60} computes
\[
q_*\partial(\eta_{4k+1})
=\Theta_{4k-1}
=\eta_{4k-1}\circ\eta_{4k},
\]
which is the nonzero element of
\(\pi_{4k+1}(S^{4k-1})\cong\mathbb Z/2\).  Since
\(\eta_{4k+1}\) generates
\(\pi_{4k+2}(S^{4k+1})\cong\mathbb Z/2\), the boundary is an
isomorphism.
\end{proof}

\subsubsection*{Motivic results}

We will use the comparison theorem for $GL_n$ from
\cite[Theorem~5.5]{AF14}.

\begin{theorem}\label{thm:AF-realization-GLn}
For any integer $n\ge 3$, the homomorphisms induced by complex realization
\[
{\pi}^{\A^1}_{n-1,n}(GL_n)(\C)\ \longrightarrow\ \pi_{2n-1}\!\big(GL_n(\C)\big)\cong \pi_{2n-1}(U(n)),
\]
and
\[
{\pi}^{\A^1}_{n-1,n+1}(GL_n)(\C)\ \longrightarrow\ \pi_{2n}\!\big(GL_n(\C)\big)\cong \pi_{2n}(U(n)),
\]
are isomorphisms.
\end{theorem}

\medskip
\noindent
We now turn to the second non-stable range for $SL_n$.

\begin{proposition}\label{prop:deep-dn1-complex}
Let \(n\geq4\) be even.  After \((n+1)\)-fold contraction and evaluation at
\(\mathbb C\), the Stiefel differential is an isomorphism
\[
(d_{n,1})_{-(n+1)}(\mathbb C):
\pi^{\A^1}_{n+1,n+1}(\A^{n+1}\setminus0)(\mathbb C)
\xrightarrow{\ \sim\ }
\pi^{\A^1}_{n,n+1}(\A^n\setminus0)(\mathbb C).
\]
Under complex realization it is identified with
\(\partial^{\mathrm{St}}_{\mathrm{top}}\) from
Lemma~\ref{lem:Kervaire-two-frame-boundary}; in particular, it is the
nonzero homomorphism \(\mathbb Z/2\to\mathbb Z/2\).
\end{proposition}

\begin{proof}
Put \(X_r=\A^r\setminus0\simeq
S_s^{r-1}\wedge\G_m^{\wedge r}\).  Naturality of complex realization and
Lemma~\ref{lem:stiefel-parity} give a commutative square
\[
\begin{tikzcd}[column sep=large]
\pi^{\A^1}_{n+1,n+1}(X_{n+1})(\mathbb C)
  \ar[r,"(d_{n,1})_{-(n+1)}"]
  \ar[d,"\sim"']
&
\pi^{\A^1}_{n,n+1}(X_n)(\mathbb C)
  \ar[d,"\sim"]
\\
\pi_{2n+2}(S^{2n+1})
  \ar[r,"\partial^{\mathrm{St}}_{\mathrm{top}}"']
&
\pi_{2n+1}(S^{2n-1}).
\end{tikzcd}
\]
For the left vertical map, apply
\cite[Proposition~4.3]{gant2026motivichomotopygroupsspheres} with
\[
(x,y,d,e)=(n,n+1,n+1,n+1).
\]
For the right vertical map use
\[
(x,y,d,e)=(n-1,n,n,n+1).
\]
The inequalities in that proposition hold for \(n\geq4\), and in both
cases \(y-1\leq e\), so realization is an isomorphism.  No Wood
compatibility is used.  The bottom horizontal morphism is an isomorphism by
Lemma~\ref{lem:Kervaire-two-frame-boundary}; hence so is the top horizontal
morphism.
\end{proof}

\begin{proposition}\label{prop:odd-rank-projection-surjective}
Let \(m\geq1\), and put \(n=2m+2\).  Then the
projection
\[
\bigl(q_{2m+1,1}\bigr)_{-(2m+3)}:
\pi^{\A^1}_{n,n+1}(SL_n)(\mathbb C)
\longrightarrow
\pi^{\A^1}_{n,n+1}(\A^n\setminus0)(\mathbb C)
\]
is surjective.
\end{proposition}

\begin{proof}
By definition, the contracted differential of Section~3 factors as
\[
(d_{n,1})_{-(n+1)}(\mathbb C)
=\bigl(q_{2m+1,1}\bigr)_{-(2m+3)}
 \circ(\delta_{n,1})_{-(n+1)}(\mathbb C).
\]
Proposition~\ref{prop:deep-dn1-complex} says that the left-hand side is an
isomorphism.  Hence
\(\bigl(q_{2m+1,1}\bigr)_{-(2m+3)}\) is surjective.
\end{proof}

\begin{theorem}\label{thm:realization-SLn}
Complex realization has the following properties.
\begin{enumerate}
\item For \(m\geq2\), the homomorphism
\[
{\pi}^{\A^1}_{2m,\,2m+2}(SL_{2m})(\C)
\ \longrightarrow\
\pi_{4m+2}\!\big(SL_{2m}(\C)\big)\ \cong\ \pi_{4m+2}\!\big(SU(2m)\big),
\]
is an isomorphism, and
\[
{\pi}^{\A^1}_{2m,\,2m+2}(SL_{2m})(\C)
\cong \Z/2\oplus\Z/(2m+1)!.
\]
\item For \(m\geq1\), the homomorphism
\[
{\pi}^{\A^1}_{2m+1,\,2m+3}(SL_{2m+1})(\C)
\ \longrightarrow\
\pi_{4m+4}\!\big(SL_{2m+1}(\C)\big)\ \cong\ \pi_{4m+4}\!\big(SU(2m+1)\big),
\]
is an isomorphism, and the motivic group is
\(\mathbb Z/((2m+2)!/2)\).
\end{enumerate}
\end{theorem}

\begin{proof}
We treat the even and odd cases separately.  To shorten the exact sequences,
write \(X_r:=\A^r\setminus\{0\}\).  Throughout, complex realization
\(\mathfrak R_{\C}\) sends the motivic fiber sequence
\begin{equation}\label{eq:SL-fiber-motivic-2}
 SL_n \longrightarrow SL_{n+1} \longrightarrow SL_{n+1}/SL_n \simeq_{\A^1} X_{n+1}
\end{equation}
(\cite[\S2]{AffineRep2}) to the classical fiber sequence
\begin{equation}\label{eq:SU-fiber-top-2}
SU(n)\longrightarrow SU(n+1)\longrightarrow S^{2n+1}.
\end{equation}
Therefore, for any bidegree $(i,j)$, realization induces a morphism between the corresponding long
exact sequences of homotopy groups. Our argument is a diagram chase in the relevant segment of
these two long exact sequences, together with the identifications of the neighboring terms.
For $i\geq1$, the determinant fiber sequences identify the positive
homotopy sheaves of $SL_N$ with those of $GL_N$, and likewise identify
the positive homotopy groups of $SU(N)$ with those of $U(N)$.
Consequently, Theorem~\ref{thm:AF-realization-GLn} applies to the
special-linear terms used below.

\medskip\noindent
\textbf{Case 1: $n=2m$ (even).}
Consider the part of the long exact sequence associated with \eqref{eq:SL-fiber-motivic-2} in
bidegree $(2m,2m+2)$:
\begin{equation}\label{eq:LES-even-motivic-2}
\begin{aligned}
 \cdots \longrightarrow {\pi}^{\A^1}_{2m+1,\,2m+2}(X_{2m+1})(\C)
&\longrightarrow
{\pi}^{\A^1}_{2m,\,2m+2}\!\big(SL_{2m}\big)(\C)
\longrightarrow
{\pi}^{\A^1}_{2m,\,2m+2}\!\big(SL_{2m+1}\big)(\C)
\\
&\longrightarrow
 {\pi}^{\A^1}_{2m,\,2m+2}(X_{2m+1})(\C)  \longrightarrow \cdots
\end{aligned}
\end{equation}
and the corresponding topological segment of the long exact sequence for \eqref{eq:SU-fiber-top-2}:
\begin{equation}\label{eq:LES-even-top-2}
\cdots \longrightarrow \pi_{4m+3}(S^{4m+1})
\longrightarrow
\pi_{4m+2}(SU(2m))
\longrightarrow
\pi_{4m+2}(SU(2m+1))
\longrightarrow
\pi_{4m+2}(S^{4m+1}) \longrightarrow \cdots
\end{equation}
Realization yields a commutative diagram with rows \eqref{eq:LES-even-motivic-2} and
\eqref{eq:LES-even-top-2}.

Using \cite[Theorem~7.2.1]{ABH} with $d=2m+1$ and $j=2m+2$, and then evaluating after $(2m+2)$-fold contraction at $\C$, we obtain a short exact sequence
\[
0 \longrightarrow (\mathbf{K}^{M}_{2m+3}/24)_{-(2m+2)}(\C)
\longrightarrow
 \pi^{\A^1}_{2m+1,\,2m+2}(X_{2m+1})(\C)
\longrightarrow
(\mathbf{GW}^{2m+1}_{2m+2})_{-(2m+2)}(\C)
\longrightarrow 0.
\]
Here the sequence is short exact because \(j=2m+2 \ge d-3=2m-2\) by \cite[Theorem~7.2.1]{ABH}. Now
\[
(\mathbf{K}^{M}_{2m+3}/24)_{-(2m+2)}(\C)=0
\]
and
\[
(\mathbf{GW}^{2m+1}_{2m+2})_{-(2m+2)}(\C)\cong \mathbf{GW}^{3}_{0}(\C)\cong \Z/2.
\]
Hence, and using that the nonzero Grothendieck--Witt class realizes to
$\eta_{\mathrm{top}}^2$ as in \cite[Example~6.3]{Gant_Williams_2025},
\[
 \pi^{\A^1}_{2m+1,\,2m+2}(X_{2m+1})(\C)\cong \Z/2.
\] 
The resulting realization map to
$\pi_{4m+3}(S^{4m+1})\cong\Z/2$ is an isomorphism.  At the other end of the
displayed segment, Morel's identification of the first non-vanishing homotopy
sheaf of punctured affine space gives
\[
 \pi^{\A^1}_{2m,\,2m+2}(X_{2m+1})(\C)
\cong \mathbf K^{MW}_{-1}(\C)
\cong \Z/2.
\]
Its generator realizes to $\eta_{\mathrm{top}}$, so it maps isomorphically to
\[
\pi_{4m+2}(S^{4m+1})\cong\Z/2.
\]
Second, the term ${\pi}^{\A^1}_{2m,\,2m+2}(SL_{2m+1})(\C)$ lies one step into the stable
range; the realization map on this term is an isomorphism by \cite[Theorem~5.5]{AF14}. Finally, Kervaire's computation \cite[\S I]{Kervaire60} gives
\[
\pi_{4m+2}(SU(2m))\cong \Z/2 \oplus \Z/(2m+1)!,
\qquad
\pi_{4m+2}(SU(2m+1))\cong \Z/(2m+1)!.
\]
Since realization gives a morphism of long exact sequences, the diagram
commutes; in particular, the map on the middle term
\(
{\pi}^{\A^1}_{2m,\,2m+2}(SL_{2m})(\C)\to \pi_{4m+2}(SU(2m))
\)
fits into a commutative diagram in which the three adjacent realization maps
just identified are isomorphisms.  Lemma~\ref{lem:Kervaire-maps}, with
$k=m$, identifies the topological segment as the short exact sequence
\[
0\longrightarrow\Z/2\longrightarrow
\pi_{4m+2}(SU(2m))\longrightarrow
\Z/(2m+1)!\longrightarrow0.
\]
The topological left arrow is injective.  Since realization on its source is
an isomorphism, the motivic left arrow is injective as well.  The topological
map following $\mathbb Z/(2m+1)!$ is zero; since realization on its source and
target is an isomorphism, the corresponding motivic map is zero.  Exactness
therefore makes the motivic segment short exact.  The short five lemma now
shows that the middle realization map is an isomorphism.

\medskip\noindent
\textbf{Case 2: $n=2m+1$ (odd), with \(m\geq1\).}
Similarly, consider the segment in bidegree $(2m+1,2m+3)$:
\begin{equation}\label{eq:LES-odd-motivic-2}
\begin{aligned}
 \cdots \longrightarrow {\pi}^{\A^1}_{2m+2,\,2m+3}(X_{2m+2})(\C)
&\longrightarrow
{\pi}^{\A^1}_{2m+1,\,2m+3}\!\big(SL_{2m+1}\big)(\C)
\longrightarrow
{\pi}^{\A^1}_{2m+1,\,2m+3}\!\big(SL_{2m+2}\big)(\C)
\\
&\longrightarrow
 {\pi}^{\A^1}_{2m+1,\,2m+3}(X_{2m+2})(\C) \longrightarrow \cdots
\end{aligned}
\end{equation}
and the corresponding topological segment
\begin{equation}\label{eq:LES-odd-top-2}
\begin{aligned}
\cdots \longrightarrow \pi_{4m+5}(S^{4m+3})
\longrightarrow
\pi_{4m+4}(SU(2m+1))
&\longrightarrow
\pi_{4m+4}(SU(2m+2))
\\
&\longrightarrow
\pi_{4m+4}(S^{4m+3}) \longrightarrow \cdots
\end{aligned}
\end{equation}
Again realization gives a commutative diagram with rows \eqref{eq:LES-odd-motivic-2} and
\eqref{eq:LES-odd-top-2}.
Denote the first displayed motivic boundary by
\[
\beta_m:
\pi^{\A^1}_{2m+2,2m+3}(X_{2m+2})(\C)
\longrightarrow
\pi^{\A^1}_{2m+1,2m+3}(SL_{2m+1})(\C).
\]

Applying \cite[Theorem~7.2.1]{ABH} with \(d=2m+2\) and \(j=2m+3\), and then evaluating the \((2m+3)\)-fold contraction at \(\C\), yields a short exact sequence
\[
0 \longrightarrow (\mathbf{K}^{M}_{2m+4}/24)_{-(2m+3)}(\C)
\longrightarrow
 \pi^{\A^1}_{2m+2,\,2m+3}(X_{2m+2})(\C)
\longrightarrow
(\mathbf{GW}^{2m+2}_{2m+3})_{-(2m+3)}(\C)
\longrightarrow 0.
\]
The short exactness follows from the inequality \(j=2m+3\ge d-3=2m\). Now
\[
(\mathbf{K}^{M}_{2m+4}/24)_{-(2m+3)}(\C)=0
\]
and
\[
(\mathbf{GW}^{2m+2}_{2m+3})_{-(2m+3)}(\C)\cong \mathbf{GW}^{3}_{0}(\C)\cong \Z/2.
\]
Therefore, again using the realization of the Grothendieck--Witt generator as
$\eta_{\mathrm{top}}^2$,
\[
 \pi^{\A^1}_{2m+2,\,2m+3}(X_{2m+2})(\C)\cong \Z/2.
\]
The realization map from this group to
$\pi_{4m+5}(S^{4m+3})\cong\Z/2$ is an isomorphism.  The right-hand punctured
affine-space term is also
\[
 \pi^{\A^1}_{2m+1,\,2m+3}(X_{2m+2})(\C)
\cong \mathbf K^{MW}_{-1}(\C)\cong\Z/2,
\]
and maps isomorphically to
$\pi_{4m+4}(S^{4m+3})\cong\Z/2$.
Moreover, the realization map
\[
{\pi}^{\A^1}_{2m+1,\,2m+3}(SL_{2m+2})(\C)\longrightarrow \pi_{4m+4}(SU(2m+2))
\]
is an isomorphism by \cite[Theorem~5.5]{AF14}. Finally, Kervaire's computation
\cite[\S I]{Kervaire60} gives
\[
\pi_{4m+4}(SU(2m+1))\cong \Z/\big((2m+2)!/2\big),
\qquad
\pi_{4m+4}(SU(2m+2))\cong \Z/(2m+2)!.
\]

Lemma~\ref{lem:Kervaire-maps}, now with $k=m+1$, identifies the topological
part as
\[
0\longrightarrow
\pi_{4m+4}(SU(2m+1))
\longrightarrow
\pi_{4m+4}(SU(2m+2))
\longrightarrow\mathbb Z/2
\longrightarrow0.
\]
In particular, the stabilization map is injective with cokernel
$\mathbb Z/2$, and the preceding topological boundary
\[
\pi_{4m+5}(S^{4m+3})\longrightarrow
\pi_{4m+4}(SU(2m+1))
\]
is zero.  Notice that this last vanishing uses Kervaire's calculation of the
projection on the preceding homotopy group; it does not follow merely from
the orders of the displayed groups.

Write \(B_M,C_M,D_M\) for the three consecutive motivic
terms beginning with the desired $SL_{2m+1}$ term, and write
$B_T,C_T,D_T$ for their topological counterparts.  Since
$C_M\to C_T$ and $D_M\to D_T$ are isomorphisms, they identify
\[
\ker(C_M\longrightarrow D_M)
\quad\text{with}\quad
\ker(C_T\longrightarrow D_T)
=\operatorname{im}(B_T\longrightarrow C_T)\cong B_T.
\]
Exactness of the motivic row therefore makes $B_M\to B_T$ surjective.  If an
element of $B_M$ realizes to zero, its image in $C_M$ is zero because
$C_M\to C_T$ is injective; exactness then puts it in
$\operatorname{im}(\beta_m)$.  Conversely,
$\operatorname{im}(\beta_m)$ realizes to zero because the corresponding
topological boundary is zero.  Hence
\[
\ker(B_M\longrightarrow B_T)=\operatorname{im}(\beta_m).
\]
Exactness one term earlier gives
\[
\ker(\beta_m)=
\operatorname{im}\!\bigl((q_{2m+1,1})_{-(2m+3)}\bigr),
\]
where the source of \(\beta_m\) was identified above with \(\mathbb Z/2\).
Thus there is an exact sequence
\[
\mathbb Z/2\xrightarrow{\ \beta_m\ }B_M
\longrightarrow B_T\longrightarrow0.
\]
Proposition~\ref{prop:odd-rank-projection-surjective} makes the preceding
projection surjective.  Hence
\(\ker(\beta_m)=\mathbb Z/2\), so \(\beta_m=0\).  The realization map is
therefore injective as well as surjective, and the topological calculation
identifies the motivic group with
\(\mathbb Z/((2m+2)!/2)\).
\end{proof}

\begin{remark}
In even rank the topological sphere boundary is injective and the relevant
portion of the long exact sequence is short exact, so the two end
isomorphisms determine the middle term.  In odd rank the corresponding
topological boundary is zero.  Commutativity then gives only
$\mathfrak R_{\C}\circ\beta_m=0$ and does not by itself imply
\(\beta_m=0\).  A direct five-lemma argument would require control of the
preceding realization map
\[
\pi^{\A^1}_{2m+2,\,2m+3}(SL_{2m+2})(\C)
\longrightarrow \pi_{4m+5}(SU(2m+2)),
\]
for example its surjectivity.

Theorem~\ref{thm:AF-realization-GLn} alone does not control this preceding
group.  Likewise, the general Stiefel-variety comparison
of \cite[Theorem~1.5]{gant2026motivichomotopygroupsspheres} assumes
$r\leq n-2$ for $V_r(\A^n)$.  The identifications
\[
V_{n-1}(\A^n)\cong SL_n,
\qquad V_n(\A^n)=GL_n
\]
show that the special linear group is exactly the first excluded boundary.
The proof above bypasses that missing comparison: it applies the sphere
realization theorem only to the source and target of the adjacent even-stage
Stiefel differential, and then uses Kervaire's explicit computation of that
classical boundary.  Thus the odd-rank realization argument is independent
of Corollary~\ref{prop:suslin-stiefel-wood}, and supplies a separate
topological consistency check.
\end{remark}

The motivic-sphere comparison also applies directly to all three highly
contracted punctured affine-space groups calculated in
Section~\ref{sec:third-unstable}.

\begin{theorem}\label{thm:complex-realization-contracted-second-stem}
Let $n\geq5$ and let $\delta\in\{2,3,4\}$.  Complex realization induces an
isomorphism
\[
\pi^{\A^1}_{n+1,\,n+\delta}(\A^n\setminus0)(\C)
\xrightarrow{\ \sim\ }
\pi_{2n+\delta+1}(S^{2n-1}).
\]
Under the identifications of
Theorem~\ref{thm:Pn-second-stem}, these groups are
\[
\begin{array}{c|c|c}
\delta & \text{motivic group} & \text{stable topological group}\\ \hline
2 & \mathbf K^M_2(\C)/2=0 & \pi^S_4=0,\\
3 & \mathbf K^M_1(\C)/2=0 & \pi^S_5=0,\\
4 & \mathbf K^M_0(\C)/2\cong\mathbb Z/2 &
\pi^S_6\cong\mathbb Z/2.
\end{array}
\]
In particular, the terminal class for $\delta=4$ in the contracted
calculation realizes to the nonzero class $\nu^2\in\pi^S_6$.
\end{theorem}

\begin{proof}
Write
\(\A^n\setminus0\simeq S_s^{n-1}\wedge\G_m^{\wedge n}\) and apply
\cite[Proposition~4.3]{gant2026motivichomotopygroupsspheres} with
\[
(x,y,d,e)=(n-1,n,n+1,n+\delta).
\]
The relevant inequalities reduce to
\[
n+1\leq2n-4,\qquad
\delta\leq4,\qquad
n-1\leq n+\delta,
\]
and hence hold for $n\geq5$ and $2\leq\delta\leq4$.  This proves that
realization is an isomorphism.  The motivic groups are the three contractions
computed in Theorem~\ref{thm:Pn-second-stem}; evaluating positive Milnor
$K$-theory modulo $2$ at $\C$ gives zero.  Finally, the indicated topological
groups are already stable in this range, and
$\pi^S_4=\pi^S_5=0$, $\pi^S_6\cong\mathbb Z/2$.
\end{proof}

\subsection{Real realization and collapse of the weight}\label{subsec:real-realization}

Let $k\hookrightarrow\mathbb R$ be a fixed real embedding.  Taking real points
defines a realization functor
\[
\mathfrak R_{\mathbb R}:\mathpzc H_\bullet(k)
\longrightarrow \mathpzc H_\bullet^{\mathrm{top}}.
\]
The crucial difference from complex realization is
\[
\mathfrak R_{\mathbb R}(S_s^i\wedge\G_m^{\wedge j})
\simeq S^i,
\]
because $\mathbb R^\times$, pointed at $1$, is homotopy equivalent to $S^0$.
Consequently, real realization induces maps
\begin{equation}\label{eq:real-realization-bigraded}
\pi_{i,j}^{\A^1}(X)(k)\longrightarrow
\pi_i\bigl(\mathfrak R_{\mathbb R}X\bigr),
\end{equation}
and all weights with a fixed simplicial degree have the same topological
target.  In particular,
\[
\mathfrak R_{\mathbb R}(\A^n\setminus\{0\})\simeq S^{n-1},
\qquad
\mathfrak R_{\mathbb R}(SL_n)\simeq SO(n)
\]
in positive homotopy degrees.

The weight collapse makes a general real comparison theorem more subtle than
Theorem~\ref{thm:realization-SLn}.  On the other hand, the highly contracted
punctured affine-space terms from Section~\ref{sec:third-unstable} admit a
particularly clean description.  Recall that $\rho=\{-1\}$ and that $\nu$
denotes the motivic Hopf element on the relevant stem.  Under real
realization,
\[
\mathfrak R_{\mathbb R}(\rho)=1,
\qquad
\mathfrak R_{\mathbb R}(\nu)=\eta_{\mathrm{top}}.
\]

The next comparison follows from \cite[Example~6.6]{Gant_Williams_2025}.

\begin{theorem}\label{thm:real-punctured-comparison}
Let \(k\) be a field equipped with an embedding \(k\hookrightarrow\mathbb R\),
let $n\geq 5$, let $d\in\{2,3,4\}$, and suppose
\[
\mathbf K^M_{4-d}(k)/2\cong\mathbb Z/2.
\]
Then real realization induces an isomorphism
\[
\pi_{n+1,n+d}^{\A^1}(\A^n\setminus\{0\})(k)
\xrightarrow{\ \sim\ }
\pi_{n+1}(S^{n-1})\cong\mathbb Z/2.
\]
The motivic generators $\rho^2\nu^2$, $\rho\nu^2$, and $\nu^2$ for
$d=2,3,4$, respectively, all realize to the nonzero class
$\eta_{\mathrm{top}}^2$.
\end{theorem}

\begin{proof}
Theorem~\ref{thm:Pn-second-stem} identifies the source with
$\mathbf K^M_{4-d}(k)/2$.  Under this identification,
\cite[Example~6.6]{Gant_Williams_2025} identifies the classes represented
by $\rho^{4-d}\nu^2$ and shows that their real realizations are
$\eta_{\mathrm{top}}^2$.  The fixed real embedding sends
$\rho^{4-d}$ to the nonzero element detected by the corresponding ordering.
Hence the hypothesis $\mathbf K^M_{4-d}(k)/2\cong\mathbb Z/2$ makes this
class a generator.  Since $n\geq5$, the target is in the stable range and
$\pi_{n+1}(S^{n-1})\cong\pi_2^S\cong\mathbb Z/2$, generated by
$\eta_{\mathrm{top}}^2$.  The realization homomorphism is therefore an
isomorphism.
\end{proof}

\begin{corollary}\label{cor:real-dn2-zero}
Let $n\geq5$ and $k=\mathbb R$.  After evaluating at $\mathbb R$, the maps
\[
(d_{n,2})_{-(n+3)}:\mathbf K^M_2/2\longrightarrow\mathbf K^M_1/2,
\qquad
(d_{n,2})_{-(n+4)}:\mathbf K^M_1/2\longrightarrow\mathbf K^M_0/2
\]
are zero maps between groups isomorphic to $\mathbb Z/2$.  Real realization
identifies their sources and targets with the corresponding groups
\[
\pi_{n+2}(S^n),\quad \pi_{n+1}(S^{n-1}),
\]
but the induced homomorphisms remain zero.
\end{corollary}

\begin{proof}
For $\mathbb R$, the groups $\mathbf K_q^M(\mathbb R)/2$ are
$\mathbb Z/2$ for $q=0,1,2$.  Apply
Theorems~\ref{thm:real-punctured-comparison}
and~\ref{thm:dn2-contracted-zero}.
\end{proof}

\begin{remark}\label{rem:real-SL-strategy}
For $SL_n$, one obtains a commutative diagram between the contracted exact
rows of Corollary~\ref{cor:third-unstable-contracted-rows} and the
long exact homotopy sequences of
\[
SO(n-1)\longrightarrow SO(n)\longrightarrow S^{n-1}.
\]
The punctured affine-space vertical maps are controlled by
Theorem~\ref{thm:real-punctured-comparison}.  Thus a real analogue of
Theorem~\ref{thm:realization-SLn} reduces to determining the neighboring
realization maps for the orthogonal stable and first non-stable terms.  We do
not assert an isomorphism in all weights: the collapse
\eqref{eq:real-realization-bigraded} means that such a statement cannot be
deduced by copying the complex argument weight for weight.
\end{remark}

\subsection{Vector bundles on split odd quadrics}\label{subsec:quadric-comparison}

Let $k$ be an infinite perfect field of characteristic different from $2$ and
let
\[
Q_{2n-1}=\left\{\sum_{i=1}^n x_i y_i=1\right\}.
\]
The projection onto the $x_i$-coordinates gives an $\A^1$-weak equivalence
\[
Q_{2n-1}\simeq_{\A^1}\A^n\setminus\{0\}
\simeq S_s^{n-1}\wedge\G_m^{\wedge n}.
\]
By affine representability and the fact that $BSL_r$ is
$\A^1$-simply connected, oriented rank $r$ vector bundles satisfy
\begin{equation}\label{eq:quadric-general-classification}
\begin{aligned}
\mathrm{Vect}^{\mathrm{alg},o}_r(Q_{2n-1})
&\cong [Q_{2n-1},BSL_r]_{\A^1}\\
&\cong \pi^{\A^1}_{n-1,n}(BSL_r)(k)
\cong \pi^{\A^1}_{n-2,n}(SL_r)(k).
\end{aligned}
\end{equation}
Thus decreasing the rank on a fixed quadric moves successively through the
non-stable layers:
\[
\begin{array}{c|c}
\text{rank on }Q_{2n-1} & \text{controlling sheaf}\\ \hline
n-1 & \pi^{\A^1}_{n-2,n}(SL_{n-1}),\\
n-2 & \pi^{\A^1}_{n-2,n}(SL_{n-2}),\\
n-3 & \pi^{\A^1}_{n-2,n}(SL_{n-3}).
\end{array}
\]
The three rows correspond to the first, second, and third non-stable layers,
respectively.

\subsubsection*{Critical rank: the Asok--Fasel classification}

For comparison, we recall the complete answer in rank $n-1$ from
\cite[Theorems~4.8 and~4.10]{AF14}.

\begin{theorem}\label{thm:AF-critical-rank-quadric}
There are canonical bijections
\[
\mathrm{Vect}^{\mathrm{alg},o}_{n-1}(Q_{2n-1})
\cong
\begin{cases}
\mathbb Z/(n-1)!, & n \text{ even},\\[3pt]
\mathbb Z/(n-1)!\times_{\mathbb Z/2}W(k), & n \text{ odd}.
\end{cases}
\]
In the odd case the two maps to $\mathbb Z/2$ are reduction modulo $2$ and
the rank homomorphism on the Witt group.
\end{theorem}

This result explains two features that recur in our calculations.  The
factorial quotient is detected by complex realization and Bott periodicity,
while the Witt factor records information visible under real realization.

\subsubsection*{One rank lower: complex comparison}

Now assume $k=\mathbb C$.  Since
$Q_{2n-1}(\mathbb C)\simeq S^{2n-1}$, complex realization gives a comparison
\[
\Re_{\mathbb C}:
[Q_{2n-1},BSL_{n-2}]_{\A^1}
\longrightarrow
[S^{2n-1},BSU(n-2)]_{\mathrm{top}}.
\]

\begin{theorem}\label{thm:comparison-quadric}
Let $n\geq 5$.
The preceding realization map is an isomorphism.
\end{theorem}

\begin{proof}
Using \eqref{eq:quadric-general-classification} and topological clutching, the
comparison map is identified with
\[
\pi^{\A^1}_{n-2,n}(SL_{n-2})(\mathbb C)
\longrightarrow
\pi_{2n-2}(SU(n-2)).
\]
Set $r=n-2$.  Then the source is
$\pi^{\A^1}_{r,r+2}(SL_r)(\mathbb C)$, and $r\geq3$.  If $r=2m$, then
\(n\) is even, so \(r\geq4\), and
Theorem~\ref{thm:realization-SLn}\textup{(1)} applies.  If \(r=2m+1\),
part~\textup{(2)} applies, including the boundary case \(r=3\).
\end{proof}

\begin{corollary}\label{cor:rank-n-2-explicit}
For even \(n\geq6\),
\[
\mathrm{Vect}^{\mathrm{alg},o}_{n-2}(Q_{2n-1})
\cong\mathbb Z/2\oplus\mathbb Z/(n-1)!.
\]
For odd \(n\geq5\),
\[
\mathrm{Vect}^{\mathrm{alg},o}_{n-2}(Q_{2n-1})
\cong\mathbb Z/\big((n-1)!/2\big).
\]
\end{corollary}

\begin{proof}
Put \(r=n-2\) in Theorem~\ref{thm:realization-SLn}.  For even \(n\), the
integer \(r\) is even and part~\textup{(1)} applies.  For odd \(n\), it is
odd and part~\textup{(2)} gives the asserted cyclic group.
\end{proof}

\subsubsection*{Two ranks lower: an exact row and a cross-composite}
\label{subsec:quadric-lower-ranks}

The rank $n-3$ problem is the first direct geometric application of
Section~\ref{sec:third-unstable}.  Put $r=n-3$.  Then
\[
\mathrm{Vect}^{\mathrm{alg},o}_{n-3}(Q_{2n-1})
\cong \pi^{\A^1}_{r+1,r+3}(SL_r)(k).
\]
Corollary~\ref{cor:third-unstable-contracted-rows} therefore gives the
following concrete reduction.

\begin{proposition}\label{prop:rank-n-3-quadric-exact}
Let $k$ be a field of characteristic $0$ and let $n\geq8$.  With $r=n-3$,
there is an exact sequence
\begin{equation}\label{eq:rank-n-3-quadric-exact}
\begin{aligned}
\pi^{\A^1}_{r+1,r+3}(SL_{r-1})(k)
&\longrightarrow
\mathrm{Vect}^{\mathrm{alg},o}_{n-3}(Q_{2n-1})
\longrightarrow
\mathbf K^M_1(k)/2\\
&\longrightarrow
\pi^{\A^1}_{r,r+3}(SL_{r-1})(k)
\longrightarrow
\pi^{\A^1}_{r,r+3}(SL_r)(k).
\end{aligned}
\end{equation}
There is also a complex
\begin{equation}\label{eq:rank-n-3-quadric-cross-complex}
\mathbf K^M_2(k)/2
\longrightarrow
\mathrm{Vect}^{\mathrm{alg},o}_{n-3}(Q_{2n-1})
\longrightarrow
\mathbf K^M_1(k)/2,
\end{equation}
whose first arrow has image contained in the image of the first arrow of
\eqref{eq:rank-n-3-quadric-exact}.  Thus the punctured affine-space quotient
in this rank is controlled by the square-class group
$k^\times/(k^\times)^2$, but the left stabilization image remains part of the
classification problem.  If \(n\) is even, then \(r=n-3\) is odd and
\eqref{eq:rank-n-3-quadric-cross-complex} is the indicated contraction of an
already zero uncontracted cross-composite.
\end{proposition}

\begin{proof}
Apply \eqref{eq:third-lower-row-weight-next-three} with its index $n$ replaced by
$r=n-3$, evaluate at $k$, and use
\eqref{eq:quadric-general-classification}.  The additional complex and its
factorization property are the corresponding assertions of
Corollary~\ref{cor:third-unstable-contracted-rows}.  The final assertion is
Proposition~\ref{prop:dn2-odd-zero} applied to \(r\).
\end{proof}

\begin{corollary}\label{cor:rank-n-3-quadric-splitting}
Let $k$ be a field of characteristic $0$ such that
$k^\times=(k^\times)^2$, and let $n\geq8$.  Stabilization induces a
surjection
\[
\mathrm{Vect}^{\mathrm{alg},o}_{n-4}(Q_{2n-1})
\longrightarrow
\mathrm{Vect}^{\mathrm{alg},o}_{n-3}(Q_{2n-1}),
\qquad [\mathcal E]\longmapsto[\mathcal E\oplus\mathcal O].
\]
Consequently, every oriented rank $n-3$ vector bundle on $Q_{2n-1}$ splits
off a trivial line bundle.  In particular, this holds over every
quadratically closed field of characteristic $0$.
\end{corollary}

\begin{proof}
For $r=n-3$, the source of the first arrow in
\eqref{eq:rank-n-3-quadric-exact} is
\[
\pi^{\A^1}_{n-2,n}(SL_{n-4})(k)
\cong\mathrm{Vect}^{\mathrm{alg},o}_{n-4}(Q_{2n-1}),
\]
and its map to the middle term is induced by
$SL_{n-4}\hookrightarrow SL_{n-3}$, hence by adding a trivial line bundle.  The
next term in the exact sequence is
\[
\mathbf K^M_1(k)/2=k^\times/(k^\times)^2=0.
\]
Exactness therefore proves the asserted surjectivity.
\end{proof}

\begin{remark}
Proposition~\ref{prop:rank-n-3-quadric-exact} is a reduction, rather than a
complete classification: the source and the next term lie one rank lower in
the special linear tower.  A complete answer requires those stabilization
groups.  The adjacent sphere-level cross-composite is already determined by
Proposition~\ref{prop:dn2-stable-eta} and
Corollary~\ref{cor:dn2-kq}.
Corollary~\ref{cor:rank-n-3-quadric-splitting}
extracts the geometric consequence when the Milnor term vanishes.  Over real
closed fields that term is $\mathbb Z/2$ and is detected by
Theorem~\ref{thm:real-punctured-comparison}.
\end{remark}

\subsection{Stiefel varieties and efficient generation of projective modules}
\label{subsec:efficient-generation}

We now explain the connection with \cite{AOSS26}.  Let $X=\operatorname{Spec}R$
be smooth affine, let $M$ be a projective $R$-module of rank $r$, and write
\[
\operatorname{St}_r(N):=GL_N/GL_{N-r}.
\]
The fiber sequence
\begin{equation}\label{eq:Stiefel-Grassmann-fiber}
\operatorname{St}_r(N)\longrightarrow Gr_r(N)\longrightarrow BGL_r
\end{equation}
has the following concrete meaning: $M$ can be generated by $N$ elements if
and only if its classifying map $X\to BGL_r$ lifts to $Gr_r(N)$.  The
Moore--Postnikov tower of \eqref{eq:Stiefel-Grassmann-fiber} produces
obstructions
\begin{equation}\label{eq:generation-obstructions}
o_{i,N,r}(M)\in
H^{i+1}_{\mathrm{Nis}}\!\left(
X,\pi_i^{\A^1}(\operatorname{St}_r(N))(\det M)
\right).
\end{equation}
If the $\A^1$-cohomological dimension of $X$ is at most $d$, only the
coefficients with $i<d$ can contribute.

The first coefficient is Milnor--Witt or Milnor $K$-theory according to
parity.  Its characteristic class is a quadratic refinement of the top Segre
class.  Asok--Opie--Shin--Syed use this observation to recover the
Forster--Swan bound $r+d$ and to study the sharper bounds $r+d-1$ and
$r+d-2$.

\begin{theorem}\label{thm:AOSS-generation}
Let \(X=\operatorname{Spec}R\) be smooth affine over a perfect field $k$,
and let \(M\) be a projective \(R\)-module of rank \(r\geq2\).
\begin{enumerate}
\item If $k=\mathbb R$ and \(X\) has dimension \(d\geq3\), then \(M\)
can be generated by $r+d-1$ elements if and only if its top Segre class
$s_d(M)$ vanishes.
\item Suppose $X$ has $\A^1$-cohomological dimension at most $d\geq4$, $M$ is
generated by $r+d-1$ elements, and
\[
H^d_{\mathrm{Nis}}\!\left(X,
\pi_{d-1}^{\A^1}(\operatorname{St}_r(r+d-2))(L)\right)=0
\]
for every line bundle $L$.  If $d$ is odd, $M$ is generated by $r+d-2$
elements precisely when $s_{d-1}(M)=0$.  If $d$ is even, the corresponding
criterion is the vanishing of the Euler class of a rank $d-1$ complementary
module.
\item If $k$ is algebraically closed of characteristic $0$ and \(X\) has
dimension \(d\geq4\), then $M$ is generated by $r+d-2$
elements if and only if
\[
s_d(M)=s_{d-1}(M)=0.
\]
\end{enumerate}
\end{theorem}

The last assertion gives a two-Segre-class criterion over algebraically
closed fields.  Our calculation naturally points one step
further, from $r+d-2$ to $r+d-3$ generators.  The following exact segments
make that relationship precise.

\begin{proposition}\label{prop:calculation-obstruction-dictionary}
Let $d\geq4$ and $r\geq2$.  Forgetting the last vector in a frame gives exact
segments
\[
\begin{aligned}
\pi_{d-1}^{\A^1}(\A^{d-1}\setminus0)
&\longrightarrow
\pi_{d-1}^{\A^1}(\operatorname{St}_r(r+d-2))\\
&\longrightarrow
\pi_{d-1}^{\A^1}(\operatorname{St}_{r-1}(r+d-2))
\longrightarrow \mathbf K^{MW}_{d-1}
\end{aligned}
\]
and
\[
\begin{aligned}
\mathcal P_{d-2}
=\pi_{d-1}^{\A^1}(\A^{d-2}\setminus0)
&\longrightarrow
\pi_{d-1}^{\A^1}(\operatorname{St}_r(r+d-3))\\
&\longrightarrow
\pi_{d-1}^{\A^1}(\operatorname{St}_{r-1}(r+d-3))
\longrightarrow
\pi_{d-2}^{\A^1}(\A^{d-2}\setminus0).
\end{aligned}
\]
The arrows are equivariant for the natural \(\G_m\)-actions coming from the
determinant, so both segments remain exact after twisting by a line bundle.
Thus the first non-stable punctured-affine-space sheaf studied in Section~3
enters the $r+d-2$ coefficient.  At the next stage, the second-stem sheaf
calculated in Section~\ref{sec:third-unstable} maps directly to the
$r+d-3$ coefficient, whose following boundary lands in the first non-stable
sheaf studied in Section~3.  Thus both adjacent non-stable layers occur in one
exact Stiefel segment.
For Stiefel ranks $2$ and $3$, the maps between adjacent punctured affine
spaces obtained by iterating these sequences are the linear spectral-sequence
differentials discussed in \cite[Lemmas~18--19 and Remark~20]{AOSS26}.
In the indexing of Remark~\ref{rem:dn2-linear-spectral-sequence}, the
\(q=0\) differential is \(d_{n,1}\); the \(q=1\) differential \(d_{n,2}\)
is the next map studied in Section~4.
\end{proposition}

\begin{proof}
Use the fiber sequences
\[
\A^{j+1}\setminus\{0\}\longrightarrow
\operatorname{St}_{N-j}(N)\longrightarrow
\operatorname{St}_{N-j-1}(N)
\]
and their long exact homotopy sequences; compare
\cite[Propositions~2 and~15]{AOSS26}.  Taking
$(N,j)=(r+d-2,d-2)$ and $(r+d-3,d-3)$ gives the two displayed
segments.  Equivariance follows from the change-of-frame action; as in the
obstruction tower of \cite{AOSS26}, this action factors through the
determinant.  Morel's connectivity theorem identifies the right-hand
punctured-affine-space term in the first segment with
$\mathbf K^{MW}_{d-1}$; the right-hand term in the second segment lies one
degree beyond the connectivity range and is the first non-stable sheaf of
Section~3.  The final assertion is precisely the
description in \cite[Lemmas~18--19 and Remark~20]{AOSS26}.
\end{proof}

The calculation of Section~\ref{sec:third-unstable} gives more than a
dictionary of coefficient sheaves: it constrains the new top obstruction.

\begin{proposition}\label{prop:tertiary-two-primary}
Assume that $k$ has characteristic $0$, let $d\geq7$, and put
\[
\mathcal F_{r,d}:=
\pi_{d-1}^{\A^1}(\operatorname{St}_r(r+d-3)).
\]
For every smooth affine $d$-fold $X$ over $k$ and every line bundle $L$ on
$X$, the subgroup
\[
\operatorname{im}\!\left(
H^d_{\mathrm{Nis}}(X,\mathcal P_{d-2}(L))
\longrightarrow
H^d_{\mathrm{Nis}}(X,\mathcal F_{r,d}(L))
\right)
\]
is annihilated by $2$.  In the Rost--Schmid complex, its classes are
represented in codimension \(d\) by line-bundle-twisted Milnor symbols modulo \(2\) in
\[
\mathbf K^M_2(k(x))/2,
\qquad x\in X^{(d)}.
\]
Consequently, the part of the $r+d-3$ generator obstruction induced by the
punctured affine-space term in
Proposition~\ref{prop:calculation-obstruction-dictionary} is intrinsically
$2$-primary.
\end{proposition}

\begin{proof}
Apply Theorem~\ref{thm:Pn-second-stem} with $n=d-2$.  Since $d\geq7$, its
range hypothesis holds, and the first of its three contraction formulas gives
\[
(\mathcal P_{d-2})_{-d}\cong\mathbf K^M_2/2.
\]
The degree-$d$ term of the Rost--Schmid complex computing
$H^d_{\mathrm{Nis}}(X,\mathcal P_{d-2}(L))$ is the direct sum over
$x\in X^{(d)}$ of the corresponding line bundle-twisted contractions.  Every such
summand is annihilated by $2$, and top cohomology is a quotient of this
degree-$d$ term.  Hence $H^d_{\mathrm{Nis}}(X,\mathcal P_{d-2}(L))$, and
therefore its image in the displayed Stiefel-coefficient cohomology group, is
annihilated by $2$.
\end{proof}

The \(2\)-primary conclusion becomes an actual vanishing theorem when
\(\operatorname{cd}_2(k)\leq1\).

\begin{theorem}\label{thm:P-top-cd2-vanishing}
Let \(k\) be a field of characteristic \(0\) with
\(\operatorname{cd}_2(k)\leq1\), let \(d\geq7\), and let \(X\) be a smooth
pure \(d\)-dimensional \(k\)-scheme.  For every line bundle \(L\) on \(X\),
\[
H^d_{\mathrm{Nis}}\!\left(X,\mathcal P_{d-2}(L)\right)=0.
\]
Consequently, the punctured-affine-space contribution to the top
\(r+d-3\) Stiefel obstruction in
Proposition~\ref{prop:tertiary-two-primary} vanishes, rather than merely being
annihilated by \(2\).
\end{theorem}

\begin{proof}
As in the proof of Proposition~\ref{prop:tertiary-two-primary},
\[
(\mathcal P_{d-2})_{-d}\cong\mathbf K^M_2/2.
\]
The degree-\(d\) Rost--Schmid cochain group is a direct sum indexed by
\(x\in X^{(d)}\); each summand is a line bundle twist of
\(\mathbf K^M_2(k(x))/2\).  The field \(k(x)\) is finite over \(k\), hence
\(\operatorname{cd}_2(k(x))\leq1\).  The norm-residue isomorphism
\cite{OVV07} gives
\[
\mathbf K^M_2(k(x))/2
\cong H^2_{\mathrm{\acute et}}
   (k(x),\boldsymbol\mu_2^{\otimes2})=0.
\]
Thus the entire degree-\(d\) cochain group is zero, and so is its quotient
\(H^d_{\mathrm{Nis}}(X,\mathcal P_{d-2}(L))\).
\end{proof}

\subsection{Secondary obstruction calculations}\label{sec:secondary}

For $n\geq2$, set
\[
\mathcal Q_n:=\pi_n^{\A^1}(\A^n\setminus0).
\]
Thus the first non-stable coefficient in the corank-two problem is
$\mathcal Q_{d-2}$.

\begin{proposition}\label{prop:Q-contracted-Gersten-tail}
Let $k$ be a field of characteristic $0$ and let $n\geq4$.  There is a
natural short exact sequence
\begin{equation}\label{eq:Q-penultimate-contracted-extension}
0\longrightarrow\mathbf K^M_2/24
\longrightarrow(\mathcal Q_n)_{-n}
\longrightarrow\mathbf{GW}^0_1\longrightarrow0,
\end{equation}
a natural short exact sequence
\begin{equation}\label{eq:Q-first-contracted-extension}
0\longrightarrow\mathbf K^M_1/24
\longrightarrow(\mathcal Q_n)_{-(n+1)}
\longrightarrow\mathbf{GW}^3_0\longrightarrow0,
\end{equation}
and a natural isomorphism
\begin{equation}\label{eq:Q-terminal-contraction}
(\mathcal Q_n)_{-(n+2)}\cong\mathbf K^M_0/24
\cong\underline{\Z/24}.
\end{equation}
On fields, $\mathbf{GW}^3_0$ is constant with value $\Z/2$.
For every field extension \(F/k\), determinant and spinor norm identify
\begin{equation}\label{eq:GW10-field-explicit}
\mathbf{GW}^0_1(F)
\cong
\mathbf K^M_1(F)/2\oplus\mathbb Z/2.
\end{equation}

If $X$ is smooth and pure of dimension $n+2$ and $L$ is a line bundle on
$X$, the last three terms of the Rost--Schmid complex computing
$H^{n+1}_{\mathrm{Nis}}(X,\mathcal Q_n(L))$ have the form
\begin{equation}\label{eq:Q-Gersten-tail}
\begin{aligned}
\bigoplus_{x\in X^{(n)}}
  (\mathcal Q_n)_{-n}(k(x);L_x)
&\longrightarrow
\bigoplus_{x\in X^{(n+1)}}
  (\mathcal Q_n)_{-(n+1)}(k(x);L_x)\\
&\longrightarrow
\bigoplus_{x\in X^{(n+2)}}
  \underline{\Z/24}(k(x);L_x).
\end{aligned}
\end{equation}
Here $L_x$ includes the usual orientation line bundle factor in the Rost--Schmid
complex; after choosing a basis, every summand in the last direct sum is
abstractly $\Z/24$.
\end{proposition}

\begin{proof}
The exact sequence for $\mathcal Q_n$ in
\cite[Example~6.3]{Gant_Williams_2025}, obtained from the
$\mathbb P^1$-stabilization theorem of \cite{ABH} and the first stable stem,
becomes short exact after the indicated contractions.  Contracting \(n\)
times and using geometric Bott periodicity gives
\eqref{eq:Q-penultimate-contracted-extension}; one further contraction
gives \eqref{eq:Q-first-contracted-extension}.  The identification of
$\mathbf{GW}^3_0$ on fields is part of the same calculation.  One final
contraction gives \eqref{eq:Q-terminal-contraction}, as recorded in
\cite[Example~6.5]{Gant_Williams_2025}.  The description
\eqref{eq:Q-Gersten-tail} is then the degree-$n$, $n+1$, and $n+2$ part of
the Rost--Schmid complex.  Finally, the classical description of stable
orthogonal \(K_1\), recalled in \cite[Subsection~2.3]{AF17}, gives
\eqref{eq:GW10-field-explicit}: its two invariants are spinor norm and
determinant.
\end{proof}

\begin{proposition}\label{prop:Q-explicit-first-stem-residues}
Keep the hypotheses and notation of
Proposition~\ref{prop:Q-contracted-Gersten-tail}.  Under the stabilization
identifications used there, put
\[
E_i(F):=(\mathcal Q_n)_{-(n+i)}(F)
       \cong\pi_{i+1,i}\mathds1(F),
\qquad 0\leq i\leq2,
\]
for a field extension \(F/k\), and write
\[
\alpha:=\eta_{\mathrm{top}}\in E_0(F),\qquad
\beta:=\eta\eta_{\mathrm{top}}\in E_1(F),\qquad
\gamma:=\nu\in E_2(F).
\]
Then
\begin{equation}\label{eq:first-stem-E1-splitting}
E_2(F)=\mathbb Z/24\{\gamma\},
\qquad
E_1(F)=
\mathbf K^M_1(F)/24\{\gamma\}
\oplus\mathbb Z/2\{\beta\}.
\end{equation}
The first group \(E_0(F)\) is generated by
\[
\alpha,\qquad [u]\beta,
\qquad [u][w]\gamma
\quad (u,w\in F^\times).
\]
Normalize the spinor-norm and determinant coordinates in
\eqref{eq:GW10-field-explicit} by requiring that the quotient map send
\[
\alpha\longmapsto(0,1),
\qquad
[u]\beta\longmapsto(\{u\}\bmod2,0).
\]
The gluing in the generally non-split first extension is governed by
\begin{align}
[uv]\beta
 &= [u]\beta+[v]\beta+12[u][v]\gamma,
 \label{eq:first-stem-factor-set}\\
\langle a\rangle\alpha
 &=\alpha+[a]\beta,
&
\langle a\rangle\beta
 &=\beta+12[a]\gamma.
\label{eq:first-stem-line-bundle-action}
\end{align}

Let \(v\) be a geometric discrete valuation of \(F/k\), choose a
uniformizer \(\pi\), and use \(\pi\) to trivialize the normal line bundle factor in
the Rost--Schmid residue.  Denote the resulting coordinate residue by
\(\partial_v^\pi\).  If \(u=\pi^m a\), where
\(m\in\mathbb Z\) and \(a\in\mathcal O_v^\times\), then
\begin{align}
\partial_v^\pi(\alpha)
 &=\partial_v^\pi(\beta)
 =\partial_v^\pi(\gamma)=0,
 \label{eq:first-stem-unramified-residues}\\
\partial_v^\pi([u]\gamma)
 &=m\gamma,
 \label{eq:first-stem-K1-residue}\\
\partial_v^\pi([u][w]\gamma)
 &=\partial_v^M\{u,w\}\,\gamma,
 \label{eq:first-stem-K2-residue}\\
\partial_v^\pi([u]\beta)
 &=
 \begin{cases}
 0,&m\ \text{even},\\[2pt]
 \beta+12[\bar a]\gamma,&m\ \text{odd}.
 \end{cases}
 \label{eq:first-stem-hermitian-residue}
\end{align}
Here \(\partial_v^M\) is the usual Milnor \(K\)-theory residue.

The finite corestrictions that enter a Rost--Schmid differential are
explicit as well.  If \(L/K\) is a finite separable extension of degree
\(d\), let
\[
q_{L/K}:=\operatorname{Tr}^{GW}_{L/K}\langle1\rangle\in GW(K)
\]
be its trace form and let \(\Delta_{L/K}\in K^\times/K^{\times2}\) be the
determinant of that form.  Writing \(T_{L/K}\) for corestriction, one has
\begin{align}
T_{L/K}(\gamma_L)
 &=d\gamma_K,
&
T_{L/K}([u]\gamma_L)
 &=N^M_{L/K}(\{u\})\gamma_K,
\label{eq:first-stem-milnor-transfers}\\
T_{L/K}(\beta_L)
 &=d\beta_K+12[\Delta_{L/K}]\gamma_K.
\label{eq:first-stem-hermitian-transfer}
\end{align}
Thus the residue and corestriction formulas determine every differential
in the three-term tail \eqref{eq:Q-Gersten-tail}.  A change of uniformizer
or of a basis of the coefficient line bundle changes the displayed coordinates
through \eqref{eq:first-stem-line-bundle-action}; the resulting twisted
Rost--Schmid differential is independent of those choices.
\end{proposition}

\begin{proof}
The \(\mathbb P^1\)-stabilization theorem
\cite[Theorem~7.2.1]{ABH} identifies the three indicated deep contractions
with the weights \(0,1,2\) of the stable first stem.  The first-stem theorem
of R\"ondigs--Spitzweck--\O stv\ae r then gives the
short exact sequences in
Proposition~\ref{prop:Q-contracted-Gersten-tail}, identifies their kernel as
the graded Milnor \(K\)-theory module generated by \(\nu\), and identifies
the last two filtration pieces of the hermitian \(K\)-theory image as the
graded Milnor--Witt module generated by \(\eta_{\mathrm{top}}\); see
\cite[Theorem~1.1 and Remark~5.8]{RSO19}.  In weights \(0,1,2\), this says
that the kernel generators are, respectively,
\([u][w]\gamma\), \([u]\gamma\), and \(\gamma\), while the hermitian
quotients are generated by \(\alpha,[u]\beta\) in weight \(0\) and by
\(\beta\) in weight \(1\).  Since \(\eta_{\mathrm{top}}\) is the image of
the classical Hopf map, \(2\alpha=2\beta=0\).  Moreover,
\cite[Remark~5.8]{RSO19} gives
\[
\eta\beta=\eta^2\eta_{\mathrm{top}}=12\gamma.
\]
The vanishing \(\eta\gamma=0\) follows from
\(\pi_{4,3}\mathds1=0\), which is the \(i=3\) case of the same
first-stem theorem; equivalently it is the kernel-annihilation statement
of \cite[Lemma~5.4]{RSO19}.  Together with the classical spinor-norm and
determinant description of
stable orthogonal \(K_1\), \cite[Remark~5.8]{RSO19} permits the stated
normalization of the quotient coordinates and gives the splitting in
\eqref{eq:first-stem-E1-splitting}.

The Milnor--Witt identities
\[
[uv]=[u]+[v]+\eta[u][v],
\qquad
\langle a\rangle=1+\eta[a]
\]
now give \eqref{eq:first-stem-factor-set} and
\eqref{eq:first-stem-line-bundle-action}.  Notice that
\eqref{eq:first-stem-factor-set} is precisely the factor set obstructing
the evident lifts of the spinor-norm factor from being additive.

It remains to calculate residues.  By
\cite[Theorem~4.0.1]{Feld21}, homotopy modules carry their canonical
Milnor--Witt cycle-module structure.  The product and projection formulas
used below are the rules of \cite[Definition~3.1 and the discussion
following it]{Feld20}.  In particular, if \(x\) is a
Milnor--Witt symbol and \(z\) is unramified, then, in the coordinates fixed
by \(\pi\),
\[
\partial_v^\pi(xz)=\partial_v^\pi(x)\,\bar z.
\]
The three universal Hopf classes are unramified for geometric valuations.
For \(u=\pi^m a\), the Milnor--Witt residue formula is
\[
\partial_v^\pi([u])=m_\epsilon\langle\bar a\rangle.
\]
Because \(\eta\gamma=0\), the Milnor--Witt action on the module generated
by \(\gamma\) factors through Milnor \(K\)-theory.  This proves
\eqref{eq:first-stem-K1-residue} and
\eqref{eq:first-stem-K2-residue}.

Finally, \(h\beta=0\), since \(h\eta=0\) in Milnor--Witt \(K\)-theory and
\(\beta=\eta\alpha\).  Hence \(m_\epsilon\beta\) is zero for \(m\) even
and is \(\beta\) for \(m\) odd.  The last identity in
\eqref{eq:first-stem-line-bundle-action} then yields
\eqref{eq:first-stem-hermitian-residue}; the same computation for negative
\(m\) uses \(m_\epsilon=\epsilon(-m)_\epsilon\) and
\(12[-1]\gamma=0\).

For the transfer formulas, the projection formula for Milnor--Witt cycle
modules gives
\[
T_{L/K}(\beta_L)=q_{L/K}\beta_K.
\]
Diagonalizing \(q_{L/K}\) and applying
\eqref{eq:first-stem-line-bundle-action} yields
\eqref{eq:first-stem-hermitian-transfer}.  This expression is independent
of the representative of \(\Delta_{L/K}\): replacing it by
\(b^2\Delta_{L/K}\) changes the correction by
\(24[b]\gamma=0\).  On the module generated by
\(\gamma\), the action factors through rank, and the kernel inclusion is a
morphism of cycle modules; this gives
\eqref{eq:first-stem-milnor-transfers}.  The coordinate-change assertion is
one of the Milnor--Witt residue axioms, and completes the proof.
\end{proof}

\begin{corollary}\label{cor:Q-two-primary-normal-form}
In the \(2\)-primary summand, write \(\gamma_2\) for the image of
\(\gamma\) in the terminal group \(\mathbb Z/8\).  The residue on a
spinor-norm lift is valuation parity together with the correction
\[
\partial_v^\pi([\pi^{2r+1}a]\beta)
=\beta+4[\bar a]\gamma_2.
\]
Likewise, the only non-additivity of the evident spinor-norm lifts is
\[
[uv]\beta-[u]\beta-[v]\beta
=4[u][v]\gamma_2.
\]
For a finite separable extension \(L/K\), the corresponding transfer
correction is
\[
T_{L/K}(\beta_L)-d\beta_K
=4[\Delta_{L/K}]\gamma_{2,K}.
\]
Thus all hermitian gluing lies in the image of multiplication by \(4\) on
the Milnor kernels modulo \(8\).  After pushing out those kernels along
\(\mathbf K^M_*/8\to\mathbf K^M_*/4\), all displayed corrections vanish.  In
particular, the previously implicit coefficient differential is reduced to
a single order-two gluing layer between the ordinary Milnor and
valuation-parity complexes.
\end{corollary}

\begin{proof}
Project the formulas of
Proposition~\ref{prop:Q-explicit-first-stem-residues} to their \(2\)-primary
summands.  Under \(\mathbb Z/24\cong\mathbb Z/8\oplus\mathbb Z/3\), the
class \(12\gamma\) has \(2\)-primary coordinate \(4\gamma_2\).  The final
assertions follow immediately.
\end{proof}

\begin{proposition}\label{prop:Q-two-primary-Steenrod-filtration}
Let \(k\) be a field of characteristic \(0\), let \(n\geq4\), let \(X\)
be a smooth pure \((n+2)\)-dimensional \(k\)-scheme, and let \(L\) be a
line bundle on \(X\).  Put
\[
\mathcal B_n:=\operatorname{im}\!\left(
 \mathcal Q_n\longrightarrow\mathbf{GW}_{n+1}^{n}
\right)
\]
and write
\(\operatorname{Ch}^i(X):=CH^i(X)/2\).  Then there is a natural
isomorphism
\begin{equation}\label{eq:B-cohomology-twisted-Sq2}
H^{n+1}_{\mathrm{Nis}}(X,\mathcal B_n(L))
\cong
\operatorname{coker}\!\left(
 \operatorname{Ch}^{n}(X)
 \xrightarrow{\ Sq_L^2\ }
 \operatorname{Ch}^{n+1}(X)
\right),
\end{equation}
where
\[
Sq_L^2(z):=Sq^2(z)+c_1(L)z.
\]
Moreover, the two-primary part of the secondary coefficient group occurs
in the natural exact sequence
\begin{equation}\label{eq:Q-two-primary-Steenrod-exact-sequence}
\begin{aligned}
H^{n+1}_{\mathrm{Nis}}(X,\mathbf K^M_{n+2}/8)
&\longrightarrow
H^{n+1}_{\mathrm{Nis}}(X,\mathcal Q_n(L))\{2\}\\
&\longrightarrow
\operatorname{coker}(Sq_L^2)
\longrightarrow
H^{n+2}_{\mathrm{Nis}}(X,\mathbf K^M_{n+2}/8).
\end{aligned}
\end{equation}
In particular, if \(Sq_L^2\) is surjective and
\[
H^{n+1}_{\mathrm{Nis}}(X,\mathbf K^M_{n+2}/8)=0,
\]
then
\[
H^{n+1}_{\mathrm{Nis}}(X,\mathcal Q_n(L))\{2\}=0.
\]
\end{proposition}

\begin{proof}
The stable comparison of \cite[Theorem~7.2.1]{ABH} gives an exact sequence
\[
0\longrightarrow\mathbf K^M_{n+2}/24
\longrightarrow\mathcal Q_n
\longrightarrow\mathcal B_n\longrightarrow0.
\]
It also says that the morphism from the \(j\)-fold contraction of
\(\mathcal Q_n\) to the corresponding contraction of
\(\mathbf{GW}_{n+1}^{n}\) is surjective for \(j\geq n-3\).  Contraction is
exact on strictly \(\mathbb A^1\)-invariant sheaves.  Consequently
\[
(\mathcal B_n)_{-j}
\xrightarrow{\ \cong\ }
(\mathbf{GW}_{n+1}^{n})_{-j}
\qquad (j\geq n-3).
\]

Only the codimension-\(n\), \(n+1\), and \(n+2\) terms of a Rost--Schmid
complex enter its cohomology in degree \(n+1\).  The preceding
identifications, including the actions of the determinant line bundle and the
residue maps,
therefore give
\[
H^{n+1}_{\mathrm{Nis}}(X,\mathcal B_n(L))
\cong
H^{n+1}_{\mathrm{Nis}}
 (X,\mathbf{GW}_{n+1}^{n}(L)).
\]
The computation of \cite[Theorem~3.7.1]{AF15b} identifies the group on the
right with the
cokernel of the twisted Steenrod square in
\eqref{eq:B-cohomology-twisted-Sq2}.

Apply cohomology to the displayed short exact sequence.  The
\(\mathbf K^M_{n+2}/24\)-kernel has trivial line-bundle action: equivalently,
the first-stem relation \(\eta\nu=0\) makes the Grothendieck--Witt action on
this summand factor through rank.  Its canonical two-primary summand is
\(\mathbf K^M_{n+2}/8\), while the group in
\eqref{eq:B-cohomology-twisted-Sq2} is annihilated by \(2\).  Taking the
two-primary part of the resulting long exact sequence gives
\eqref{eq:Q-two-primary-Steenrod-exact-sequence}.  The final assertion
follows from exactness.
\end{proof}

\begin{definition}\label{def:linear-secondary-action}
Let \(E\) be a rank-\(n\) vector bundle on \(X\), put \(L=\det E\), and
suppose that its Euler obstruction is zero.  Write \(\Lambda(E)\) for the
set of homotopy classes of lifts to the primary Moore--Postnikov stage.  The
cohomology group
\[
\mathcal A_n^{X,L}:=
H^{n-1}_{\mathrm{Nis}}(X,\mathbf K_n^{MW}(L))
\]
acts on \(\Lambda(E)\), but this action need not be free.  We say that
\((X,E)\) has the \emph{linear secondary-action property} if there is a
homomorphism
\[
\Delta_E:\mathcal A_n^{X,L}
\longrightarrow H^{n+1}_{\mathrm{Nis}}(X,\mathcal Q_n(L))
\]
such that, whenever \(\lambda'=a\cdot\lambda\),
\begin{equation}\label{eq:linear-secondary-action-property}
o_2(E,\lambda')-o_2(E,\lambda)=\Delta_E(a).
\end{equation}
Thus the definition includes independence of the chosen representative
\(a\) and of the base lift \(\lambda\).  These properties are not a formal
consequence of the existence of the Moore--Postnikov tower.
\end{definition}

\begin{theorem}\label{thm:Q-two-primary-action-normal-form}
Retain the hypotheses and notation of
Proposition~\ref{prop:Q-two-primary-Steenrod-filtration}.  Let \(E\) be a
rank-\(n\) vector bundle on \(X\), put \(L=\det E\), suppose that its Euler
obstruction is zero, and let \(\Delta_E\) be the change homomorphism
supplied by Theorem~\ref{thm:parameterized-two-stage-comparison}.
Put
\[
\mathcal A_n^{X,L}:=
H^{n-1}_{\mathrm{Nis}}(X,\mathbf K_n^{MW}(L)).
\]
Let
\[
\Delta_{E,2}:\mathcal A_n^{X,L}
\longrightarrow
H^{n+1}_{\mathrm{Nis}}(X,\mathcal Q_n(L))\{2\}
\]
be the two-primary component of \(\Delta_E\).  Write
\[
\begin{aligned}
\mathcal C_{n,2}^{X,L}
&:=\operatorname{coker}\!\left(
 \operatorname{Ch}^{n}(X)
 \xrightarrow{\ Sq_L^2\ }
 \operatorname{Ch}^{n+1}(X)
 \right),\\
\mathcal R_{n,2}^{X,L}
&:=\operatorname{coker}\!\left(
 H^n_{\mathrm{Nis}}(X,\mathcal B_n(L))
 \xrightarrow{\ \partial_{L,2}\ }
 H^{n+1}_{\mathrm{Nis}}(X,\mathbf K^M_{n+2}/8)
 \right),
\end{aligned}
\]
where \(\partial_{L,2}\) is the two-primary component of the connecting
morphism for the coefficient extension.  Then the exact sequence
\eqref{eq:Q-two-primary-Steenrod-exact-sequence} induces a natural
isomorphism
\begin{equation}\label{eq:two-primary-residual-identification}
\iota_L:\mathcal R_{n,2}^{X,L}
\xrightarrow{\ \cong\ }
\ker\!\left(
 H^{n+1}_{\mathrm{Nis}}(X,\mathcal Q_n(L))\{2\}
 \xrightarrow{\ p_L\ }
 \mathcal C_{n,2}^{X,L}
\right).
\end{equation}
Consequently the two-primary action has two successive, natural pieces
\[
\Delta_{E,2}^{\mathrm h}:=p_L\Delta_{E,2}:
\mathcal A_n^{X,L}\longrightarrow\mathcal C_{n,2}^{X,L}
\]
and
\begin{equation}\label{eq:two-primary-residual-action}
\Delta_{E,2}^{\nu}:=
\iota_L^{-1}\Delta_{E,2}\big|_{\ker(\Delta_{E,2}^{\mathrm h})}:
\ker(\Delta_{E,2}^{\mathrm h})
\longrightarrow\mathcal R_{n,2}^{X,L}.
\end{equation}

For an Euler nullhomotopy \(\lambda\), denote the two-primary secondary obstruction by
\(o_{2,2}(E,\lambda)\).  The class
\begin{equation}\label{eq:canonical-hermitian-secondary-class}
\mathfrak h_2(E):=
\big[p_L(o_{2,2}(E,\lambda))\big]
\in\operatorname{coker}(\Delta_{E,2}^{\mathrm h})
\end{equation}
is independent of \(\lambda\).  It vanishes if and only if \(\lambda\)
can be chosen so that the hermitian projection of the secondary obstruction
is zero.  If \(\mathfrak h_2(E)=0\), choose such a \(\lambda\).  Then
\begin{equation}\label{eq:canonical-residual-secondary-class}
\mathfrak n_2(E):=
\big[\iota_L^{-1}(o_{2,2}(E,\lambda))\big]
\in\operatorname{coker}(\Delta_{E,2}^{\nu})
\end{equation}
is independent of that choice.  It vanishes if and only if an Euler
nullhomotopy can be chosen for which the entire two-primary secondary
obstruction is zero.
\end{theorem}

\begin{proof}
Take the two-primary component of the long exact
cohomology sequence associated with
\[
0\longrightarrow\mathbf K^M_{n+2}/24
\longrightarrow\mathcal Q_n
\longrightarrow\mathcal B_n\longrightarrow0.
\]
Its relevant part is
\[
\begin{aligned}
H^n_{\mathrm{Nis}}(X,\mathcal B_n(L))
&\xrightarrow{\partial_{L,2}}
H^{n+1}_{\mathrm{Nis}}(X,\mathbf K^M_{n+2}/8)\\
&\longrightarrow
H^{n+1}_{\mathrm{Nis}}(X,\mathcal Q_n(L))\{2\}
\xrightarrow{p_L}
H^{n+1}_{\mathrm{Nis}}(X,\mathcal B_n(L)).
\end{aligned}
\]
After using \eqref{eq:B-cohomology-twisted-Sq2} on the last group,
exactness gives \eqref{eq:two-primary-residual-identification}.  If
\(a\in\ker(\Delta_{E,2}^{\mathrm h})\), then
\(\Delta_{E,2}(a)\in\ker(p_L)\), so
\eqref{eq:two-primary-residual-action} is well defined.

Changing \(\lambda\) by \(a\in\mathcal A_n^{X,L}\) changes
\(o_{2,2}(E,\lambda)\) by \(\Delta_{E,2}(a)\), by
\eqref{eq:linear-secondary-action-property}.  Applying \(p_L\)
proves the independence and the vanishing criterion for
\(\mathfrak h_2(E)\).  Once the hermitian projection is zero, two choices
with this property differ by an element of
\(\ker(\Delta_{E,2}^{\mathrm h})\).  Transport through \(\iota_L\) then
shows that their residual classes differ by
\(\Delta_{E,2}^{\nu}\).  This proves the independence and the final
vanishing criterion for \(\mathfrak n_2(E)\).

At cochain level, \(\partial_{L,2}\) is exactly the two-primary component
of the connecting morphism
determined by the gluing correction calculated in
Corollary~\ref{cor:Q-two-primary-normal-form}.  Hence
\(\mathcal R_{n,2}^{X,L}\) records precisely the residual Milnor class
after the already determined coefficient gluing has been divided out.
\end{proof}

\begin{proposition}\label{prop:Q-Gersten-three-primary}
Under the hypotheses of
Proposition~\ref{prop:Q-contracted-Gersten-tail}, let \(X\) be smooth and
pure of dimension \(n+2\), and let \(L\) be a line bundle on \(X\).
The \(3\)-primary direct summand of the last three terms of the
Rost--Schmid complex for \(\mathcal Q_n(L)\) is canonically the Milnor
\(K\)-theory complex
\begin{equation}\label{eq:Q-Gersten-three-primary}
\bigoplus_{x\in X^{(n)}}\mathbf K^M_2(k(x))/3
\longrightarrow
\bigoplus_{x\in X^{(n+1)}}\mathbf K^M_1(k(x))/3
\longrightarrow
\bigoplus_{x\in X^{(n+2)}}\mathbb Z/3.
\end{equation}
Its arrows are the usual residue, or tame-symbol, morphisms.  Consequently,
writing \(\{3\}\) for the \(3\)-primary subgroup, there is a natural
isomorphism
\begin{equation}\label{eq:Q-cohomology-three-primary}
H^{n+1}_{\mathrm{Nis}}(X,\mathcal Q_n(L))\{3\}
\cong
H^{n+1}_{\mathrm{Nis}}
 \left(X,\mathbf K^M_{n+2}/3\right).
\end{equation}
More precisely, there is a natural primary decomposition
\begin{equation}\label{eq:Q-cohomology-primary-splitting}
H^{n+1}_{\mathrm{Nis}}(X,\mathcal Q_n(L))
\cong
H^{n+1}_{\mathrm{Nis}}(X,\mathcal Q_n(L))\{2\}
\oplus
H^{n+1}_{\mathrm{Nis}}
 \left(X,\mathbf K^M_{n+2}/3\right).
\end{equation}
In particular, the line-bundle twist is invisible on the second summand.
\end{proposition}

\begin{proof}
On fields, \(\mathbf{GW}^0_1\) has exponent \(2\).  Indeed,
\cite[Subsection~2.3]{AF17} identifies it with the stable orthogonal
abelianization, and the Cartan--Dieudonn\'e theorem shows that this
abelianization is generated by classes of reflections.  The determinant and
spinor norm give its familiar explicit invariants.  The sheaf
\(\mathbf{GW}^3_0\) is constant with value
\(\mathbb Z/2\).  Therefore the \(3\)-primary subgroups in
\eqref{eq:Q-penultimate-contracted-extension} and
\eqref{eq:Q-first-contracted-extension} are, respectively,
\[
\mathbf K^M_2/3
\qquad\text{and}\qquad
\mathbf K^M_1/3,
\]
while the \(3\)-primary subgroup of
\eqref{eq:Q-terminal-contraction} is \(\mathbf K^M_0/3\).
The same short exact sequences show that all three coefficient sheaves are
annihilated by an integer whose only prime divisors are \(2\) and \(3\).
Their primary decompositions are therefore canonical and functorial, so every
Rost--Schmid residue preserves them.

The inclusion of the Milnor \(K\)-theory kernel in the stabilized
first-stem sequence is a morphism of strictly
\(\mathbb A^1\)-invariant sheaves and commutes with contraction.
Naturality of the Rost--Schmid construction therefore identifies its
codimension boundary maps with the standard Milnor \(K\)-theory residues.
This proves \eqref{eq:Q-Gersten-three-primary}.

Cohomology in degree \(n+1\) depends only on the terms in codimensions
\(n,n+1,n+2\).  Taking the canonical \(3\)-primary direct summands thus
identifies it with the degree-\((n+1)\) cohomology of the Rost--Schmid
complex for \(\mathbf K^M_{n+2}/3\), proving
\eqref{eq:Q-cohomology-three-primary}.  The same functorial primary
decomposition gives \eqref{eq:Q-cohomology-primary-splitting}.  Finally, the action of
\(\mathbf K^{MW}_0=GW\) on Milnor \(K\)-theory modulo \(3\) factors through
the rank morphism \(GW\to\mathbb Z\): the relation
\(\langle a\rangle=1+\eta[a]\) becomes \(\langle a\rangle=1\) after passing
from Milnor--Witt to Milnor \(K\)-theory.  Hence twisting by the determinant
line bundle acts trivially on this summand.
\end{proof}

\begin{lemma}\label{lem:two-stage-action-naturality}
Let \(B\) be a motivic space and let
\(\mathcal T\to\mathcal T'\) be a morphism of two-stage relative
Moore--Postnikov towers over \(B\).  Suppose that the primary and secondary
coefficient local systems are, respectively,
\[
(\mathcal A,\mathcal Q)
\qquad\text{and}\qquad
(\mathcal A',\mathcal Q'),
\]
and write \(u:\mathcal A\to\mathcal A'\) and
\(v:\mathcal Q\to\mathcal Q'\) for the induced morphisms.  For every
\(x:X\to B\), choose a primary lift \(\ell\), and let
\(a\in H^{n-1}_{\mathrm{Nis}}(X,\mathcal A_x)\) act on it.  If \(F\ell\)
is the induced lift in \(\mathcal T'\), then
\begin{equation}\label{eq:based-change-naturality}
v_*\bigl(\kappa_{\mathcal T}(a\cdot\ell)
          -\kappa_{\mathcal T}(\ell)\bigr)
=
\kappa_{\mathcal T'}(u_*a\cdot F\ell)
          -\kappa_{\mathcal T'}(F\ell).
\end{equation}
In particular, if in a given lifting problem these differences are
independent of the base lift and define homomorphisms, then those
homomorphisms form the commutative square obtained from
\eqref{eq:based-change-naturality}.  Without that additional hypothesis,
the lemma makes no assertion of additivity or lift independence.
\end{lemma}

\begin{proof}
Relative Postnikov truncation is functorial in the motivic
\(\infty\)-category over \(B\); this is the relative form of the
Moore--Postnikov construction used in
\cite[Section~6.1]{AF15b}.  The cohomology group
\[
[X,K(\mathcal A_x,n-1)]
=H^{n-1}_{\mathrm{Nis}}(X,\mathcal A_x)
\]
acts on primary lifts; on homotopy classes the resulting set of lifts can
be a quotient of this group, as explicitly noted in
\cite[Section~6.1]{AF15b}.  The secondary obstruction is the pullback of
the relative \(k\)-invariant with values in
\(K(\mathcal Q_x,n+1)\).  A morphism of the two-stage towers intertwines
the actions by \(u\)
and the relative \(k\)-invariants by \(v\).  Applying it to the displayed
two differences gives \eqref{eq:based-change-naturality}.
The construction takes place in the category of local systems on \(B\);
hence the same calculation is valid with monodromy and after pullback along
\(x\).
\end{proof}

\begin{lemma}\label{lem:Q-three-primary-coniveau-comparison}
Let \(k\) be a field of characteristic \(0\), let \(n\geq4\), let \(X\)
be a smooth pure \((n+2)\)-dimensional \(k\)-scheme, and let \(L\) be a
line bundle on \(X\).  In the contractions that occur in codimensions
\(n\), \(n+1\), and \(n+2\), the suspension-spectrum unit for
\(\mathbb A^n\setminus0\) identifies the three-primary part of the
non-stable first-stem coefficient complex with the Milnor \(K\)-theory
complex modulo \(3\) in the stable first stem.  Consequently it induces
the identity under the identification
\[
H^{n+1}_{\mathrm{Nis}}(X,\mathcal Q_n(L))\{3\}
\cong
H^{n+1}_{\mathrm{Nis}}(X,\mathbf K^M_{n+2}/3).
\]
This is a coefficient-complex comparison; it does not by itself compare
the relative \(k\)-invariants or the change of secondary obstruction.
\end{lemma}

\begin{proof}
Let \(\mathcal Q_n^{\mathrm{st}}\) denote the stable first-stem coefficient
of the suspension spectrum of \(\mathbb A^n\setminus0\).  The
degree-\((n+1)\) target is computed by the codimension \(n\), \(n+1\),
and \(n+2\) terms of the Rost--Schmid complex.  The
\(\mathbb P^1\)-stabilization theorem
\cite[Theorem~7.2.1]{ABH} identifies the contractions
\((\mathcal Q_n)_{-n}\), \((\mathcal Q_n)_{-(n+1)}\), and
\((\mathcal Q_n)_{-(n+2)}\) occurring there with the corresponding weights
of the stable first stem.  These identifications are induced by the
suspension-spectrum unit.  Since the unit is a morphism of homotopy modules,
they commute with contractions, Rost--Schmid residues, and finite
corestrictions.  On the three-primary direct summands the determinant
action is trivial, so the same comparison applies after twisting by \(L\).
Proposition~\ref{prop:Q-Gersten-three-primary} identifies the induced map,
in the three codimensions that compute degree \(n+1\), with the identity of
the Milnor \(K\)-theory complex modulo \(3\).  Passing to cohomology proves
the assertion.  Nothing in this argument constructs a morphism of the
relative two-stage towers over \(BGL_n\).
\end{proof}

\smallskip
\noindent\textbf{The mod-\(3\) Thom correction.}
For any vector bundle \(E\) on a smooth scheme \(X\), define its first
mod-\(3\) motivic Wu class by
\[
\omega_3(E):=
\operatorname{red}_3\!\left(c_1(E)^2-2c_2(E)\right)
=
\operatorname{red}_3\!\left(c_1(E)^2+c_2(E)\right)
\in H^{4,2}(X,\mathbb Z/3)
\]
and the Thom-corrected reduced power by
\[
\mathcal P_E^1(x):=P^1(x)+\omega_3(E)\mathbin{\smile}x.
\]

\smallskip
\noindent\textbf{The parameterized stable section torsor.}
Put \(G=GL_n\) and \(F_n=\mathbb A^n\setminus0\), with its natural
\(G\)-action.  The augmentation of the suspension spectrum is
\(G\)-equivariant:
\[
\epsilon:\Sigma^\infty_+F_n\longrightarrow\mathds1.
\]
Let
\[
Q_1(F_n):=
\operatorname{hofib}_{1}\!\left(
 \Omega^\infty\Sigma^\infty_+F_n
 \xrightarrow{\Omega^\infty\epsilon}
 \Omega^\infty\mathds1
\right),
\]
where the fiber is taken over the section represented by the identity of
\(\mathds1\).  The suspension-spectrum unit sends every point of \(F_n\)
to that fiber and therefore gives a canonical \(G\)-equivariant map
\begin{equation}\label{eq:augmentation-one-unit}
\iota_n:F_n\longrightarrow Q_1(F_n).
\end{equation}
The use of the fiber over \(1\), rather than a reduced suspension spectrum
formed from a chosen point of \(F_n\), is what makes
\eqref{eq:augmentation-one-unit} equivariant.
All Borel constructions below are homotopy Borel constructions in
motivic spaces.

\begin{lemma}\label{lem:augmentation-one-connectivity}
Let \(k\) be a field of characteristic \(0\) and let \(n\geq4\).  The
homotopy fiber of \eqref{eq:augmentation-one-unit} is
\((2n-4)\)-\(\mathbb A^1\)-connected.  If
\[
\mathscr S_n:=EG\times^GQ_1(F_n)\longrightarrow BG,
\]
then \eqref{eq:augmentation-one-unit} induces a morphism over \(BG\)
\begin{equation}\label{eq:relative-stabilization-map}
BGL_{n-1}\simeq EG\times^GF_n
\longrightarrow \mathscr S_n
\end{equation}
and an equivalence of relative Moore--Postnikov truncations
\begin{equation}\label{eq:relative-two-stage-equivalence}
\tau^{BG}_{\leq n}(BGL_{n-1})
\xrightarrow{\ \simeq\ }
\tau^{BG}_{\leq n}(\mathscr S_n).
\end{equation}
In particular, this equivalence contains both the primary and secondary
relative stages for the fiber \(F_n\).
\end{lemma}

\begin{proof}
Choose a point \(v\in F_n(k)\).  It splits the augmentation and gives an
equivalence of spectra
\[
\Sigma^\infty_+F_n\simeq
\mathds1\vee\Sigma^\infty(F_n,v)
\]
under which \(\epsilon\) is projection onto the first factor.  Translation
by the section represented by \(v\) consequently identifies \(Q_1(F_n)\)
with \(\Omega^\infty\Sigma^\infty(F_n,v)\), and it identifies \(\iota_n\)
with the pointed stable unit.  This identification is used only to compute
the homotopy fiber; the map \(\iota_n\) itself is independent of \(v\).

Since \(F_n\simeq S^{2n-1,n}\),
\cite[Theorem~6.3.3]{ABH} places the fiber of the stable unit in
\(O(S^{4n-3,2n})\).  By \cite[Corollary~3.1.27]{ABH} it is therefore
\((2n-4)\)-\(\mathbb A^1\)-connected.
\(F_n\) and \(Q_1(F_n)\) are connected, so
\cite[Proposition~3.1.23]{ABH} moreover shows that this cellular
equivalence is universal
under base change.

The Borel construction applied to the \(G\)-equivariant map \(\iota_n\)
gives \eqref{eq:relative-stabilization-map}; the identification of its
source uses the stabilizer
\[
P_n=\operatorname{Stab}_{GL_n}(e_n)
\cong GL_{n-1}\ltimes\mathbb A^{n-1}.
\]
Indeed, \(F_n=G/P_n\), so \(EG\times^GF_n\simeq BP_n\), and the
\(\mathbb A^1\)-contractibility of the unipotent radical gives
\(BP_n\simeq_{\mathbb A^1}BGL_{n-1}\); this is the homogeneous-space
model of \cite[Theorem~2.2.4]{AffineRep2}.  Under this equivalence the map
to \(BG=BGL_n\) is the standard block stabilization.  Nisnevich-locally
on \(BG\), the two
associated fibrations are the products with \(F_n\) and \(Q_1(F_n)\), and
their relative homotopy sheaves are the corresponding homotopy sheaves
with the descended monodromy action.  The preceding universal connectivity
estimate therefore shows that \eqref{eq:relative-stabilization-map} induces
isomorphisms on relative homotopy sheaves through degree \(2n-4\).
Relative Postnikov truncation is functorial and is determined by these
relative homotopy sheaves and their \(k\)-invariants.  Since
\(n\leq2n-4\), the relative Whitehead theorem gives
\eqref{eq:relative-two-stage-equivalence}.
\end{proof}

\begin{lemma}\label{lem:stable-sections-euler-nullhomotopies}
Let \(E\to X\) be a rank-\(n\) vector bundle and let
\(E^\times=E\setminus X\) be the complement of its zero section.  Pulling
back \(\mathscr S_n\) along the classifying map of \(E\) gives a stable
section torsor \(\mathscr S(E)\to X\).  Its space of sections is naturally
the space of stable nullhomotopies of the Euler map of \(E\).

More precisely, in the stable motivic category over \(X\) there is a
cofiber sequence
\begin{equation}\label{eq:fiberwise-euler-cofiber}
\Sigma^\infty_{X,+}E^\times
\xrightarrow{\epsilon_E}\mathds1_X
\xrightarrow{e^{\mathrm{st}}(E)}S^E,
\end{equation}
where \(S^E\) is the fiberwise Thom sphere, and
\begin{equation}\label{eq:stable-section-nullhomotopy-equivalence}
\Gamma_X\mathscr S(E)
\simeq
\operatorname{Null}\bigl(e^{\mathrm{st}}(E)\bigr).
\end{equation}
The difference of two such nullhomotopies belongs naturally to the
infinite loop space associated with \(\Sigma^{-1}S^E\).  Consequently the
change of every stable Postnikov obstruction is independent of the base
nullhomotopy and is a homomorphism in that difference.
\end{lemma}

\begin{proof}
Homotopy purity applied over \(X\) gives the pointed cofiber sequence
\[
E^\times_+\longrightarrow E_+\longrightarrow S^E.
\]
Fiberwise scalar contraction identifies \(E_+\) with the unit over \(X\).
Stabilization gives \eqref{eq:fiberwise-euler-cofiber}; this is the
fiberwise form of the Euler cofiber construction used in
\cite[Lemma~5.2.1 and Proposition~5.3.1]{AF16Euler}.

We first identify the pullback of the Borel construction with the
fiberwise stable construction.  Let
\(P=\operatorname{Fr}(E)\to X\) be the frame torsor.  It is Nisnevich
locally trivial, and pullback of the universal homotopy Borel construction
along the classifying map identifies
\(\mathscr S(E)\) with \(P\times^GQ_1(F_n)\).  On a trivializing morphism
\(U\to X\), the two objects
\[
 \left.\Sigma^\infty_{X,+}E^\times\right|_U
 \qquad\text{and}\qquad
 P|_U\times^G\Sigma^\infty_+F_n
\]
are canonically identified with
\(\Sigma^\infty_{U,+}(U\times F_n)\), and the two augmentations are both
induced by the projection \(U\times F_n\to U\).  On an overlap, these
identifications differ by the given \(G\)-action, which is precisely the
descent datum defining the associated object.  Nisnevich descent in
\(\mathbf{SH}(-)\) therefore gives an augmentation-preserving equivalence
\begin{equation}\label{eq:associated-spectrum-punctured-bundle}
 P\times^G\Sigma^\infty_+F_n
 \xrightarrow{\ \simeq\ }
 \Sigma^\infty_{X,+}E^\times.
\end{equation}

The functor \(\Omega_X^\infty\) preserves homotopy limits.  Homotopy
fibers of motivic spaces and their section spaces also satisfy Nisnevich
descent.  Taking the fiber over the unit section in
\eqref{eq:associated-spectrum-punctured-bundle} consequently gives a
canonical equivalence over \(X\)
\begin{equation}\label{eq:associated-Q1-fiber-identification}
 \mathscr S(E)
 \simeq
 \operatorname{hofib}_{1}\!\left(
  \Omega_X^\infty\Sigma^\infty_{X,+}E^\times
  \longrightarrow\Omega_X^\infty\mathds1_X
 \right).
\end{equation}
This construction is independent of the chosen trivializing cover, since
both sides are obtained from the same \(G\)-equivariant augmentation by
effective Nisnevich descent.

By the suspension-spectrum adjunction and
\eqref{eq:associated-Q1-fiber-identification}, a section of
\(\mathscr S(E)\) is equivalently a morphism
\(\mathds1_X\to\Sigma^\infty_{X,+}E^\times\) whose composite with
\(\epsilon_E\) is the identity.  Applying
\(\operatorname{Map}(\mathds1_X,-)\) to
\eqref{eq:fiberwise-euler-cofiber} identifies the space of these splittings
with the space of nullhomotopies of \(e^{\mathrm{st}}(E)\), proving
\eqref{eq:stable-section-nullhomotopy-equivalence}.  Exactness also
identifies the difference space with
\(\operatorname{Map}(\mathds1_X,\Sigma^{-1}S^E)\).  This is an infinite
loop space.  Its Postnikov \(k\)-invariants are maps of spectra, so their
actions on cohomology are homomorphisms and are invariant under translation
of the chosen nullhomotopy.
\end{proof}

\smallskip
\noindent\textbf{Three-primary sign convention.}
Fix Voevodsky's standard generator of \(\mathbb Z/3\) and the corresponding
reduced power \(P^1\).  Let
\[
\nu_{\mathrm{top}}^{(3)}\in(\pi_3^s)_{(3)}
\]
be the unique generator for which the first three-local \(k\)-invariant of
the classical sphere is \(+P^1\operatorname{red}_3\).  The universal motivic
Hopf class is defined over \(\mathbb Z\), and complex realization is an
isomorphism on the three-primary part of
\(\pi_{3,2}(\mathds1_{\mathbb C})\)
\cite[Theorem~1.1 and Remark~5.8]{RSO19}.  There is consequently a unique
choice of sign for its three-primary component over \(\mathbb Q\) whose
complex realization is \(\nu_{\mathrm{top}}^{(3)}\); denote this class and
all its characteristic-zero base changes by \(\nu_{(3)}\).  Every
identification below of the three-primary first-stem summand with
\(\mathbf K_*^M/3\) sends \(\nu_{(3)}\) to \(1\).  This convention fixes the
sign that realization alone would determine only up to a unit of
\(\mathbb Z/3\).

\begin{lemma}\label{lem:degree-42-integral-mod3-operations}
Let \(k\) be a field of characteristic zero.  The group of bistable
motivic cohomology operations of bidegree \((4,2)\) from
\(\mathbb Z_{(3)}\)-cohomology to \(\mathbb Z/3\)-cohomology is
\[
 [H\mathbb Z_{(3)},\Sigma^{4,2}H\mathbb Z/3]
 \cong
 \mathbb Z/3\{P^1\operatorname{red}_3\}.
\]
This identification is compatible with extension of the ground field.
\end{lemma}

\begin{proof}
Apply \([- ,\Sigma^{4,2}H\mathbb Z/3]\) to the coefficient cofiber
sequence
\[
 H\mathbb Z_{(3)}\xrightarrow{\ 3\ }H\mathbb Z_{(3)}
 \longrightarrow H\mathbb Z/3.
\]
Multiplication by \(3\) on the target is zero, so every operation from
\(H\mathbb Z_{(3)}\) is the precomposition with
\(\operatorname{red}_3\) of a mod-
\(3\) operation.  Voevodsky's classification of bistable operations at an
odd prime \cite{Voevodsky03,Voevodsky10} expresses the latter as a module
over the coefficient ring, with basis the admissible reduced-power and
Bockstein monomials.  In total bidegree \((4,2)\), the only positive
Steenrod monomial is \(P^1\).  The only possible additional terms are
multiplication of the identity, or of the Bockstein, by coefficient classes
in \(H^{4,2}(k,\mathbb Z/3)\), respectively
\(H^{3,2}(k,\mathbb Z/3)\).  Both groups vanish because motivic cohomology
of a field vanishes in bidegrees \((p,q)\) with \(p>q\).  Thus the mod-
\(3\) source group is \(\mathbb Z/3\{P^1\}\).  Its image under
precomposition is nonzero over every such field: on motivic projective space,
\(P^1\operatorname{red}_3\) sends the reduction of the integral first Chern
class in \(H^{2,1}(\mathbb P^\infty,\mathbb Z_{(3)})\) to its nonzero third
power in \(H^{6,3}(\mathbb P^\infty,\mathbb Z/3)\).  Hence the surjection
above is an isomorphism and proves the displayed
identification.  Reduced powers, reduction, and the coefficient generator
are preserved by base change, which proves the last assertion.
\end{proof}

\begin{lemma}\label{lem:three-primary-Postnikov-base-change}
Let \(L/k\) be an extension of characteristic-zero fields.  Extension of
scalars carries the first three-primary homotopy-\(t\)-structure Postnikov
invariant of the sphere over \(k\) to the corresponding invariant over
\(L\).  Under the identification of
Lemma~\ref{lem:degree-42-integral-mod3-operations}, the scalar multiplying
\(P^1\operatorname{red}_3\) is therefore unchanged by extension of the
ground field.
\end{lemma}

\begin{proof}
Since characteristic-zero fields are perfect, the morphism
\(\operatorname{Spec}L\to\operatorname{Spec}k\) is essentially smooth in
the sense of \cite[Paragraph~2.1.5]{ABH}.  Stable motivic extension of
scalars and its continuity for such morphisms are recalled in
\cite[Paragraph~2.2.10]{ABH}.  The nonnegative part of the homotopy
\(t\)-structure is generated under colimits and extensions by
\(\Sigma^{p,q}\Sigma^\infty X_+\), where \(X\) is smooth and \(p-q\geq0\)
\cite[Paragraph~2.2.19]{ABH}; extension of scalars preserves these
generators.  For the coconnective part, a smooth \(L\)-scheme and the finite
diagram defining any section of a homotopy sheaf descend to a finitely
generated smooth subextension of \(L/k\).  Mapping-space continuity
\cite[Paragraphs~2.1.5 and~2.2.10]{ABH} then expresses the corresponding
section after base change as a filtered colimit of sections in the same
homotopy degree over smooth \(k\)-schemes, using smooth adjunction at each
finite stage.  Hence positive homotopy sheaves
of a coconnective spectrum remain zero.  Extension of scalars is therefore
\(t\)-exact and commutes with homotopy truncations.

It follows that extension of scalars carries the truncation triangle, its
two homotopy modules, and the connecting Postnikov map to their counterparts
over \(L\).  The low sphere homotopy modules and the quotient generated by
\(\nu_{(3)}\) are themselves compatible with extension of fields by the
natural first-stem sequence of \cite[Theorem~1.1]{RSO19}.  Hence the
resulting map
\[
H\mathbb Z_{(3)}\longrightarrow\Sigma^{4,2}H\mathbb Z/3
\]
is the scalar extension of the map over \(k\).  Finally,
Lemma~\ref{lem:degree-42-integral-mod3-operations} identifies the operation
group with the constant cyclic group generated by
\(P^1\operatorname{red}_3\), compatibly with base change.  The coefficient
is therefore unchanged.
\end{proof}

\begin{proposition}\label{prop:P1-Thom-Wu}
Let \(X\) be a smooth scheme over a field of characteristic \(0\), and let
\(E\) be a vector bundle of rank \(n\) on \(X\).  If
\[
\Phi_E:
H^{p,q}(X,\mathbb Z/3)
\xrightarrow{\ \cong\ }
\widetilde H^{p+2n,q+n}(\operatorname{Th}(E),\mathbb Z/3)
\]
is the positive Thom isomorphism, then
\[
\Phi_E^{-1}P^1\Phi_E(x)
=
P^1(x)+\omega_3(E)\mathbin{\smile}x
=
\mathcal P_E^1(x).
\]
The same correction is obtained from \(E^\vee\).
\end{proposition}

\begin{proof}
Let \(t_E\) denote the mod-\(3\) Thom class, so that
\(\Phi_E(x)=x\smile t_E\).  Voevodsky's Thom formula
\cite[Corollary~14.4]{Voevodsky03} gives
\[
P^1(t_E)=c_{1,2}(E)\smile t_E.
\]
After pullback to a splitting space with Chern roots \(x_1,\ldots,x_n\),
the characteristic class on the right is
\[
c_{1,2}(E)
=\sum_i\operatorname{red}_3(x_i)^2
=\operatorname{red}_3\!\left(c_1(E)^2-2c_2(E)\right).
\]
The splitting principle therefore gives this identity on \(X\).  The
Cartan formula now yields
\[
P^1(x\smile t_E)
=
\bigl(P^1(x)+
\operatorname{red}_3(c_1(E)^2-2c_2(E))\smile x\bigr)\smile t_E.
\]
Applying \(\Phi_E^{-1}\) proves the formula.  Since
\(c_1(E^\vee)=-c_1(E)\) and \(c_2(E^\vee)=c_2(E)\), one has
\(\omega_3(E^\vee)=\omega_3(E)\).
\end{proof}

\begin{remark}\label{rem:positive-Thom-convention}
The convention in Proposition~\ref{prop:P1-Thom-Wu} and in the comparison
below is the positive Thom cofiber
\[
E^\times_+\longrightarrow E_+\simeq X_+
\longrightarrow\operatorname{Th}(E).
\]
Using instead the inverse Thom class of the virtual bundle \(-E\) changes
the sign of the first Wu term.  All geometric statements below use the
positive Euler--Gysin convention displayed here.
\end{remark}

\begin{lemma}\label{lem:Thom-twisted-Postnikov}
Let \(p:X\to\operatorname{Spec}k\) be smooth, let \(E\) be a rank-
\(n\) vector bundle on \(X\), and let
\[
 \kappa:H\mathbb Z_{(3)}\longrightarrow
 \Sigma^{4,2}H\mathbb Z/3
\]
be the first three-primary homotopy-
\(t\)-structure \(k\)-invariant of the sphere.  The corresponding two-layer
Postnikov invariant of the difference spectrum \(\Sigma^{-1}S^E\), after
projection to the three-primary Milnor summand, is the
Thom twist
\begin{equation}\label{eq:Thom-twisted-Postnikov-map}
 S^E\mathbin{\wedge}p^*H\mathbb Z_{(3)}
 \xrightarrow{\ \mathrm{id}_{S^E}\wedge p^*\kappa\ }
 \Sigma^{4,2}
 \bigl(S^E\mathbin{\wedge}p^*H\mathbb Z/3\bigr),
\end{equation}
with the common suspension coming from \(\Sigma^{-1}S^E\) suppressed.
Under the positive mod-
\(3\) Thom isomorphism, the operation induced on the base is
\begin{equation}\label{eq:Thom-conjugated-Postnikov-operation}
 \Phi_E^{-1}\circ\kappa\circ\Phi_E.
\end{equation}
The construction is compatible with Nisnevich base change and with the
determinant local system on the integral bottom layer.
\end{lemma}

\begin{proof}
Choose a Nisnevich cover \(U\to X\) on which \(E\) is trivialized by an
ordered frame.  On \(U\) there is an identification
\[
 S^{E|_U}\simeq S^{2n,n}\mathbin{\wedge}\mathds1_U.
\]
Thus the two relevant homotopy-
\(t\)-structure layers of \(\Sigma^{-1}S^{E|_U}\) are the corresponding
layers of the sphere shifted by \(S^{2n-1,n}\), and their \(k\)-invariant
is the same shift of \(p^*\kappa\).  This is exactly the restriction of
\eqref{eq:Thom-twisted-Postnikov-map} to \(U\).

On a double overlap, a change of frame acts on the Thom sphere and on both
Postnikov layers by the same Thom-coordinate automorphism.  Since
\eqref{eq:Thom-twisted-Postnikov-map} is obtained by smashing a map of
spectra with that Thom sphere, it commutes with this automorphism.  The
local shifted Postnikov maps therefore carry the descent datum of
\(S^E\), and Nisnevich descent for \(\mathbf{SH}(-)\), its homotopy sheaves,
and mapping spaces glues them uniquely to
\eqref{eq:Thom-twisted-Postnikov-map}.  Before reduction, the same frame
change is the determinant action on the Milnor--Witt bottom layer.  After
reduction modulo \(3\), it becomes trivial because
\(\langle u\rangle=1+\eta[u]\) and \(\eta=0\) on Milnor
\(K\)-theory.  This proves the asserted compatibility with the local
system.

The positive Thom class \(t_E\) identifies a class \(x\) on \(X\) with
\(\Phi_E(x)=x\smile t_E\) on \(\operatorname{Th}(E)\).  Applying
\eqref{eq:Thom-twisted-Postnikov-map} to this class and then applying the
inverse Thom isomorphism gives
\(\Phi_E^{-1}\kappa\Phi_E(x)\).  This proves
\eqref{eq:Thom-conjugated-Postnikov-operation} and the lemma.
\end{proof}

\begin{theorem}\label{thm:parameterized-two-stage-comparison}
Let \(k\) be a field of characteristic \(0\), let \(n\geq4\), let \(X\)
be a smooth \(k\)-scheme, and let \(E\) be a rank-\(n\) vector bundle on
\(X\).  Pullback of \eqref{eq:relative-stabilization-map} along the
classifying map of \(E\) identifies the first two relative
Moore--Postnikov stages of the punctured-bundle section problem with the
first two stages of its stable Euler-nullhomotopy problem.

Suppose that the Euler obstruction of \(E\) vanishes and put \(L=\det E\).
Then \((X,E)\) has the linear secondary-action property of
Definition~\ref{def:linear-secondary-action}.  If \(X\) is pure of
dimension \(n+2\), then, under the
coefficient comparison of
Lemma~\ref{lem:Q-three-primary-coniveau-comparison}, the three-primary
component of its change homomorphism is the Thom conjugate of the first
three-primary stable sphere \(k\)-invariant \(\kappa^{\mathrm{st}}_3\):
\begin{equation}\label{eq:relative-Thom-comparison}
\Delta_E\{3\}
=
\Phi_E^{-1}\circ\kappa^{\mathrm{st}}_3\circ\Phi_E
\circ\operatorname{red}_3^L.
\end{equation}
Here \(\Phi_E\) is the positive mod-\(3\) Thom isomorphism, with the
bidegree determined by the argument of the operation.  Thus
\eqref{eq:relative-Thom-comparison} is an equality of
change homomorphisms, not merely an identification of their coefficient
sheaves.
\end{theorem}

\begin{proof}
Let \(x_E:X\to BGL_n\) classify \(E\), and form the homotopy pullback of
\eqref{eq:relative-stabilization-map} along \(x_E\).
The universal base-change estimate in
Lemma~\ref{lem:augmentation-one-connectivity} shows that the induced map
over \(X\) is still an isomorphism on relative homotopy sheaves through
degree \(n\).  Applying the functorial relative Moore--Postnikov
construction after this pullback and using the relative Whitehead theorem
therefore gives an equivalence between the primary and secondary stages
of the two section problems over \(X\).  The primary class on the geometric side is
the Euler obstruction.  Functoriality of the primary relative
\(k\)-invariant carries it to the primary class of the stable section
problem, which Lemma~\ref{lem:stable-sections-euler-nullhomotopies}
identifies with the second map in
\eqref{eq:fiberwise-euler-cofiber}.  Thus the comparison itself carries
Euler nullhomotopies to stable Euler nullhomotopies; no independent choice
of a Thom or Chow--Witt generator enters this assertion.

Let \(\lambda\) be a primary lift and let
\[
a\in H^{n-1}_{\mathrm{Nis}}
 (X,\mathbf K_n^{MW}(L)).
\]
Functoriality of the relative \(k\)-invariant gives
\[
o_2(E,a\cdot\lambda)-o_2(E,\lambda)
=
\Delta_E^{\mathrm{st}}(a),
\]
where the right-hand side is the change in the corresponding stable
section problem.  By
Lemma~\ref{lem:stable-sections-euler-nullhomotopies}, stable
nullhomotopies form a torsor under an infinite loop space and their
Postnikov invariants are maps of spectra.  Hence the right-hand side is
independent of \(\lambda\), depends only on the class of \(a\), and is
additive in \(a\).  If \(a\) fixes a geometric primary lift, its image
fixes the corresponding stable lift, so the stable \(k\)-invariant sends
\(a\) to zero.  Thus the formula also descends through the possible
stabilizers in the action on primary lifts.  Transport through the
two-stage equivalence proves
the linear secondary-action property and identifies \(\Delta_E\) with
this stable change homomorphism.

It remains to identify the parameterized stable invariant.  The cofiber
sequence \eqref{eq:fiberwise-euler-cofiber} identifies the spectrum of
differences of stable nullhomotopies with \(\Sigma^{-1}S^E\).
Lemma~\ref{lem:Thom-twisted-Postnikov} identifies its first relevant
three-primary \(k\)-invariant with the Thom twist of the sphere invariant
and identifies the induced operation on \(X\) with its conjugate by
\(\Phi_E\).  The suspension-spectrum unit induces the
coefficient identification in
Lemma~\ref{lem:Q-three-primary-coniveau-comparison}, including its
Rost--Schmid residues and determinant monodromy.  This proves
\eqref{eq:relative-Thom-comparison}.
\end{proof}

\begin{theorem}\label{thm:Q-three-primary-Postnikov-operation}
Let \(k\) be a field of characteristic \(0\), let \(n\geq4\), let \(X\)
be a smooth pure \((n+2)\)-dimensional \(k\)-scheme, and let \(L\) be a
line bundle on \(X\).  The first three-primary stable \(k\)-invariant of
the sphere induces a natural homomorphism
\[
\Delta^{\mathrm{st},X,L}_{n,3}:
H^{n-1}_{\mathrm{Nis}}(X,\mathbf K_n^{MW}(L))
\longrightarrow
H^{n+1}_{\mathrm{Nis}}(X,\mathbf K^M_{n+2}/3).
\]
With the three-primary generator normalized above, one has

\begin{equation}\label{eq:three-primary-k-invariant-P1}
\Delta^{\mathrm{st},X,L}_{n,3}
=P^1\circ\operatorname{red}_3,
\end{equation}

where the determinant twist becomes trivial after
\(\operatorname{red}_3:\mathbf K_n^{MW}(L)\to\mathbf K_n^M/3\), and
\(P^1\) is Voevodsky's first reduced power, viewed as the bistable operation

\[
P^1:
H^{2n-1,n}(-,\mathbb Z/3)
\longrightarrow
H^{2n+3,n+2}(-,\mathbb Z/3).
\]

Let \(E\) be a rank \(n\) vector bundle on \(X\), put \(L=\det E\), and
suppose that its primary Euler obstruction vanishes.  If \(\lambda\)
and \(\lambda'\) are two nullhomotopies of that obstruction and

\[
a(\lambda,\lambda')\in
H^{n-1}_{\mathrm{Nis}}(X,\mathbf K_n^{MW}(L))
\]

is a difference carrying \(\lambda\) to \(\lambda'\), then the
corresponding secondary obstructions satisfy

\begin{equation}\label{eq:three-primary-secondary-difference}
o_2(E,\lambda')\{3\}-o_2(E,\lambda)\{3\}
=\mathcal P_E^1\!\left(\operatorname{red}_3^L
 a(\lambda,\lambda')\right)
\end{equation}

in
\(H^{n+1}_{\mathrm{Nis}}(X,\mathbf K^M_{n+2}/3)\).  Consequently the
class

\begin{equation}\label{eq:canonical-three-primary-secondary-class}
\overline o_{2,3}(E)\in
\operatorname{coker}\!\left(
H^{n-1}_{\mathrm{Nis}}(X,\mathbf K_n^{MW}(L))
\xrightarrow{\ \mathcal P_E^1\operatorname{red}_3^L\ }
H^{n+1}_{\mathrm{Nis}}(X,\mathbf K^M_{n+2}/3)
\right)
\end{equation}

is independent of \(\lambda\).  It vanishes if and only if the
three-primary component of the secondary obstruction can be killed by a
change of Euler nullhomotopy.  If \(\omega_3(E)=0\), the corrected
operation reduces to \(P^1\operatorname{red}_3^L\).
\end{theorem}

\begin{proof}
Under the first-stem sequence of
\cite[Theorem~1.1]{RSO19}, the three-primary terms form the homotopy
module generated by the stabilized Hopf class \(\nu_{(3)}\), and the
residue maps are the Milnor residues.  Consequently
\(\Delta^{\mathrm{st},X,L}_{n,3}\) is the operation induced, in this
coniveau range, by the \(\nu_{(3)}\)-component of the first stable
\(k\)-invariant of the motivic sphere.

After inverting \(2\), the idempotent splitting of the motivic sphere
separates the \(\eta\)-trivial Milnor component from the Witt component
\cite{MorelA1}.  The relation \(\eta\nu_{(3)}=0\) in the stable first stem and
the \(\eta\)-periodicity of the Witt component show that the homotopy module
generated by \(\nu_{(3)}\) lies in the former
component.  Thus the
operation factors through
\(\mathbf K_*^{MW}/(\eta,3)=\mathbf K_*^M/3\), and its universal stable
representative is a map

\[
H\mathbb Z_{(3)}
\longrightarrow
\Sigma^{4,2}H\mathbb Z/3.
\]

Lemma~\ref{lem:degree-42-integral-mod3-operations} shows that this map is a
unique scalar multiple of \(P^1\operatorname{red}_3\).  It remains only to
determine whether the scalar is zero.  We make the realization
argument at the level of the two-stage Postnikov object, rather than merely
at the level of an element.  Work first over \(\mathbb Q\), base change to
\(\mathbb C\), and let \(T^{\mathrm{mot}}\) be the three-primary two-stage
object defined by the fiber sequence
\[
 \Sigma^{3,2}H\mathbb Z/3\longrightarrow T^{\mathrm{mot}}
 \longrightarrow H\mathbb Z_{(3)}
 \xrightarrow{\ \Delta_{\mathbb C}\ }
 \Sigma^{4,2}H\mathbb Z/3.
\]
By construction there is a map
\(\mathds1_{(3)}\to T^{\mathrm{mot}}\) inducing the identity on the bottom
homotopy module and the quotient onto the summand generated by
\(\nu_{(3)}\).

Complex Betti realization is exact and symmetric monoidal.  It sends motivic
Eilenberg--MacLane spectra to the corresponding classical spectra and sends
\(\Sigma^{4,2}\) to \(\Sigma^4\).  Therefore
\(\operatorname{Re}_{\mathbb C}(T^{\mathrm{mot}})\) is a classical
three-local spectrum with homotopy \(\mathbb Z_{(3)}\) in degree zero,
\(\mathbb Z/3\) in degree three, and zero in the intervening degrees.  The
realized sphere map is an isomorphism in degree zero.  It is also an
isomorphism in degree three: complex realization
\[
 \pi_{3,2}(\mathds1_{\mathbb C})
 \longrightarrow \pi_3^s
\]
is an isomorphism.  Indeed,
\cite[Theorem~1.1 and Remark~5.8]{RSO19} give
\(\pi_{3,2}(\mathds1_{\mathbb C})\cong\mathbb Z/24\), generated by the
motivic Hopf element \(\nu\), and naturality of the Hopf construction sends
\(\nu\) to the classical Hopf element \(\nu_{\mathrm{top}}\), a generator of
\(\pi_3^s\cong\mathbb Z/24\).  After localization at \(3\), the displayed map
therefore carries \(\nu_{(3)}\) to the normalized generator
\(\nu_{\mathrm{top}}^{(3)}\)
of \((\pi_3^s)_{(3)}\cong\mathbb Z/3\).  Since
\((\pi_1^s)_{(3)}=(\pi_2^s)_{(3)}=0\), the realized map exhibits
\(\operatorname{Re}_{\mathbb C}(T^{\mathrm{mot}})\) as
\(\tau_{\leq3}\mathds1_{(3)}\).

It follows that the realization of \(\Delta_{\mathbb C}\) is the first
three-local \(k\)-invariant of the classical sphere spectrum.  This is the
nonzero operation \(P^1\operatorname{red}_3\), equivalently the first
three-primary differential in the stable cohomotopy Atiyah--Hirzebruch
spectral sequence \cite{AdamsStable}.  Hence the motivic coefficient is
nonzero.  By the three-primary sign convention fixed before the theorem,
this coefficient is \(1\), which proves
\eqref{eq:three-primary-k-invariant-P1}.  This normalization is fixed over
\(\mathbb Q\).  Lemma~\ref{lem:three-primary-Postnikov-base-change}
shows that the Postnikov invariant itself, and not only the abstract
generator of the operation group, is preserved by extension of fields.
Hence the conclusion
holds over every characteristic-zero field without choosing an embedding
of \(k\) into \(\mathbb C\).

Theorem~\ref{thm:parameterized-two-stage-comparison} identifies the
three-primary component of the geometric change with the Thom conjugate of
the stable operation just computed.  Proposition~\ref{prop:P1-Thom-Wu}
identifies that conjugate with \(\mathcal P_E^1\).  Applying it to the
difference \(a(\lambda,\lambda')\) gives
\eqref{eq:three-primary-secondary-difference}.  The quotient in
\eqref{eq:canonical-three-primary-secondary-class} is precisely the orbit
quotient for that action, which proves the final assertions.
\end{proof}

\begin{corollary}\label{cor:complete-secondary-orbit-criterion}
Let \(k\) be a field of characteristic \(0\) with
\(\operatorname{cd}_2(k)\leq1\), let \(d\geq7\), let \(X\) be a smooth
affine pure \(d\)-dimensional \(k\)-scheme, and let \(E\) be a vector
bundle of rank \(d-2\) on \(X\).  Put \(n=d-2\) and \(L=\det E\).
Assume that \(e(E)=0\).  Then \(E\) splits
off a trivial line bundle if and only if the following successive classes
vanish:
\[
\overline o_{2,3}(E),\qquad
\mathfrak h_2(E),\qquad
\mathfrak n_2(E).
\]
Here the first class is
\eqref{eq:canonical-three-primary-secondary-class}, the second is
\eqref{eq:canonical-hermitian-secondary-class}, and the third is defined by
\eqref{eq:canonical-residual-secondary-class} once
\(\mathfrak h_2(E)=0\).  Thus this criterion does not require the whole
secondary coefficient group to vanish.  It is the exact orbit criterion
for the three-primary
Thom-corrected reduced-power class, the two-primary hermitian class, and the
residual Milnor class modulo \(8\).
\end{corollary}

\begin{proof}
Theorem~\ref{thm:Q-three-primary-Postnikov-operation} says that the first
class vanishes exactly when a change of Euler nullhomotopy can kill the
three-primary secondary obstruction.  Theorem~\ref{thm:Q-two-primary-action-normal-form}
gives the corresponding
two-stage assertion at the prime \(2\).

These primewise changes can be made simultaneously.  Indeed, the
three-primary target is annihilated by \(3\).  By
\eqref{eq:Q-two-primary-Steenrod-exact-sequence}, the relevant
two-primary cohomology group is an extension of a subgroup of
\(\mathcal C_{n,2}^{X,L}\), which is annihilated by \(2\), by a quotient
of \(H^{n+1}_{\mathrm{Nis}}(X,\mathbf K^M_{n+2}/8)\), which is annihilated
by \(8\).  It is therefore annihilated by \(16\).  If
\(a_3\) kills the three-primary class, then \(16a_3\) has the same effect
at \(3\) and has no effect at \(2\).  If \(a_2\) kills the two-primary
class, then \(33a_2\) has the same effect at \(2\) and has no effect at
\(3\).  Adding these two changes kills the full secondary obstruction.

The next obstruction lies in
\(H^d_{\mathrm{Nis}}(X,\mathcal P_{d-2}(L))\), which vanishes by
Theorem~\ref{thm:P-top-cd2-vanishing}; all later obstruction groups vanish
by dimension.  Hence the classifying map lifts to \(BGL_{d-3}\), which is
equivalent to splitting a trivial line bundle from \(E\).  Conversely, such a
splitting supplies a lift through the whole Moore--Postnikov tower and
forces all three orbit classes to vanish.
\end{proof}

\begin{corollary}\label{cor:action-surjectivity-corank-two}
Use the notation of
Corollary~\ref{cor:complete-secondary-orbit-criterion}.  Keep its
assumptions on \(k,d,X\), and \(E\), but do not assume that \(e(E)=0\).
Suppose that, if \(e(E)=0\), each of the three homomorphisms
\[
\begin{aligned}
\mathcal P_E^1\operatorname{red}_3^L&:\mathcal A_n^{X,L}
 \longrightarrow
 H^{n+1}_{\mathrm{Nis}}(X,\mathbf K^M_{n+2}/3),\\
\Delta_{E,2}^{\mathrm h}&:\mathcal A_n^{X,L}
 \longrightarrow\mathcal C_{n,2}^{X,L},\\
\Delta_{E,2}^{\nu}&:\ker(\Delta_{E,2}^{\mathrm h})
 \longrightarrow\mathcal R_{n,2}^{X,L}
\end{aligned}
\]
are surjective.  Then \(E\) splits off a trivial line bundle if and only if
\(e(E)=0\).
\end{corollary}

\begin{proof}
If the Euler class vanishes, the three successive orbit groups containing
\(\overline o_{2,3}(E)\), \(\mathfrak h_2(E)\), and
\(\mathfrak n_2(E)\) vanish by the three surjectivity assumptions.
Corollary~\ref{cor:complete-secondary-orbit-criterion} therefore gives the
splitting.  Necessity follows from naturality of the Euler class.
\end{proof}

\begin{corollary}\label{cor:Q-secondary-three-primary-killable}
Let \(k,n,X,E\), and \(L=\det E\) be as in the geometric part of
Theorem~\ref{thm:Q-three-primary-Postnikov-operation}; in particular,
assume that the Euler obstruction vanishes.  Suppose that

\[
\mathcal P_E^1\operatorname{red}_3^L:
H^{n-1}_{\mathrm{Nis}}(X,\mathbf K_n^{MW}(L))
\longrightarrow
H^{n+1}_{\mathrm{Nis}}(X,\mathbf K^M_{n+2}/3)
\]

is surjective.  Once the Euler obstruction vanishes, the Euler
nullhomotopy can be chosen so that the three-primary component of the
secondary obstruction is zero.  Thus only its two-primary component remains.
\end{corollary}

\begin{proof}
This is immediate from
\eqref{eq:three-primary-secondary-difference}: change the initial
nullhomotopy by a preimage of the negative of its three-primary secondary
obstruction.
\end{proof}

\begin{corollary}\label{cor:Q-secondary-two-primary}
In the situation of Proposition~\ref{prop:Q-Gersten-three-primary}, if
\[
H^{n+1}_{\mathrm{Nis}}
 \left(X,\mathbf K^M_{n+2}/3\right)=0,
\]
then \(H^{n+1}_{\mathrm{Nis}}(X,\mathcal Q_n(L))\) is \(2\)-primary.
\end{corollary}

\begin{proof}
The Rost--Schmid terms in the three relevant degrees have no primary
components other than \(2\) and \(3\).  Proposition~\ref{prop:Q-Gersten-three-primary}
kills the \(3\)-primary component under the displayed hypothesis.
\end{proof}

\begin{remark}\label{rem:Q-secondary-true-boundary}
Proposition~\ref{prop:Q-contracted-Gersten-tail} makes the remaining secondary
problem explicit, but it does not force its vanishing.  If
$\dim X=n+2$, then a point of codimension $n+1$ has residue field of
transcendence degree $1$ over $k$, rather than a finite extension of $k$.
Consequently, even when $\operatorname{cd}_2(k)\leq1$, the middle term in
\eqref{eq:Q-Gersten-tail} need not vanish.  The \(3\)-primary coefficient
portion is explicit: Proposition~\ref{prop:Q-Gersten-three-primary}
identifies it with ordinary Milnor \(K\)-theory cohomology, and
Theorem~\ref{thm:Q-three-primary-Postnikov-operation} identifies the
corresponding stable \(k\)-invariant with
\(P^1\operatorname{red}_3\), and
Theorem~\ref{thm:parameterized-two-stage-comparison} identifies its Thom
conjugate with the geometric change homomorphism.  The resulting cokernel
class need not vanish and is not controlled by a bound on
\(2\)-cohomological dimension.

The genuinely hermitian part of the differential is \(2\)-primary.
Writing \(\{2\}\) for the \(2\)-primary subgroup, in the last three degrees
its coefficient extensions have the form
\[
\begin{aligned}
0&\longrightarrow\mathbf K^M_2/8
\longrightarrow(\mathcal Q_n)_{-n}\{2\}
\longrightarrow\mathbf{GW}^0_1\longrightarrow0,\\
0&\longrightarrow\mathbf K^M_1/8
\longrightarrow(\mathcal Q_n)_{-(n+1)}\{2\}
\longrightarrow\underline{\mathbb Z/2}\longrightarrow0,
\end{aligned}
\]
followed by the terminal coefficient \(\mathbb Z/8\).  On every residue
field, the quotient in the first line is explicitly
\(\mathbf K^M_1/2\oplus\mathbb Z/2\) by
\eqref{eq:GW10-field-explicit}.  The corresponding first-stem extension is
not split in general
\cite[Theorem~1.1 and the discussion following it]{RSO19}.  Nevertheless,
Proposition~\ref{prop:Q-explicit-first-stem-residues} computes the residue
without choosing a splitting.  The relation
\(12\nu=\eta^2\eta_{\mathrm{top}}\) turns the non-additivity of the
spinor-norm lifts and their cross-residue into multiplication by \(4\) on
the Milnor kernels modulo \(8\); see
Corollary~\ref{cor:Q-two-primary-normal-form}.  Hence the coefficient-level
differential in \eqref{eq:Q-Gersten-tail} is now completely determined.

This is different from the top-degree Gersten argument in
\cite[Theorem~7.1.1]{ABH}: there the final residue fields are algebraically
closed and the relevant last contraction vanishes.  Here the unresolved
object is the cohomology of the explicit \(2\)-primary complex, and then the
two-primary secondary Moore--Postnikov class represented in that cohomology.
Neither the displayed residue formula nor the hypothesis
\(\operatorname{cd}_2(k)\leq1\) forces this class to vanish on a general
smooth affine scheme.
\end{remark}

\begin{proposition}\label{prop:Q-vector-bundle-vanishing}
Let \(k\) be a perfect field, let \(n\geq2\), and let \(S\) be a smooth
\(k\)-scheme of dimension at most \(n\).  If
\[
p:X=\operatorname{Tot}(V)\longrightarrow S
\]
is the total space of a vector bundle and \(L_0\) is a line bundle on
\(S\), then
\[
H^{n+1}_{\mathrm{Nis}}\!
 \left(X,\mathcal Q_n(p^*L_0)\right)=0.
\]
\end{proposition}

\begin{proof}
The sheaf \(\mathcal Q_n\) is an abelian higher
\(\mathbb A^1\)-homotopy sheaf, hence is strictly
\(\mathbb A^1\)-invariant.  The same is true after any line bundle twist.
Let \(s:S\to X\) be the zero section.  Fiberwise scalar
multiplication defines an \(\mathbb A^1\)-homotopy from
\(\operatorname{id}_X\) to \(sp\), and it preserves the pullback
\(p^*L_0\).  Strict homotopy invariance therefore gives mutually inverse
isomorphisms
\[
p^*:H^i_{\mathrm{Nis}}\!
 \left(S,\mathcal Q_n(L_0)\right)
\xrightarrow{\ \cong\ }
H^i_{\mathrm{Nis}}\!
 \left(X,\mathcal Q_n(p^*L_0)\right)
\]
for every \(i\).  The left-hand group is zero for \(i=n+1\), since the
Nisnevich cohomological dimension of \(S\) is at most
\(\dim S\leq n\).
\end{proof}

Proposition~\ref{prop:Q-explicit-first-stem-residues} determines the last
coefficient-level input in the general corank-two splitting problem.
Proposition~\ref{prop:Q-vector-bundle-vanishing} forces the resulting
secondary cohomology group to vanish for vector-bundle total spaces over
sufficiently low-dimensional bases.  The next statement records the
geometric consequence whenever that cohomology group vanishes.

Our twist notation follows the determinant-action convention in
\cite[Section~6.1]{AF15b}: $\mathcal F(\det E)$ denotes the local-system twist
induced by the determinant action.  Under the alternative Chow--Witt
convention using dual orientation line bundles, the same coefficient is written
with $(\det E)^{-1}$; this is a change of convention, not a different
obstruction group.

\begin{theorem}\label{thm:corank-two-splitting}
Let \(k\), \(d\), and \(X=\operatorname{Spec}R\) satisfy the hypotheses of
Theorem~\ref{thm:P-top-cd2-vanishing}, with \(X\) smooth affine.  Let \(E\) be
a vector bundle of rank \(d-2\) on \(X\), and assume
\begin{equation}\label{eq:secondary-Q-vanishing}
H^{d-1}_{\mathrm{Nis}}\!\left(
X,\mathcal Q_{d-2}(\det E)\right)=0.
\end{equation}
Then \(E\) splits off a trivial line bundle if and only if its Euler class \(e(E)\)
vanishes.
\end{theorem}

\begin{proof}
Apply the Moore--Postnikov obstruction theory of \cite[Section~6.1]{AF15b}
to
\[
\A^{d-2}\setminus0\longrightarrow BGL_{d-3}\longrightarrow BGL_{d-2}.
\]
Because the fiber is \(\A^1\)-\((d-4)\)-connected and \(X\) has Nisnevich
cohomological dimension at most \(d\), the only potentially nonzero
obstructions to lifting the classifying map of \(E\) lie successively in
\[
\begin{aligned}
&H^{d-2}_{\mathrm{Nis}}
  (X,\mathbf K^{MW}_{d-2}(\det E)),\\
&H^{d-1}_{\mathrm{Nis}}
  (X,\mathcal Q_{d-2}(\det E)),\\
&H^d_{\mathrm{Nis}}
  (X,\mathcal P_{d-2}(\det E)).
\end{aligned}
\]
The first obstruction is the Euler class.  The second group is zero by
\eqref{eq:secondary-Q-vanishing}, and the third is zero by
Theorem~\ref{thm:P-top-cd2-vanishing}.  All subsequent obstruction groups
vanish by dimension.  Hence \(e(E)=0\) is sufficient for the lift to
\(BGL_{d-3}\), and such a lift is equivalent to splitting a trivial line bundle.
Necessity is immediate.
\end{proof}

\begin{theorem}\label{thm:corank-two-splitting-reduced-power}
Let \(k\), \(d\), \(X\), and \(E\) satisfy the hypotheses of
Theorem~\ref{thm:corank-two-splitting}, except that
\eqref{eq:secondary-Q-vanishing} is replaced by the following two
conditions:

\begin{equation}\label{eq:secondary-Q-two-primary-vanishing}
H^{d-1}_{\mathrm{Nis}}\!\left(
X,\mathcal Q_{d-2}(\det E)\right)\{2\}=0,
\end{equation}

and the homomorphism

\begin{equation}\label{eq:corank-two-P1-surjectivity}
\begin{aligned}
\mathcal P_E^1\operatorname{red}_3^{\det E}:\quad
H^{d-3}_{\mathrm{Nis}}\!\left(
X,\mathbf K_{d-2}^{MW}(\det E)\right)
\longrightarrow
H^{d-1}_{\mathrm{Nis}}\!\left(
X,\mathbf K_d^M/3\right)
\end{aligned}
\end{equation}

is surjective.  Then \(E\) splits
off a trivial line bundle if and only if its Euler class vanishes.
\end{theorem}

\begin{proof}
Put \(n=d-2\).  If the Euler class vanishes, choose a nullhomotopy of the
primary obstruction.  By
Theorem~\ref{thm:Q-three-primary-Postnikov-operation} and the surjectivity
in \eqref{eq:corank-two-P1-surjectivity}, that nullhomotopy may be changed
so that the three-primary component of the secondary obstruction vanishes.
Its two-primary component vanishes by
\eqref{eq:secondary-Q-two-primary-vanishing}.  The next obstruction lies in

\[
H^d_{\mathrm{Nis}}
 \left(X,\mathcal P_{d-2}(\det E)\right),
\]

which is zero by Theorem~\ref{thm:P-top-cd2-vanishing}; all later
obstruction groups vanish by dimension.  Thus the classifying map lifts to
\(BGL_{d-3}\), which is equivalent to splitting a trivial line bundle from \(E\).
The converse follows from naturality of the Euler class.
\end{proof}

\begin{corollary}\label{cor:corank-two-Steenrod-reduced-power}
Let \(k\) be a field of characteristic \(0\) with
\(\operatorname{cd}_2(k)\leq1\), let \(d\geq7\), let \(X\) be a smooth
affine pure \(d\)-dimensional \(k\)-scheme, and let \(E\) be a vector bundle
of rank \(d-2\) on \(X\).  Put \(L=\det E\).  Assume that
\begin{enumerate}
\item the twisted Steenrod square
\[
Sq_L^2:\operatorname{Ch}^{d-2}(X)
\longrightarrow\operatorname{Ch}^{d-1}(X)
\]
is surjective;
\item
\[
H^{d-1}_{\mathrm{Nis}}(X,\mathbf K_d^M/8)=0;
\]
\item the Thom-corrected reduced-power homomorphism
\[
\mathcal P_E^1\operatorname{red}_3^L:
H^{d-3}_{\mathrm{Nis}}
 (X,\mathbf K_{d-2}^{MW}(L))
\longrightarrow
H^{d-1}_{\mathrm{Nis}}(X,\mathbf K_d^M/3)
\]
is surjective.
\end{enumerate}
Then \(E\) splits off a trivial line bundle if and only if \(e(E)=0\).
\end{corollary}

\begin{proof}
Apply Proposition~\ref{prop:Q-two-primary-Steenrod-filtration} with
\(n=d-2\).  Conditions (1) and (2) give
\[
H^{d-1}_{\mathrm{Nis}}
 (X,\mathcal Q_{d-2}(L))\{2\}=0.
\]
Condition (3) is precisely
\eqref{eq:corank-two-P1-surjectivity}.  The result now follows from
Theorem~\ref{thm:corank-two-splitting-reduced-power}.
\end{proof}

This gives a generator theorem in the sense of \cite{AOSS26}.
Proposition~\ref{prop:Q-two-primary-Steenrod-filtration} and the preceding
corollary replace its abstract two-primary vanishing hypothesis by standard
Chow and Milnor \(K\)-cohomology conditions.

\begin{theorem}\label{thm:corank-two-vector-bundle-total-space}
Let \(k\) be a field of characteristic \(0\).  Let \(S\) be a smooth affine
pure \(s\)-dimensional \(k\)-scheme, let \(V\) be a vector bundle of rank
\(t\geq2\) on \(S\), and put
\[
X=\operatorname{Tot}(V),\qquad d=s+t.
\]
Assume \(d\geq4\).  Then every rank-\((d-2)\) vector bundle \(E\) on \(X\)
splits off a trivial line bundle if and only if \(e(E)=0\).  If \(t\geq3\), every
such \(E\) splits off a trivial line bundle.
\end{theorem}

\begin{proof}
The scheme \(X\) is smooth affine and pure of dimension \(d\).  Homotopy
invariance of the Picard group gives
\[
p^*:\operatorname{Pic}(S)\xrightarrow{\ \cong\ }\operatorname{Pic}(X),
\]
so \(\det E=p^*L_0\) for a line bundle \(L_0\) on \(S\).

Apply the Moore--Postnikov obstruction theory of \cite[Section~6.1]{AF15b}
to the fiber sequence
\[
\A^{d-2}\setminus0\longrightarrow BGL_{d-3}\longrightarrow BGL_{d-2}.
\]
Its primary obstruction is
\[
e(E)\in
H^{d-2}_{\mathrm{Nis}}\!
 \left(X,\mathbf K^{MW}_{d-2}(\det E)\right);
\]
every subsequent obstruction belongs to a twisted cohomology group in
degree at least \(d-1\) with strictly \(\mathbb A^1\)-invariant
coefficients.  Each such coefficient is a homotopy sheaf of the fiber, with
the local-system twist induced by \(\det E=p^*L_0\).  Fiberwise scalar
multiplication gives an \(\mathbb A^1\)-homotopy from
\(\operatorname{id}_X\) to the zero-section retraction and preserves this
pullback twist.  Strict homotopy invariance therefore identifies, in every
degree, each obstruction group on \(X\) with the corresponding group on
\(S\); this is the same argument as
Proposition~\ref{prop:Q-vector-bundle-vanishing}, now applied to every
coefficient sheaf in the tower.  Since
\[
s=d-t\leq d-2,
\]
all groups in degree at least \(d-1\) vanish.  Thus the vanishing of
\(e(E)\) is sufficient, and it is necessary by naturality of the Euler
class.

If \(t\geq3\), then \(s\leq d-3\), so the same argument also makes the
primary obstruction group in degree \(d-2\) vanish.  Hence every \(E\)
splits in that case.
\end{proof}

\begin{corollary}\label{cor:r-plus-d-minus-three-generation}
Let \(X=\operatorname{Spec}R\) be as in
Theorem~\ref{thm:corank-two-splitting}, and let \(M\) be a projective
\(R\)-module of rank \(r\geq4\) that is generated by \(r+d-2\) elements.
Assume
\begin{equation}\label{eq:generation-secondary-vanishing}
H^{d-1}_{\mathrm{Nis}}\!\left(
X,\mathcal Q_{d-2}((\det M)^{-1})\right)=0.
\end{equation}
Then the following conditions are equivalent.
\begin{enumerate}
\item The module \(M\) is generated by \(r+d-3\) elements.
\item There is a surjection
\(R^{r+d-2}\twoheadrightarrow M\) whose rank-\(d-2\) kernel \(E\) has
\(e(E)=0\).
\end{enumerate}
\end{corollary}

\begin{proof}
Suppose first that \(M\) is generated by \(r+d-3\) elements.  Adjoining a
zero generator to such a surjection produces a presentation with kernel
\(E'\oplus R\); its Euler class is zero.

Conversely, choose the presentation in the second condition.  Its kernel
\(E\) has determinant \((\det M)^{-1}\), so
Theorem~\ref{thm:corank-two-splitting} and
\eqref{eq:generation-secondary-vanishing} give \(E\cong E'\oplus R\).  The
copy of \(R\) is a direct summand of \(R^{r+d-2}\).  Its quotient \(P\) is
stably free of rank \(r+d-3>d\), and Bass cancellation makes \(P\) free; this
is the same cancellation step used in the proof of
\cite[Proposition~21]{AOSS26}.  Quotienting the original presentation by that
copy of \(R\) therefore yields a surjection
\(R^{r+d-3}\twoheadrightarrow M\).
\end{proof}

\begin{corollary}\label{cor:generation-Steenrod-reduced-power}
Let \(k\) be a field of characteristic \(0\) with
\(\operatorname{cd}_2(k)\leq1\), let \(X=\operatorname{Spec}R\) be smooth
affine and pure of dimension \(d\geq7\), and let \(M\) be a projective
\(R\)-module of rank \(r\geq4\) generated by \(r+d-2\) elements.  Put
\(L=(\det M)^{-1}\), and assume that
\[
Sq_L^2:\operatorname{Ch}^{d-2}(X)
\longrightarrow\operatorname{Ch}^{d-1}(X)
\]
is surjective and that
\[
H^{d-1}_{\mathrm{Nis}}(X,\mathbf K_d^M/8)=0.
\]
Assume also that, for every rank-\((d-2)\) presentation kernel \(E\) with
zero Euler class, the homomorphism
\[
\mathcal P_E^1\operatorname{red}_3^L:
H^{d-3}_{\mathrm{Nis}}
 (X,\mathbf K_{d-2}^{MW}(L))
\longrightarrow
H^{d-1}_{\mathrm{Nis}}(X,\mathbf K_d^M/3)
\]
is surjective.
Then the following conditions are equivalent.
\begin{enumerate}
\item The module \(M\) is generated by \(r+d-3\) elements.
\item There is a surjection
\(R^{r+d-2}\twoheadrightarrow M\) whose rank-\((d-2)\) kernel has
vanishing Euler class.
\end{enumerate}
\end{corollary}

\begin{proof}
The kernel of a presentation in (2) has determinant \(L\), so
Corollary~\ref{cor:corank-two-Steenrod-reduced-power} splits it.  The proof
of Corollary~\ref{cor:r-plus-d-minus-three-generation}, including the same
Bass cancellation step, then applies verbatim.  The reverse implication is
obtained by adjoining a zero generator.
\end{proof}

\begin{corollary}\label{cor:action-surjectivity-generation}
Let \(k\) be a field of characteristic \(0\) with
\(\operatorname{cd}_2(k)\leq1\), let \(X=\operatorname{Spec}R\) be smooth
affine and pure of dimension \(d\geq7\), and let \(M\) be a projective
\(R\)-module of rank \(r\geq4\) generated by \(r+d-2\) elements.  Put
\(n=d-2\) and \(L=(\det M)^{-1}\).  Assume that for every presentation
whose rank-\((d-2)\) kernel \(E\) has zero Euler class, all three maps in
Corollary~\ref{cor:action-surjectivity-corank-two}, formed with \(E\),
are surjective.  Then the following conditions are equivalent.
\begin{enumerate}
\item The module \(M\) is generated by \(r+d-3\) elements.
\item There is a surjection \(R^{r+d-2}\twoheadrightarrow M\) whose
rank-\((d-2)\) kernel has vanishing Euler class.
\end{enumerate}
\end{corollary}

\begin{proof}
For a presentation in (2), the kernel \(E\) has determinant \(L\).
Corollary~\ref{cor:action-surjectivity-corank-two} splits it.  The Bass
cancellation argument in
Corollary~\ref{cor:r-plus-d-minus-three-generation} then produces
\(r+d-3\) generators of \(M\).  Conversely, adjoining a zero generator to
a presentation with \(r+d-3\) generators gives a kernel with a trivial
summand and hence with zero Euler class.
\end{proof}

\begin{corollary}\label{cor:generation-vector-bundle-total-space}
Retain the hypotheses and notation of
Theorem~\ref{thm:corank-two-vector-bundle-total-space}, write
\(X=\operatorname{Spec}R\), and let \(M\) be a projective \(R\)-module of
rank \(\rho\geq4\) generated by \(\rho+d-2\) elements.  Then the following
conditions are equivalent.
\begin{enumerate}
\item The module \(M\) is generated by \(\rho+d-3\) elements.
\item There is a surjection
\(R^{\rho+d-2}\twoheadrightarrow M\) whose rank-\((d-2)\) kernel has
vanishing Euler class.
\end{enumerate}
If \(t\geq3\), these conditions always hold.
\end{corollary}

\begin{proof}
The proof of Corollary~\ref{cor:r-plus-d-minus-three-generation} uses its
secondary-cohomology hypothesis only to split the rank-\((d-2)\) kernel
when its Euler class vanishes.  That implication is supplied here by
Theorem~\ref{thm:corank-two-vector-bundle-total-space}.  If \(t\geq3\), the
same theorem says that the kernel of any presentation
\(R^{\rho+d-2}\twoheadrightarrow M\) splits.  Quotienting by its trivial
summand and applying the same Bass cancellation argument gives a generating
set of cardinality \(\rho+d-3\).
\end{proof}

Theorems~\ref{thm:P-top-cd2-vanishing}
and~\ref{thm:corank-two-splitting} are uniform in \(d\geq7\).  Together they
give a family of corank-two splitting and efficient-generation statements
conditional on the \(Q\)-type cohomology vanishing in
\eqref{eq:secondary-Q-vanishing}.  Theorem~\ref{thm:corank-two-vector-bundle-total-space}
and Corollary~\ref{cor:generation-vector-bundle-total-space} remove that condition for
vector-bundle total spaces of relative rank at least two, without the
cohomological-dimension hypothesis or the restriction \(d\geq7\).  The
coefficient sheaves and the full residue--corestriction differential in the
last two Rost--Schmid degrees are explicit by
Propositions~\ref{prop:Q-contracted-Gersten-tail}
and~\ref{prop:Q-explicit-first-stem-residues}.  On the \(3\)-primary side,
Theorem~\ref{thm:Q-three-primary-Postnikov-operation} computes the stable
\(k\)-invariant as \(P^1\operatorname{red}_3\), while
Proposition~\ref{prop:P1-Thom-Wu} computes its conjugate for \(E\) as
\(\mathcal P_E^1\operatorname{red}_3\).
Theorem~\ref{thm:parameterized-two-stage-comparison} identifies this
conjugate with the geometric change and packages the obstruction as the
cokernel class \eqref{eq:canonical-three-primary-secondary-class}.
At the prime \(2\), Theorem~\ref{thm:Q-two-primary-action-normal-form}
separates the two coefficient layers of the change homomorphism.  Thus
Corollary~\ref{cor:complete-secondary-orbit-criterion} gives the resulting
exact splitting criterion without assuming that the secondary coefficient
group vanishes.  The remaining limitation is an explicit description of
the stable two-primary \(k\)-invariant on the hermitian and residual Milnor
quotients.  The coefficient-level cross-term in the residual target is the
order-two correction of
Corollary~\ref{cor:Q-two-primary-normal-form}.

\bibliographystyle{plain}
\bibliography{refs}

@article{Bott57,
  author  = {Bott, Raoul},
  title   = {The stable homotopy of classical groups},
  journal = {Proceedings of the National Academy of Sciences of the United States of America},
  volume  = {43},
  number  = {10},
  year    = {1957},
  pages   = {933--935},
  doi     = {10.1073/pnas.43.10.933},
}

@article{Kervaire60,
  author  = {Kervaire, Michel A.},
  title   = {Some nonstable homotopy groups of Lie groups},
  journal = {Illinois Journal of Mathematics},
  volume  = {4},
  number  = {2},
  year    = {1960},
  pages   = {161--169},
  doi     = {10.1215/ijm/1255455861},
}

@article{MV99,
  author  = {Morel, Fabien and Voevodsky, Vladimir},
  title   = {$\mathbb{A}^1$-homotopy theory of schemes},
  journal = {Publications Math{\'e}matiques de l'IH{\'E}S},
  volume  = {90},
  year    = {1999},
  pages   = {45--143},
  doi     = {10.1007/BF02698831},
}

@book{MorelA1,
  author    = {Morel, Fabien},
  title     = {$\mathbb{A}^1$-Algebraic Topology over a Field},
  publisher = {Springer},
  year      = {2012},
  doi       = {10.1007/978-3-642-29514-0},
}

@article{AF14,
  author  = {Asok, Aravind and Fasel, Jean},
  title   = {Algebraic vector bundles on spheres},
  journal = {Journal of Topology},
  volume  = {7},
  number  = {3},
  year    = {2014},
  pages   = {894--926},
  doi     = {10.1112/jtopol/jtt046},
}

@article{CalmesHornbostel11,
  author  = {Calm\`es, Baptiste and Hornbostel, Jens},
  title   = {Push-forwards for {Witt} groups of schemes},
  journal = {Commentarii Mathematici Helvetici},
  volume  = {86},
  number  = {2},
  year    = {2011},
  pages   = {437--468},
  doi     = {10.4171/CMH/230},
}

@article{AABD,
  author  = {Asok, Aravind and Doran, Brent},
  title   = {$\mathbb{A}^1$-homotopy groups, excision, and solvable quotients},
  journal = {Advances in Mathematics},
  volume  = {221},
  year    = {2009},
  pages   = {1144--1190},
  note    = {Available as arXiv:0902.1564},
}

@article{faseldegree,
  author  = {Fasel, Jean},
  title   = {A degree map on unimodular rows},
  journal = {Journal of the Ramanujan Mathematical Society},
  volume  = {27},
  number  = {1},
  year    = {2012},
  pages   = {21--40},
  note    = {Available as arXiv:1103.4780},
}

@article{AffineRep2,
  author  = {Asok, Aravind and Hoyois, Marc and Wendt, Matthias},
  title   = {Affine representability results in $\mathbb{A}^1$-homotopy theory, II: Principal bundles and homogeneous spaces},
  journal = {Geometry \& Topology},
  volume  = {22},
  number  = {2},
  year    = {2018},
  pages   = {1181--1225},
  doi     = {10.2140/gt.2018.22.1181},
}

@article{OVV07,
  author  = {Orlov, Dmitri and Vishik, Alexander and Voevodsky, Vladimir},
  title   = {An exact sequence for $K_*^M/2$ with applications to quadratic forms},
  journal = {Annals of Mathematics},
  volume  = {165},
  number  = {1},
  year    = {2007},
  pages   = {1--13},
  doi     = {10.4007/annals.2007.165.1},
}

@article{Voevodsky03,
  author  = {Voevodsky, Vladimir},
  title   = {Reduced power operations in motivic cohomology},
  journal = {Publications Math{\'e}matiques de l'IH{\'E}S},
  volume  = {98},
  year    = {2003},
  pages   = {1--57},
  doi     = {10.1007/s10240-003-0009-z},
}

@article{Voevodsky10,
  author  = {Voevodsky, Vladimir},
  title   = {Motivic {Eilenberg--MacLane} spaces},
  journal = {Publications Math{\'e}matiques de l'IH{\'E}S},
  volume  = {112},
  year    = {2010},
  pages   = {1--99},
  doi     = {10.1007/s10240-010-0024-9},
}

@book{AdamsStable,
  author    = {Adams, J. Frank},
  title     = {Stable Homotopy and Generalised Homology},
  series    = {Chicago Lectures in Mathematics},
  publisher = {University of Chicago Press},
  address   = {Chicago},
  year      = {1974},
}

@article{Morel04,
  author  = {Morel, Fabien},
  title   = {Sur les puissances de l'id{\'e}al fondamental de l'anneau de Witt},
  journal = {Commentarii Mathematici Helvetici},
  volume  = {79},
  year    = {2004},
  pages   = {689--703},
  doi     = {10.1007/s00014-004-0815-z},
}

@article{Morel05,
  author  = {Morel, Fabien},
  title   = {Milnor's conjecture on quadratic forms and mod $2$ motivic complexes},
  journal = {Rendiconti del Seminario Matematico della Universit\`a di Padova},
  volume  = {114},
  year    = {2005},
  pages   = {63--101},
}

@article{Sch2010b,
  author  = {Schlichting, Marco},
  title   = {The {Mayer--Vietoris} principle for {Grothendieck--Witt} groups of schemes},
  journal = {Inventiones mathematicae},
  volume  = {179},
  number  = {2},
  year    = {2010},
  pages   = {349--433},
  doi     = {10.1007/s00222-009-0219-1},
}

@article{Sch17,
  author  = {Schlichting, Marco},
  title   = {Hermitian $K$-theory, derived equivalences and {Karoubi}'s fundamental theorem},
  journal = {Journal of Pure and Applied Algebra},
  volume  = {221},
  number  = {7},
  year    = {2017},
  pages   = {1729--1844},
  doi     = {10.1016/j.jpaa.2016.12.026},
}

@article{AWW17,
  author  = {Asok, Aravind and Wickelgren, Kirsten and Williams, Ben},
  title   = {The simplicial suspension sequence in $\mathbb{A}^1$-homotopy},
  journal = {Geometry \& Topology},
  volume  = {21},
  number  = {4},
  year    = {2017},
  pages   = {2093--2160},
  doi     = {10.2140/gt.2017.21.2093},
}

@misc{ABH,
  author        = {Asok, Aravind and Bachmann, Tom and Hopkins, Michael J.},
  title         = {On $\mathbb{P}^1$--stabilization in unstable motivic homotopy theory},
  year          = {2026},
  eprint        = {2306.04631},
  archivePrefix = {arXiv},
  primaryClass  = {math.AG},
  note          = {Version 3; final version before page proofs, accepted for publication in the Annals of Mathematics},
  url           = {https://arxiv.org/abs/2306.04631},
}

@article{AF17,
  author  = {Asok, Aravind and Fasel, Jean},
  title   = {An explicit {KO}-degree map and applications},
  journal = {Journal of Topology},
  volume  = {10},
  number  = {1},
  year    = {2017},
  pages   = {268--300},
  doi     = {10.1112/topo.12007},
}

@article{AF17Spheres,
  author  = {Asok, Aravind and Fasel, Jean},
  title   = {Algebraic vs. topological vector bundles on spheres},
  journal = {Journal of the Ramanujan Mathematical Society},
  volume  = {32},
  number  = {3},
  year    = {2017},
  pages   = {201--216},
  eprint  = {1402.4156},
  archivePrefix = {arXiv},
}

@article{AF15b,
  author  = {Asok, Aravind and Fasel, Jean},
  title   = {Splitting vector bundles outside the stable range and $\mathbb{A}^1$-homotopy sheaves of punctured affine spaces},
  journal = {Journal of the American Mathematical Society},
  volume  = {28},
  number  = {4},
  year    = {2015},
  pages   = {1031--1062},
}

@article{AF14b,
  author  = {Asok, Aravind and Fasel, Jean},
  title   = {A cohomological classification of vector bundles on smooth affine threefolds},
  journal = {Duke Mathematical Journal},
  volume  = {163},
  number  = {14},
  year    = {2014},
  pages   = {2561--2601},
}

@article{AFW20,
  author  = {Asok, Aravind and Fasel, Jean and Williams, Ben},
  title   = {Motivic spheres and the image of the {Suslin--Hurewicz} map},
  journal = {Inventiones mathematicae},
  volume  = {219},
  year    = {2020},
  pages   = {39--73},
  doi     = {10.1007/s00222-019-00907-z},
}

@article{Rondigs23,
  author  = {R{\"o}ndigs, Oliver},
  title   = {Endomorphisms of the projective plane and the image of the {Suslin--Hurewicz} map},
  journal = {Inventiones mathematicae},
  volume  = {232},
  year    = {2023},
  pages   = {1161--1194},
  doi     = {10.1007/s00222-023-01179-4},
}

@article{Gant_Williams_2025, 
        title={Free summands of stably free modules}, 
        volume={13}, DOI={10.1017/fms.2025.39}, journal={Forum of Mathematics, Sigma}, 
        author={Gant, Sebastian and Williams, Ben}, 
        year={2025}, 
        pages={e85}
        }

@misc{gant2026motivichomotopygroupsspheres,
      title={Motivic Homotopy Groups of Spheres and Free Summands of Stably Free Modules}, 
      author={Sebastian Gant and Ben Williams},
      year={2026},
      eprint={2510.10033},
      archivePrefix={arXiv},
      primaryClass={math.AT},
      note={Version 3, 6 March 2026},
      url={https://arxiv.org/abs/2510.10033}, 
}

@article{RSO21,
author = {Oliver R{\"o}ndigs and Markus Spitzweck and Paul Arne {\O}stv{\ae}r},
title = {{The second stable homotopy groups of motivic spheres}},
volume = {173},
journal = {Duke Mathematical Journal},
number = {6},
publisher = {Duke University Press},
pages = {1017 -- 1084},
year = {2024},
doi = {10.1215/00127094-2023-0023},
URL = {https://doi.org/10.1215/00127094-2023-0023}
}

@article{RSO19,
  author  = {R{\"o}ndigs, Oliver and Spitzweck, Markus and {\O}stv{\ae}r, Paul Arne},
  title   = {The first stable homotopy groups of motivic spheres},
  journal = {Annals of Mathematics},
  volume  = {189},
  number  = {1},
  year    = {2019},
  pages   = {1--74},
  doi     = {10.4007/annals.2019.189.1.1},
}

@incollection{RSO18Cellular,
  author    = {R{\"o}ndigs, Oliver and Spitzweck, Markus and {\O}stv{\ae}r, Paul Arne},
  title     = {Cellularity of hermitian \(K\)-theory and Witt-theory},
  booktitle = {\(K\)-Theory---Proceedings of the International Colloquium, Mumbai, 2016},
  series    = {Tata Institute of Fundamental Research Studies in Mathematics},
  volume    = {19},
  pages     = {35--40},
  publisher = {Hindustan Book Agency},
  address   = {New Delhi},
  year      = {2018},
  eprint    = {1603.05139},
  archivePrefix = {arXiv},
  primaryClass  = {math.AT},
}

@article{Feld20,
  author  = {Feld, Niels},
  title   = {Milnor--Witt cycle modules},
  journal = {Journal of Pure and Applied Algebra},
  volume  = {224},
  number  = {7},
  year    = {2020},
  pages   = {106298},
  doi     = {10.1016/j.jpaa.2019.106298},
}

@article{Feld21,
  author  = {Feld, Niels},
  title   = {Morel homotopy modules and Milnor--Witt cycle modules},
  journal = {Documenta Mathematica},
  volume  = {26},
  year    = {2021},
  pages   = {617--659},
  doi     = {10.4171/DM/824},
}

@article{BO22,
  author  = {Bachmann, Tom and {\O}stv{\ae}r, Paul Arne},
  title   = {Topological models for stable motivic invariants of regular number rings},
  journal = {Forum of Mathematics, Sigma},
  volume  = {10},
  year    = {2022},
  pages   = {e1},
  doi     = {10.1017/fms.2021.76},
}

@book{MilneCFT,
  author    = {Milne, James S.},
  title     = {Class Field Theory},
  publisher = {Available from the author's website},
  year      = {2020},
  note      = {Version 4.03},
  url       = {https://www.jmilne.org/math/CourseNotes/CFT.pdf},
}

@misc{AOSS26,
  author        = {Asok, Aravind and Opie, Morgan and Shin, Brian and Syed, Tariq},
  title         = {Efficient generation of projective modules: a motivic view},
  year          = {2026},
  eprint        = {2510.13687},
  archivePrefix = {arXiv},
  primaryClass  = {math.AG},
  note          = {Version 2, 1 March 2026},
  url           = {https://arxiv.org/abs/2510.13687},
}

@article{AF16Euler,
  author  = {Asok, Aravind and Fasel, Jean},
  title   = {Comparing {Euler} classes},
  journal = {The Quarterly Journal of Mathematics},
  volume  = {67},
  number  = {4},
  year    = {2016},
  pages   = {603--635},
  doi     = {10.1093/qmath/haw033},
}
\end{document}